\documentclass[12pt]{amsart}
\usepackage{amsfonts}
\usepackage{amsmath}
\usepackage{amssymb}
\usepackage{mathrsfs}
\usepackage{geometry}
\usepackage{graphicx}
\usepackage{subfigure}
\usepackage{hyperref}
\usepackage{microtype}
\usepackage{enumerate}
\usepackage{geometry}
\usepackage{mathrsfs}
\usepackage{tikz}
\usepackage{tikz-cd}
\usepackage{color}
\numberwithin{equation}{section}

\newcommand{\la}{\lambda}
\newcommand{\al}{\alpha}
\newcommand{\be}{\beta}
\newcommand{\ga}{\gamma}
\newcommand{\Ga}{\Gamma}

\newcommand{\R}{\mathbb{R}}

\newcommand{\T}{\mathbb{T}}

\newcommand{\ccc}{\cdot\cdot\cdot}

\newcommand{\n}[1]{\Vert #1\Vert}
\newcommand{\bn}[1]{\big \Vert #1 \big \Vert}
\newcommand{\bbn}[1]{\Big\Vert #1 \Big \Vert}
\newcommand{\lr}[1]{\left\{ #1\right\}}
\newcommand{\lrc}[1]{\left[ #1\right]}
\newcommand{\lrs}[1]{\left( #1\right)}

\newcommand{\lra}[1]{\langle #1\rangle}

\newcommand{\bbabs}[1]{\Big | #1 \Big|}
\newcommand{\wt}[1]{\widetilde{#1}}
\newcommand{\wq}{\infty}
\newcommand{\pa}{\partial}

\newcommand{\ol}{\overline}

\begin{document}

\newtheorem{theorem}{Theorem}[section]
\newtheorem{lemma}[theorem]{Lemma}

\theoremstyle{definition}
\newtheorem{definition}[theorem]{Definition}
\newtheorem{example}[theorem]{Example}
\newtheorem{remark}[theorem]{Remark}

\numberwithin{equation}{section}

\newtheorem{proposition}[theorem]{Proposition}
\newtheorem{corollary}[theorem]{Corollary}
\newtheorem{goal}[theorem]{Goal}
\newtheorem{algorithm}{Algorithm}

\renewcommand{\figurename}{Fig.}

\title[Compressible Euler From Quantum N-Body]{The derivation of the compressible Euler equation from quantum many-body dynamics}
\author[X. Chen]{Xuwen Chen}
\address{Department of Mathematics, University of Rochester, Rochester, NY 14627, USA}

\email{chenxuwen@math.umd.edu}

\author[S. Shen]{Shunlin Shen}
\address{School of Mathematical Sciences, Peking University, Beijing, 100871, China}
\email{1701110015@pku.edu.cn; slshen100871@gmail.com}

\author[J. Wu]{Jiahao Wu}
\address{School of Mathematical Sciences, Peking University, Beijing, 100871, China}
\email{2101110023@stu.pku.edu.cn}

\author[Z. Zhang]{Zhifei Zhang}
\address{School of Mathematical Sciences, Peking University, Beijing, 100871, China}

\email{zfzhang@math.pku.edu.cn}


\date{}

\dedicatory{}

\begin{abstract}
We study the three dimensional many-particle quantum dynamics in mean-field setting. We forge together the hierarchy method and the modulated energy method. We prove rigorously that the compressible Euler equation is the limit as the particle number tends to infinity and the Planck's constant tends to zero. We establish strong and quantitative microscopic to macroscopic convergence of mass and momentum densities up to the 1st blow up time of the limiting Euler equation. We justify that the macroscopic pressure emerges from the space-time averages of microscopic interactions, which are in fact, Strichartz-type bounds. We have hence found a physical meaning for Strichartz type bounds which were first raised
by Klainerman and Machedon in this context.
 \end{abstract}
\maketitle
\tableofcontents

\section{Introduction}
The analysis of the nonlinear fluid equations like the Euler equations and the Navier-Stokes equations, is an important (if not vital) part of many areas of pure and applied mathematics, science, and engineering. On the one hand, their validity has certainly been checked countless times against the experiments. On the other hand, the rigorous derivation of these macroscopic continuum equations from basic microscopic Newtonian / Maxwell / quantum particle models has largely remained open. It is certainly of fundamental interest in mathematics to establish such derivations and prove that macroscopic quantities like pressure and viscosity emerge from the averaging of microscopic quantities. In this paper, we prove the derivation of the compressible Euler equation from the quantum $N$-body dynamic in the mean-field setting. We choose to start from the quantum theory as it is, at the moment, the most accurate microscopic model and such a derivation would also establish (again) that there is no obvious gap between the basic models in quantum and classical scales.

In the setting of classical mechanics, a strategy of the derivation of fluid equations from particle systems is to 1st pass to a mesoscopic Boltzmann equation, then derive the desired fluid equation from the Boltzmann equation. (See, for example, the standard monographs \cite{CIP94,GSRT13,SR09} and references within.) However, such a route may not suit our purpose here. On the one hand, the validity of the classical Boltzmann equations is only justified up to a sufficiently small time and is not clear if it covers the 1st blow up time of the Euler equation. On the other hand, the derivation of the quantum Boltzmann equation is at a rudimentary stage. (See, for example, \cite{CH21on,CG15,ESY04} and the references within.) Not to mention the possibility that one might need to pass to another classical Boltzmann equation if one takes such a route. Moreover, we would like to understand the fine interplay between $\hbar$ and $N$, the two fundamental constants, which differ by $10^{57}$ in SI units. In fact, starting from 2019, the mass unit is defined via the Planck's constant. Thus, we choose to derive the compressible Euler equation directly from quantum many-body dynamics.

We consider Bosons in this paper as it is more directly related to the Newton-Maxwell particles due to the assumption that particles are indistinguishable. (Fermions are also interesting, see for example, the survey \cite{NY02}.) We consider the 3D linear $N$-body bosonic Schr\"{o}dinger equation:
\begin{align}\label{equ:N-body schroedinger equation}
i\hbar\pa_{t}\psi_{N,\hbar}=H_{N,\hbar}\psi_{N,\hbar}
\end{align}
with Hamiltonian $H_{N,\hbar}$ given by
\begin{align}\label{equ:hamiltonian}
H_{N,\hbar}=\sum_{j=1}^{N}-\frac{1}{2}\hbar^{2}\Delta_{x_{j}}+\frac{1}{N}\sum_{1\leq j<k\leq N}V_{N}(x_{j}-x_{k})
\end{align}
where
\begin{align}
V_{N}(x)=N^{3\be}V(N^{\be}x),
\end{align}
and the factor $1/N$ is to make sure the interactions grow like $N$ instead $N^{2}$, a mean-field like scaling.
The marginal densities $\ga_{N,\hbar}^{(k)}$ associated with $\psi_{N,\hbar}$ in kernel form are given by
\begin{align}
\ga_{N,\hbar}^{(k)}(t,\textbf{x}_{k},\textbf{x}_{k}')=\int \psi_{N,\hbar}(t,\textbf{x}_{k},\textbf{x}_{N-k})\ol{\psi_{N,\hbar}}(t,\textbf{x}_{k}',\textbf{x}_{N-k})\textbf{dx}_{N-k}
\end{align}
where $\textbf{x}_{k}=(x_{1},...,x_{k})\in \R^{3k}$ and $\textbf{x}_{N-k}=(x_{k+1},...,x_{N})\in \R^{3(N-k)}$.
Notably, one can derive cubic nonlinear Schr\"{o}dinger equation (NLS) as the $N\to \infty$ limit of $(\ref{equ:N-body schroedinger equation})$ with $\hbar$ fixed, then the well-known Madelung transform \cite{Mad27} relates Schr\"{o}dinger type equation and the macroscopic Euler equations in a formal limit process as $\hbar$ tends to zero. That is, the macroscopic equations could formally emerge from $(\ref{equ:N-body schroedinger equation})$ as an iterated limit: $\lim_{\hbar \to 0}\lim_{N\to\infty}$. Such an iterated limit is far from satisfactory in
either mathematics or physics. Not only an iterated limit could lose information in any one limit, it kills the fine interplay between $\hbar$ and $N$ and hence cannot show the $(N,\hbar)$ threshold at which classical behavior starts to dominate. Therefore, for a more complete and deeper understanding, we deal with the $(N,\hbar)$ double limit which is also a more challenging problem.

Our limiting macroscopic equation is the 3D compressible Euler equation, which is,
\begin{equation}\label{equ:euler equation}
\begin{cases}
& \partial _{t}\rho +\nabla \cdot \left( \rho u\right) =0, \\
& \partial _{t}u+(u\cdot \nabla )u+b_{0}\nabla \rho =0,\\
& (\rho,u)|_{t=0}=(\rho^{in},u^{in}),
\end{cases}
\end{equation}
if written in velocity form, or
\begin{equation}\label{equ:euler equation,momentum}
\begin{cases}
& \partial _{t}\rho +\operatorname{div} J =0,\\
& \partial _{t}J+\operatorname{div}\lrs{\frac{J\otimes J}{\rho}}+\frac{1}{2}\nabla \lrs{b_{0}\rho^{2}} =0,\\
& (\rho,J)|_{t=0}=(\rho^{in},J^{in}).
\end{cases}
\end{equation}
if written in momentum form. Here, as usual, $\rho(t,x):\R\times \R^{3}\mapsto \R$
is the mass density, $u(t,x)=(u^{1}(t,x),u^{2}(t,x),u^{3}(t,x)):\R\times \R^{3}\mapsto \R^{3}$
denotes the velocity of the fluid,  $J(t,x)=\lrs{\rho u}(t,x):\R\times \R^{3}\mapsto \R^{3}$ denotes the momentum of the fluid as the coupling constant\footnote{The equations $(\ref{equ:euler equation})$ and $(\ref{equ:euler equation,momentum})$ are not hyperbolic if the microscopic potential $V$ is focusing or $b_{0}<0$.} $b_{0}=\int V$ which is the macroscopic effect of the microscopic interaction $V$ and hints that pressure $b_{0}\rho^{2}$ should originate from the microscopic interaction between particles.

\subsection{Statement of the Main Theorem}
\begin{theorem}\label{thm:main theorem,bbgky-to-euler}
 Let $d=3$, $\be<\frac{2}{5}$, the marginal densities $\Ga_{N,\hbar}=\lr{\ga_{N,\hbar}^{(k)}}$ associated with $\psi_{N,\hbar}$ be the solution to the $N$-body dynamics with a Schwarz even pair interaction $V\geq 0$. The N-body initial data satisfies the following condition:

$(a)$ $\psi_{N,\hbar}(0)$ is normalized, that is, $\n{\psi_{N,\hbar}(0)}_{L^{2}}=1$.

$(b)$ The N-body energy bounds hold:
\begin{align} \label{equ:n-body energy bound,initial data condition}
\lra{\psi_{N,\hbar}(0),(H_{N,\hbar}/N+1)^{k}\psi_{N,\hbar}(0)}\leq (E_{0,\hbar})^{k}
\end{align}
for $k\leq (\ln N)^{100}$.

$(c)$ $\Ga_{N,\hbar}(0)$ is asymptotically factorized in the sense that
\begin{align} \label{equ:initial condition, factorized}
&\bbn{\prod_{j=1}^{k}\lra{\hbar\nabla_{x_{j}}}\lra{\hbar\nabla_{x_{j}'}}\lrc{\ga_{N,\hbar}^{(k)}(0)-|\phi_{N,\hbar}^{in}\rangle \langle \phi_{N,\hbar}^{in}|^{\otimes k}}}_{L_{x,x'}^{2}}\leq (E_{0,\hbar})^{k}N^{\frac{5}{2}\be-1}
\end{align}
for $k\leq (\ln N)^{100}$, where $\phi_{N,\hbar}^{in}$ is normalized that $\n{\phi_{N,\hbar}^{in}}_{L^{2}}=1$ and has finite energy\footnote{It is expected that $E_{0}\leq E_{0,\hbar}$ due to the correction structure.}, that is
\begin{align}
&\frac{1}{2}\n{\phi_{N,\hbar}^{in}}_{L^{2}}^{2}+\frac{1}{2}\n{\hbar\nabla \phi_{N,\hbar}^{in}}_{L^{2}}^{2}+\frac{1}{2}\lra{V_{N}*|\phi_{N,\hbar}^{in}|^{2},|\phi_{N,\hbar}^{in}|^{2}}\leq E_{0}.\label{equ:energy bound,one-body}
\end{align}

$(d)$ The initial datum $(\rho^{in},u^{in})$ to $(\ref{equ:euler equation})$ satisfy
 \begin{align}
 \rho^{in}\geq 0,\quad \int \rho^{in}(x)dx=1,
 \end{align}
 and is such that the Euler system $(\ref{equ:euler equation})$ has a solution $(\rho,u)$ satisfying
\begin{align}\label{equ:euler equation, regularity condition}
\begin{cases}
&(\rho,u)\in C([0,T_{0}];H^{s})\cap C^{1}([0,T_{0}];H^{s-1}),\\
&\rho\geq 0,\quad \int_{\R^{d}} \rho(t,x)dx=1,
\end{cases}
\end{align}
where $s>\frac{d}{2}+3$. The modulated / renormalized energy at initial time tends to zero:
\begin{align}\label{equ:modulated energy, initial convergence rate}
\int_{\R^{d}}|(i\hbar\nabla-u^{in})\phi_{N,\hbar}^{in}|^{2}dx+b_{0}\int_{\R^{d}}\lrs{|\phi_{N,\hbar}^{in}|^{2}-\rho^{in}}^{2}dx
\leq C\hbar^{2}.
\end{align}

Then under the restriction that\footnote{The composite function $e^{(n)}(x):=e^{(e^{(n-1)}(x))}$ and $C_{V}$ is a constant which only depends on some Sobolev norms of $V$ as needed in the proof.}
\begin{align}\label{equ:restriction,N,h}
N\geq e^{(2)}\lrs{\lrc{C_{V}^{2}E_{0,\hbar}^{2}T_{0}/\hbar^{7}}^{2}},
\end{align}
for $N\geq N_{0}(\be)$ and $(\rho,u)$ satisfying $(\ref{equ:euler equation})$, we have the quantitative estimates on the convergence of the mass density
\begin{align}\label{equ:the convergence of the mass denstiy}
& \n{\ga_{N,\hbar}^{(1)}(t,x;x)-\rho(t,x)}_{L_{t}^{\infty}[0,T_{0}]L_{x}^{2}(\R^{d})}\leq C(T_{0})\lrs{\frac{1}{\ln N}+\hbar},
\end{align}
on the convergence of the momentum density for $r\in (1,4/3)$
\begin{align}\label{equ;the convergence of the momentum density}
& \bbn{\operatorname{Im}\lrs{\hbar\nabla_{x_{1}} \ga_{N,\hbar}^{(1)}}(t,x;x)-(\rho u)(t,x)}_{L_{t}^{\infty}[0,T_{0}]L_{x}^{r}(\R^{d})}\leq C(T_{0})\lrs{\frac{1}{(\ln N)^{5(1-\frac{1}{r})}}+\hbar^{\frac{4-3r}{r}}},
\end{align}
and on the emergence of pressure
\begin{align}\label{equ:convergence, uniform in h,bbgky-euler,pressure}
&\bbn{\int V_{N}(x-x_{2})\ga_{N,\hbar}^{(2)}(t,x,x_{2};x,x_{2})dx_{2}-b_{0}\rho(t,x)^{2}}_{L_{t}^{1}[0,T_{0}]L_{x}^{1}(B_{R})}
\leq  C(T_{0})\lrs{\frac{R^{d/2}}{\ln N}+\hbar}.
\end{align}
where coupling constant is $b_{0}=\int V$.
\end{theorem}

Theorem $\ref{thm:main theorem,bbgky-to-euler}$ is the first of its type and involves the up-to-date techniques in the hierarchy method as well as well-developed modulated energy approach
and we can in fact see it from its assumptions.
The $N$-body energy condition in $(b)$ is inspired by purely factorized or statistically independent datum, and has been used since the 1st wave of work \cite{AGT07,EESY06,
ESY06,ESY07,ESY09,ESY10} on deriving NLS using hierarchy methods. It is usually cashed in as the $H^{1}$ bound on the marginals\footnote{We include a proof as Proposition $\ref{lemma:energy estimate}$ for completeness.}
\begin{align}\label{equ:n-body energy bound}
\bbn{\prod_{j=1}^{k}\lra{\hbar\nabla_{x_{j}}}\lra{\hbar\nabla_{x_{j}'}}\ga_{N,\hbar}^{(k)}(t)}_{L_{x,x'}^{2}}\leq \lrs{2E_{0,\hbar}}^{k}
\end{align}
for $k\leq (\ln N)^{100}$, $N\geq N_{0}(\be)$ which is independent of $k$ and $\hbar$, and all $t\in(-\infty,+\infty)$. Here, we allow the $k\geq 2$ energy bound $E_{0,\hbar}$ to depend on $\hbar$ (the $k=1$ case can be the same $E_{0}$ as in $(\ref{equ:energy bound,one-body})$) as long as it is finite for every nonzero $\hbar$, so that a larger variety of initial datum are included at the cost of the restriction $(\ref{equ:restriction,N,h})$ with an unspecific factor $E_{0,h}$. This is a natural requirement as the $k>2$ energy includes higher derivatives which do not play well with $\hbar$.
Though the initially asymptotic statistically independent assumption $(\ref{equ:initial condition, factorized})$ in $(c)$ is like usual in this line of work, the optimal decay rate is believed (and proved in some cases, see for example, \cite{BBCS19,BS19}) to be $1/N$ for every given $\hbar$. We assume $N^{\frac{5}{2}\be-1}$ here so that the paper is self-contained as we will prove this rate at the first step of bootstrapping argument. Indeed, for $\hbar=1$, the convergence rate has been achieved in \cite{CH21quantitative}. On the other hand, compared to the $N$-body energy bounds $(\ref{equ:n-body energy bound,initial data condition})$, the energy bound $E_{0}$ for $\phi_{N,h}^{in}$ is independent of $\hbar$ to be compatible with the modulated energy bound.

As for the assumptions regarding the initial datum of $(\ref{equ:euler equation})$, the local well-posedness of compressible Euler equations has been studied by many authors, for example, see the monograph \cite{Maj84}. But we remark that, there are many variants / choices / constructions of the modulated energy $(\ref{equ:modulated energy, initial convergence rate})$ which look seemingly different but are intuitively and closely related up to an error term as the initial quantities like $|\phi_{N,\hbar}^{in}|^2$ and $\rho^{in}$ are supposed to be close. In fact, the full modulated energy which we will use and is going to be controlled by $(\ref{equ:modulated energy, initial convergence rate})$ takes the form
\begin{align}\label{equ:modulated energy,outline}
\mathcal{M}\lrc{\phi_{N,\hbar},\rho,u}(t)=&\frac{1}{2}\int_{\R^{d}}|(i\hbar\nabla-u)\phi_{N,\hbar}(t)|^{2}dx\\
&+\frac{1}{2}\lra{V_{N}*|\phi_{N,\hbar}|^{2},
|\phi_{N,\hbar}|^{2}}+\frac{b_{0}}{2}\int_{\R^{d}}\rho^{2}dx-b_{0}\int_{\R^{d}}\rho|\phi_{N,\hbar}|^{2}dx.\notag
\end{align}
We assume the convergence rate $(\ref{equ:modulated energy, initial convergence rate})$ to be $\hbar^{2}$ which should also be optimal, since the smallness factor in the modulated kinetic part is at most $\hbar^{2}$. Besides, the $\hbar^{2}$ rate can be achieved with WKB type initial datum.

Theorem \ref{thm:main theorem,bbgky-to-euler} rigorously establishes the derivation of the
 macroscopic equation $(\ref{equ:euler equation})$ in classical mechanics from the quantum many-body systems as a regional double limit and provides convergent rate estimates in the strong norm sense. It also justifies the emergence of the macroscopic pressure from the space-time averages of microscopic interactions, which are in fact, Strichartz-type bounds. Notice that, the microscopic quantity converging to the pressure $\rho^{2}$ is basically $\gamma_{N,\hbar}^{(2)} (x,x,x,x)$. It is not necessarily finite or defined a.e. if we are below $H^{9/8}$ in 3D by the Sobloev embeddings, and we only have $H^{1}$ here. The Strichartz bound, 1st raised by Klanerman-Machedon (KM) \cite{KM08} in this context, makes this quantity well-defined and have unexpectedly verified the theory that pressure is the space-time averaging of the microscopic interactions under the physical $H^{1}$ assumption\footnote{Such an averaging effect certainly cannot be observed if one assumes higher than $H^{9/8}$ regularity at the $N$-body level, but we remark that it cannot be observed either if one passes through the NLS in the $H^1$ setting as $|\phi|^4$ is already defined a.e. without any need to appeal to Strichartz.}. We have hence found the 1st physical meaning for Strichartz type bounds since its original discovery in \cite{Str77}. Such a discovery is part of the main novelty of this paper. On the other hand, the limit in Theorem $\ref{thm:main theorem,bbgky-to-euler}$ is taken within the region $(\ref{equ:restriction,N,h})$ which proves the dominance of classical behaviors when $N>>\hbar$. Such a requirement is physical as they indeed differ by $10^{57}$ in reality but we believe $(\ref{equ:restriction,N,h})$ is not optimal and searching for the sharp threshold (may not exist, some mesoscopic behaviors might happen) between classical and quantum behaviors is certainly of interest. However, it would not be surprising to have totally independent $N$ and $\hbar$ in weak / weak* limits as the classical-now-elementary Riemann-Lebesgue lemma shows that a weak convergent sequence can be uniformly bounded away from its weak limit. To work with the 3D $N$-body equation smoothly in the physical $H^{1}$ energy space, we improvise and extend the up-to-date hierarchy method in KM format.

The hierarchy method in general was 1st suggested by M. Kacs and proved to be successful in Lanford's work \cite{Lan75} regarding the Boltzmann equation. The hierarchy method we use in the paper is actually more originated from the 1st wave of work \cite{AGT07,ESY07,ESY09,ESY10} by Adami-Golse-Teta and Erd\"{o}s-Schlein-Yau on deriving NLS from quantum many-body dynamics around 2005 as suggested by Spohn \cite{Spo80}.
At that time, the main difficulty lies in the uniqueness of the infinite Gross-Pitaevskii (GP)
hierarchy. With a sophisticated Feynman graph analysis in the fundamental papers \cite{ESY07,ESY09,ESY10} which derived the 3D cubic defocusing NLS, Erd{\"o}s, Schlein, and Yau proved the $H^{1}$-type unconditional uniqueness of the $\R^{3}$ cubic GP hierarchy. The first series of ground breaking papers have motivated a large amount of work.

Subsequently in 2007, by imposing an additional a-prior condition on space-time norm, Klainerman and Machedon \cite{KM08}, inspired by \cite{ESY07,KM93}, gave another uniqueness criterion of the GP hierarchy in a different space of density matrices defined by Strichartz type norms. They provided a different combinatorial argument, the now so-called Klainerman-Machedon board game, to combine the inhomogeneous terms effectively reducing their numbers and then derived a space-time estimate to control these terms.
At that time, it was open on how to prove that the limits coming from the $N$-body dynamics satisfy the now so called KM space-time bound required for uniqueness. Nonetheless, \cite{KM08} has made the delicate analysis of the GP hierarchy approachable from the perspective of PDE. Klainerman and Machedon also did not know the KM bound required for uniqueness, which is an usual product of Strichartz type well-posedness theory, actually has a physical meaning.\footnote{Private communication with M. Machedon.}

Later, Kirkpatrick, Schlein, and Staffilani \cite{KSS11} obtained the KM space-time bound via a simple trace theorem in both $\R^{2}$ and $\T^{2}$ and derived the 2D cubic defocusing NLS from the 2D quantum many-body dynamic. Such a scheme also motivated many works \cite{CP11,Che12,CH16focusing,CH17,GSS14,
HS16,Soh16,SS15,Xie15} for the uniqueness of GP hierarchies and enables the hierarchy method on the derivation 1D or 2D NLS directly from 3D \cite{CH13,CH17,She21}, which is quite different but has some similar flavor with our Theorem $\ref{thm:main theorem,bbgky-to-euler}$ here. However, how to verify the KM bound in the 3D cubic case remained fully open at that time.

 Then in 2011, T. Chen and Pavlovi{\'c} proved that the 3D cubic
KM space-time bound held for the defocusing $\be<1/4$ case in \cite{CP14}.
The result was quickly improved to $\be<2/7$ by X. Chen in \cite{Che13} and then extended to the almost optimal case, $\be<1$, by X. Chen and Holmer in \cite{CH16correlation,CH16on}, by lifting the $X_{1,b}$ space techniques from NLS theory into the field. Away from being the first work to prove the 3D KM bound,
the work \cite{CP14} hinted two unforeseen directions of the hierarchy method: one direction is to prove new NLS results via the more complicated hierarchies, while the other is that it is possible to derive NLS without a compactness or uniqueness argument as in the 1st wave of papers.

In 2013, by introducing the quantum de Finetti theorem from \cite{LNR14} to the field, T. Chen, Hainzl, Pavlovi$\acute{c}$
and Seiringer \cite{CHPS15} provided a simplified proof of the $L_{t}^{\wq}H_{x}^{1}$-type 3D cubic uniqueness theorem
as stated in \cite{ESY07}. This method motivated many work \cite{CS14,HTX15,HTX16,Soh15} and has climbed to a climax recently as the previously open $\T^{d}$ energy-critical and supercritical NLS unconditional uniqueness problems progressed in \cite{HS19} were completely and unifiedly resolved via the analysis of the supposedly more complicated GP hierarchy in \cite{CH19,CH20,CSZ21} which used, the $l^2$ decoupling theorem \cite{BD15} and has helped in the derivation of the energy-critical NLS \cite{CH19,CH20}. With these new exciting developments, it seems that KM bound method is obsolete though the KM board game stays useful. Such an impression or conclusion is apparently wrong.

 Recently, on the basis of \cite{CP14,Che13,CH16correlation,CH16on}, X. Chen and Holmer in \cite{CH21quantitative} reformatted the hierarchy method with KM space-time estimates and proved a bi-scattering theorem for the NLS to obtain almost optimal local in time convergence rate estimates under $H^{1}$ regularity. They integrate the idea from the Fock space approach (see, for example, \cite{BdOS15,BCS17,BS19,GM13,GM17}  and references within\footnote{The Fock space approach is also a vast and deep subject right now. There are certainly more references available. But this paper is not directly related to that.}), that, using H-NLS as an intermediate dynamic, into the hierarchy method. Most notably, the work \cite{CH21quantitative}, though it did not use the KM bound, sheds light on our principal part in which we prove strong, quantitative, uniform in $\hbar$, estimates regarding the BBGKY hierarchy and the H-NLS hierarchy.

On the other hand, the behavior of the wave function of cubic defocusing NLS as the
Planck's constant goes to zero is studied by many authors using various approaches. In \cite{Gre98}, Grenier derived compressible Euler equations for small time from cubic NLS by WKB.
Jin, Levermore and
McLaughlin in \cite{JLM99} established the semiclassical limit of the 1D defocusing cubic NLS for all time by using the complete integrability. In \cite{LZ06}, F. Lin and P. Zhang investigated Gross-Pitaevskii equation (a cubic Schr\"{o}dinger equation nonzero at infinity) in 2D exterior domains by adopting the modulated energy method. For a more detailed survey related to semiclassical limits of NLS, see \cite{Car08,Zha08} and references within.

As seen from above, it is highly nontrivial to derive Euler equations from NLS, let alone from quantum $N$-body dynamics. As the first breakthrough, Golse and Paul \cite{GP21}, with the help of Serfaty's inequality \cite[Corollary 3.4]{Ser20}, used the modulated energy method in the quantum $N$-body setting to justify the validity of the joint mean-field and classical limit of the quantum $N$-body dynamics leading to the pressureless Euler-Poisson with repulsive Coulomb potential. Subsequently, Rosenzweig complemented \cite{GP21} in \cite{Ros21} by combining mean-field, semiclassical and quasi-neutral limits to reach a derivation of an incompressible Euler equation on $\T^{d}$ with binary Coulomb interactions.

Though both singular, the $\delta$-interaction, which results in a compressible Euler equation, is substantially different from the Coulomb potential and calls for new ideas. The strong convergence and quantitative estimates are much more demanding as well. Our proof combines improvision and extension of up-to-date techniques in the hierarchy method and the well-developed modulated energy method.

\subsection{Outline of the Proof}
 $(\ref{equ:N-body schroedinger equation})$ is very different from our goal $(\ref{equ:euler equation})$ or $(\ref{equ:euler equation,momentum})$, at least by the look of them. Key quantities of $\ga_{N,\hbar}^{(k)}$ in
 $(\ref{equ:the convergence of the mass denstiy})-(\ref{equ:convergence, uniform in h,bbgky-euler,pressure})$ are all traces and thus as usual, are regularity thirsty and does not react well as $\hbar\to 0$, while solutions to $(\ref{equ:euler equation})$ will blow up in finite time. Thus we insert H-NLS $(\ref{equ:N-hartree equation})$\footnote{We expect more NLS like behaviors from $(\ref{equ:N-hartree equation})$ due to the context and hence we call it H-NLS.} as an intermediate dynamic.
We hence divide the proof of Theorem \ref{thm:main theorem,bbgky-to-euler} into two parts in Sections \ref{section:BBGKY hierarchy v.s. H-NLS: Long-time Uniform in h Estimates} and  \ref{section:H-NLS v.s. the Compressible Euler Equation: a Modulated Energy Approach} respectively. The first part is the quantitative estimate between the BBGKY hierarchy and the H-NLS using an improvised and extended version of cutting edge hierarchy methods, while the second part is comparing the H-NLS equation with the compressible Euler equation $(\ref{equ:euler equation})$ by means of modulated energy approach. Here, we are using the BBGKY hierarchy directly satisfied by $\ga_{N,\hbar}^{(k)}$. We are not using any Wigner transforms in this paper. Theorem \ref{thm:main theorem,bbgky-to-euler} then follows from summing the concluding estimates in Sections \ref{section:BBGKY hierarchy v.s. H-NLS: Long-time Uniform in h Estimates} and  \ref{section:H-NLS v.s. the Compressible Euler Equation: a Modulated Energy Approach}.

There are two main difficulties in Section \ref{section:BBGKY hierarchy v.s. H-NLS: Long-time Uniform in h Estimates}. One is to make sure all the differences estimates are uniform in $\hbar$. The other one is to make sure the estimates hold for every finite time despite that the method \cite{CH21quantitative} only works local in time. How to circumvent these two difficulties is also the main technical novelty of this paper. The key is to implement
the Klainerman-Machedon space-time bound, which was thought of only as a part of uniqueness, to strengthen our local in time quantitative estimate. The whole process is still very technical, we illustrate the principle logic of the proof of Section \ref{section:BBGKY hierarchy v.s. H-NLS: Long-time Uniform in h Estimates} by the following diagram.
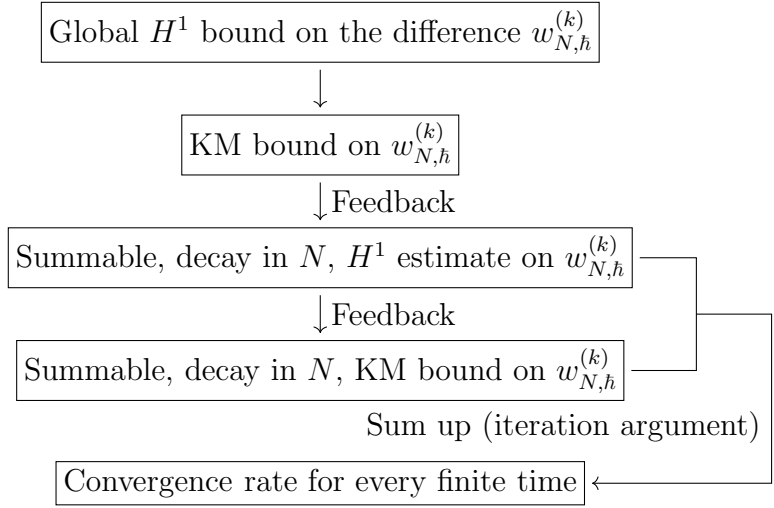
\begin{center}
\begin{tikzpicture}
\node[draw,rectangle,inner sep=0.1cm,outer sep=0.1cm] (1) at (0,0) {Global $H^{1}$ bound on
the difference $w_{N,\hbar}^{(k)}$};
\node[draw,rectangle,inner sep=0.1cm,outer sep=0.1cm] (2) at (0,-1.5) {KM bound on $w_{N,\hbar}^{(k)}$};
\node[draw,rectangle,inner sep=0.1cm,outer sep=0.1cm] (3) at (0,-3) {Summable, decay in $N$, $H^{1}$ estimate on $w_{N,\hbar}^{(k)}$};
\node[draw,rectangle,inner sep=0.1cm,outer sep=0.1cm] (4) at (0,-4.5) {Summable,
decay in $N$, KM bound on $w_{N,\hbar}^{(k)}$};
\node[draw,rectangle,inner sep=0.1cm,outer sep=0.1cm] (5) at (0,-6) {Convergence rate for every finite time};
\node[right] (6) at (0,-2.25) {Feedback};
\node[right] (7) at (0,-3.75) {Feedback};
\node[left] (8) at (6,-5.25) {Sum up (iteration argument)};
\draw[->] (1)--(2);
\draw[->] (2)--(3);
\draw[->] (3)--(4);
\draw[-] (3)--(5,-3)--(5,-4.5)--(4);
\draw[->] (5,-3.75)--(6,-3.75)--(6,-6)--(5);
\end{tikzpicture}
\end{center}

The logic above looks quite like proving global well-posedness for a $H^{1}$ subcritical NLS. However, this is the 1st time such a diagram is carried out for the hierarchy analysis. The technical reason is exactly as mention before (and in almost all paper in this field), though the $N$-body equations and hierarchies are linear, we are dealing with traces instead of powers.

In Section \ref{section:A Tool Box of Space-time Estimates}, we first provide some preliminary or crude estimates for the difference between BBGKY hierarchy and H-NLS hierarchy. We then prove in Section \ref{section:A Klainerman-Machedon Bound 1st} that $w_{N,\hbar}^{(k)}$ satisfies the Klainerman-Machedon bound by gathering information from the $(\ln N)^{10}$ coupling level.  Subsequently in Section \ref{section:Feeding the Strichartz Bound into the $H^1$ Estimate}, we feed the KM bound / a Strichartz bound back, to strengthen the $H^{1}$ estimate
for $k<(\ln N)^2$ to obtain summable and decay in $N$ estimates. We can further feed the $H^{1}$ estimate of $w_{N,\hbar}^{(k)}$ back into the KM bound proof and deduce that the KM bound actually decays in $N$. Notice the difference between the given $k$-th marginal and the selectable coupling level. For a given $k$-th marginal, how to select a suitable coupling level to yield desired information is a fine technical point. Section \ref{section:A Klainerman-Machedon Bound 1st} to \ref{section:Feeding the Strichartz Bound into the $H^1$ Estimate} addresses this issue.
Finally, in Section \ref{section:Convergence Rate for Every Finite Time}, with the conclusion in Section \ref{section:Feeding the Strichartz Bound into the $H^1$ Estimate}, we can sacrifice some decays in $N$ to bootstrap the quantitative estimates to every finite time by a clever but elementary manipulation.

As the $N$-body estimates have been set ready in Section \ref{section:BBGKY hierarchy v.s. H-NLS: Long-time Uniform in h Estimates},
in Section \ref{section:H-NLS v.s. the Compressible Euler Equation: a Modulated Energy Approach}, we adopt modulated energy method to compare directly the H-NLS equation with compressible Euler equations before the blowup time. The idea of proving convergence is via a Gronwall argument on modulated
energies assuming and using the regularity of the limiting solution.
 Therefore, in Section \ref{section:The Evolution of the Modulated Energy}, we compute the evolution of modulated energy. Subsequently in Section \ref{section:Modulated Energy Estimate}, we control the error term originating from the evolution of modulated energy to obtain a Gronwall type estimate. Due to the work in Section \ref{section:BBGKY hierarchy v.s. H-NLS: Long-time Uniform in h Estimates}, we are able to have a close match inside the modulated energy, and hence the error term is very tractable.

The main novelty of the paper is Theorem \ref{thm:main theorem,bbgky-to-euler} which establishes a strong microscopic to macroscopic derivation up to the 1st blow up time of the limiting Euler equation from the fundamental quantum $N$-body dynamics. The proof also combines the hierarchy method and the modulated energy method for the 1st time. We indeed anticipated more fusion of these two methods in the future. During the course of proof, we have implemented the Klainerman-Machedon Strichartz type bound and hence verified the emergence of pressure as the space-time averagings of microscopic interaction. This argument thus discovers a physical meaning for Strichartz type bounds for PDE and harmonic analysis.

\section{BBGKY Hierarchy v.s. H-NLS: Long-time Uniform in $\hbar$ Estimates}\label{section:BBGKY hierarchy v.s. H-NLS: Long-time Uniform in h Estimates}
The main goal in this section is to establish long-time uniform in $\hbar$ estimate for the difference $\ga_{N,\hbar}^{(k)}-|\phi_{N,\hbar}\rangle \langle \phi_{N,\hbar}|^{\otimes k}$ where
$\phi_{N,\hbar}$ is the solution to H-NLS equation as below
\begin{align}\label{equ:N-hartree equation}
\begin{cases}
&i\hbar\pa_{t}\phi_{N,\hbar}=-\frac{1}{2}\hbar^{2}\Delta\phi_{N,\hbar}+\lrs{V_{N}*|\phi_{N,\hbar}|^{2}}\phi_{N,\hbar},\\
&\phi_{N,\hbar}(0)=\phi_{N,\hbar}^{in}.
\end{cases}
\end{align}
Our strategy is to use the hierarchy approach.
It is well-known that $\Ga_{N,\hbar}(t)=\lr{\ga_{N,\hbar}^{(k)}}$ satisfies the Bogoliubov-Born-Green-Kirkwood-Yvon (BBGKY) hierarchy
\begin{align}\label{equ:bbgky hierarchy,differential form}
i\hbar\pa_{t}\ga_{N,\hbar}^{(k)}=&\sum_{j=1}^{k}\lrc{-\frac{\hbar^{2}}{2}\Delta_{x_{j}},\ga_{N,\hbar}^{(k)}}+
\frac{1}{N}\sum_{1\leq i<j\leq k}\lrc{V_{N}(x_{i}-x_{j}),\ga_{N,\hbar}^{(k)}}\\
&+\frac{N-k}{N}\sum_{j=1}^{k}\operatorname{Tr}_{k+1}\lrc{V_{N}(x_{j}-x_{k+1}),\ga_{N,\hbar}^{(k+1)}}.\notag
\end{align}
In addition to $(\ref{equ:bbgky hierarchy,differential form})$, we will use the so-called H-NLS hierarchy which takes the form
\begin{align}\label{equ:h-nls hierarchy,differential form}
i\hbar\pa_{t}\ga_{H,\hbar}^{(k)}=&\sum_{j=1}^{k}\lrc{-\frac{\hbar^{2}}{2}\Delta_{x_{j}},\ga_{H,\hbar}^{(k)}}
+\sum_{j=1}^{k}\operatorname{Tr}_{k+1}\lrc{V_{N}(x_{j}-x_{k+1}),\ga_{H,\hbar}^{(k+1)}},
\end{align}
generated by
$$\lr{\ga_{H,\hbar}^{(k)}(t,\textbf{x}_{k};\textbf{x}_{k}')=|\phi_{N,\hbar}\rangle \langle \phi_{N,\hbar}|^{\otimes k}},$$
the tensor products\footnote{As it is indeed a tensor product, the energy bound $(\ref{equ:n-body energy bound})$ also holds for $\ga_{H,\hbar}^{(k)}$ with $E_{0,\hbar}$ replaced by $E_{0}$.} of solutions to H-NLS equation $(\ref{equ:N-hartree equation})$.

Denote the difference between the BBGKY hierarchy and the H-NLS hierarchy by
\begin{align}
w_{N,\hbar}^{(k)}=\ga_{N,\hbar}^{(k)}-\ga_{H,\hbar}^{(k)}.
\end{align}
For convenience, we first set up some notations.
Define
\begin{align}\label{equ:kinetic operator,h}
S_{\hbar}^{(1,k)}=\prod_{j=1}^{k}\lra{\hbar\nabla_{x_{j}}}\lra{\hbar\nabla_{x_{j}'}},
\end{align}
the collision operator
\begin{align}
B_{N,j,k+1}f^{(k+1)}=&B_{N,j,k+1}^{+}f^{(k+1)}-B_{N,j,k+1}^{-}f^{(k+1)}\\
=&\int V_{N}(x_{j}-x_{k+1})f^{(k+1)}(\textbf{x}_{k},x_{k+1};\textbf{x}_{k}',x_{k+1})dx_{k+1}\notag\\
&-\int V_{N}(x_{j}'-x_{k+1})f^{(k+1)}(\textbf{x}_{k},x_{k+1};\textbf{x}_{k}',x_{k+1})dx_{k+1},\notag
\end{align}
and
\begin{align}
B_{N,\hbar,j,k+1}=\frac{1}{\hbar}B_{N,j,k+1},\quad B_{N,\hbar,j,k+1}^{\pm}=\frac{1}{\hbar}B_{N,j,k+1}^{\pm}.
\end{align}
Define the quantum mass density and momentum density in the quantum $N$-body setting
\begin{align}
\ga_{N,\hbar}^{(1)}(t,x;x),\quad J_{N,h}^{(1)}(t,x;x)=\operatorname{Im}\lrs{\hbar\nabla_{x_{1}} \ga_{N,\hbar}^{(1)}}(t,x;x)
\end{align}
and
\begin{align}\label{equ:quantum mass density,momentum density,one-body}
\rho_{N,\hbar}(t,x)=|\phi_{N,\hbar}(t,x)|^{2},\quad J_{N,\hbar}(t,x)=\hbar\operatorname{Im}\lrs{\ol{\phi_{N,\hbar}}(t,x)\nabla_{x}\phi_{N,\hbar}(t,x)}
\end{align}
with respect to H-NLS equation.

Our main theorem of this section is the following.
\begin{theorem}\label{theorem:bbgky-to-n-hartree}
Let $\phi_{N,\hbar}(t)$ be the solution to H-NLS equation with the initial data $\phi_{N,h}^{in}$.
Under the same conditions $(a)$, $(b)$ and $(c)$ of Theorem $\ref{thm:main theorem,bbgky-to-euler}$ and the restriction that
\begin{align}\label{equ:restriction,N,h,bbgky-hnls}
N\geq e^{(2)}\lrs{\lrc{C_{V}^{2}E_{0,\hbar}^{2}T_{0}/\hbar^{7}}^{2}},
\end{align}
then for $N\geq N_{0}(\be)$ we have the quantitative estimates
\begin{align}
&\sup_{t\in [0,T_{0}]}\bn{S_{\hbar}^{(1,1)}w_{N,\hbar}^{(1)}(t)}_{L_{x,x'}^{2}}\leq \lrs{\frac{1}{\ln N}}^{100},\label{equ:decay estimate, uniform in h,bbgky-hartree}\\
&\int_{[0,T_{0}]}\bn{S_{\hbar}^{(1,1)}B_{N,\hbar,1,2}^{\pm}w_{N,\hbar}^{(2)}(t)}_{L_{x,x'}^{2}}dt\leq \lrs{\frac{1}{\ln N}}^{100},\label{equ:decay estimate, uniform in h,bbgky-hartree,collapsing part}
\end{align}
which implies that
\begin{align}
&\n{\ga_{N,\hbar}^{(1)}(t,x;x)-\rho_{N,\hbar}(t,x)}_{L_{t}^{\infty}[0,T_{0}]L_{x}^{2}(\R^{d})}\leq \frac{C}{\ln N},\label{equ:convergence, uniform in h,bbgky-hartree,density}\\
&\n{J_{N,\hbar}^{(1)}(t,x;x)-J_{N,\hbar}(t,x)}_{L_{t}^{\infty}[0,T_{0}]L_{x}^{r}(\R^{d})}\leq \frac{C}{(\ln N)^{5\min\lr{1-\frac{1}{r},\frac{3}{r}-2}}},\label{equ:convergence, uniform in h,bbgky-hartree,moment}\\
&\bbn{\lrs{B_{N,1,2}^{\pm}\ga_{N,\hbar}^{(2)}}\lrs{t,x;x}-\lrs{\rho_{N,\hbar}
V_{N}*\rho_{N,\hbar}}(t,x)}_{L_{t}^{1}[0,T_{0}]L_{x}^{1}(B_{R})}
\leq C \frac{R^{d/2}+T_{0}}{\ln N}\label{equ:convergence, uniform in h,bbgky-hartree,pressure},
\end{align}
where $r\in (1,\frac{3}{2})$. Here $\pm$ does not matter as $\lrs{B_{N,1,2}^{+}\ga_{N,\hbar}^{(2)}}(t,x;x)=\lrs{B_{N,1,2}^{-}\ga_{N,\hbar}^{(2)}}(t,x;x)$.

\end{theorem}
\begin{proof}[\textbf{Proof of Theorem $\ref{theorem:bbgky-to-n-hartree}$}]
We prove $(\ref{equ:decay estimate, uniform in h,bbgky-hartree})$ and $(\ref{equ:decay estimate, uniform in h,bbgky-hartree,collapsing part})$ in Proposition \ref{proposition:decay estimate, uniform in h, bbgky-hartree}. Here, we prove $(\ref{equ:convergence, uniform in h,bbgky-hartree,density})-(\ref{equ:convergence, uniform in h,bbgky-hartree,pressure})$ using $(\ref{equ:decay estimate, uniform in h,bbgky-hartree})$
and $(\ref{equ:decay estimate, uniform in h,bbgky-hartree,collapsing part})$.
For the mass density estimate $(\ref{equ:convergence, uniform in h,bbgky-hartree,density})$,
we split
\begin{align}
w_{N,\hbar}^{(1)}=&\lrs{P_{\leq M}^{1'}+P_{>M}^{1'}}w_{N,\hbar}^{(1)},
\end{align}
where $P_{\leq M}$ denotes the Littlewood-Paley projection with $M$ to be determined.

For the low frequency part, by Bernstein inequality and estimate $(\ref{equ:decay estimate, uniform in h,bbgky-hartree})$, we have
\begin{align*}
\bbn{\lrs{P_{\leq M}^{1'}w_{N,\hbar}^{(1)}}(t,x;x)}_{L_{x}^{2}}
\leq& \bbn{\lrs{P_{\leq M}^{1'}w_{N,\hbar}^{(1)}}(t,x;x')}_{L_{x}^{2}L_{x'}^{\infty}}
\lesssim  M^{\frac{d}{2}}\bn{w_{N,\hbar}^{(1)}}_{L_{x_{1},x_{1}'}^{2}}
\lesssim  \frac{M^{\frac{d}{2}}}{(\ln N)^{100}}.
\end{align*}

For the high frequency part, by triangle inequality we have
\begin{align*}
\bbn{\lrs{P_{>M}^{1'}w_{N,\hbar}^{(1)}}(t,x;x)}_{L_{x}^{2}}\leq
\bbn{\lrs{P_{>M}^{1'}\ga_{N,\hbar}^{(1)}}(t,x;x)}_{L_{x}^{2}}
+\bbn{\lrs{P_{>M}^{1'}\ga_{H,\hbar}^{(1)}}(t,x;x)}_{L_{x}^{2}}.
\end{align*}
It suffices to deal with $\ga_{N,\hbar}^{(1)}$ as we can estimate $\ga_{H,\hbar}^{(1)}$ in the same way.
We use interpolation between $L^{1}$ and $L^{3}$
\begin{align}
\bbn{\lrs{P_{>M}^{1'}\ga_{N,\hbar}^{(1)}}(t,x;x)}_{L_{x}^{2}}\leq \bbn{\lrs{P_{>M}^{1'}\ga_{N,\hbar}^{(1)}}(t,x;x)}_{L_{x}^{1}}^{\frac{1}{4}}
\bbn{\lrs{P_{>M}^{1'}\ga_{N,\hbar}^{(1)}}(t,x;x)}_{L_{x}^{3}}^{\frac{3}{4}}.
\end{align}
For the $L_{x}^{1}$ norm, we have, by definition of $\ga_{N,\hbar}^{(1)}$ that
\begin{align}
\bbn{\lrs{P_{>M}^{1'}\ga_{N,\hbar}^{(1)}}(t,x;x)}_{L_{x}^{1}}
=&\int_{\R^{d}}\bbabs{\int \psi_{N,\hbar}(t,x,\textbf{x}_{2,N})\ol{P_{>M}^{1}\psi_{N,\hbar}}
(t,x,\textbf{x}_{2,N})d\textbf{x}_{2,N}}dx\notag
\end{align}
where we have used $\textbf{x}_{2,N}=(x_{2},...,x_{N})$ for short.
By Cauchy-Schwarz and Bernstein,
\begin{align*}
\leq&\n{\psi_{N,\hbar}}_{L^{2}}\n{P_{>M}^{1}\psi_{N,\hbar}}_{L^{2}}\\
\leq&\n{\psi_{N,\hbar}}_{L^{2}} \frac{1}{\hbar M}\n{\lra{\hbar\nabla_{x_{1}}}\psi_{N,\hbar}}_{L^{2}}
\end{align*}
By the $N$-body energy bound $(\ref{equ:n-body energy bound})$, we reach
\begin{align}\label{equ:littlewood,high frequency,L1 norm}
\bbn{\lrs{P_{>M}^{1'}\ga_{N,\hbar}^{(1)}}(t,x;x)}_{L_{x}^{1}}\lesssim \frac{E_{0,\hbar}^{1/2}}{\hbar M}.
\end{align}
Similarly, for the $L_{x}^{3}$ norm, we have
\begin{align}\label{equ:littlewood,high frequency,L3 norm}
&\bbn{\lrs{P_{>M}^{1'}\ga_{N,\hbar}^{(1)}}(t,x;x)}_{L_{x}^{3}}\\
=&\lrc{\int_{\R^{d}}\bbabs{\int \psi_{N,\hbar}(t,x,\textbf{x}_{2,N})\ol{P_{>M}^{1}\psi_{N,\hbar}}
(t,x,\textbf{x}_{2,N})d\textbf{x}_{2,N}}^{3}dx}^{\frac{1}{3}}\notag
\end{align}
By H\"{o}lder, Minkowski, Sobolev, and the $N$-body energy bound $(\ref{equ:n-body energy bound})$, we get
\begin{align*}
\leq& \n{\psi_{N,\hbar}}_{L_{\textbf{x}_{2,N}}^{2}L_{x_{1}}^{6}}\n{P_{>M}^{1}\psi_{N,\hbar}}_{L_{\textbf{x}_{2,N}}^{2}L_{x_{1}}^{6}}\\
\lesssim& \n{\lra{\nabla_{x_{1}}}\psi_{N,\hbar}}_{L^{2}}\n{\lra{\nabla_{x_{1}}}P_{>M}^{1}\psi_{N,\hbar}}_{L^{2}}
\lesssim \frac{E_{0,\hbar}}{\hbar^{2}}.
\end{align*}
Combining $(\ref{equ:littlewood,high frequency,L1 norm})$ and $(\ref{equ:littlewood,high frequency,L3 norm})$, we obtain
\begin{align}
\bbn{\lrs{P_{>M}^{1'}\ga_{N,\hbar}^{(1)}}(t,x;x)}_{L_{x}^{2}}\lesssim \frac{1}{M^{1/4}}\lrs{\frac{E_{0,\hbar}}{\hbar^{2}}}^{\frac{7}{8}}.
\end{align}
By taking $M=(\ln N)^{10}$ and adopting the restriction $(\ref{equ:restriction,N,h,bbgky-hnls})$, we obtain $(\ref{equ:convergence, uniform in h,bbgky-hartree,density})$.

For the momentum estimate $(\ref{equ:convergence, uniform in h,bbgky-hartree,moment})$, we set
\begin{align}
g_{N,\hbar}(t,x_{1};x_{1}')=\hbar\nabla_{x_{1}} \ga_{N,\hbar}^{(1)}(t,x_{1};x_{1}')-
\hbar\nabla_{x_{1}}\ga_{H,\hbar}^{(1)}(t,x_{1};x_{1}').
\end{align}
We split
\begin{align}
g_{N,\hbar}=\lrs{P_{\leq M}^{1'}+P_{>M}^{1'}}g_{N,\hbar}
\end{align}
with $M$ to be determined.
By Interpolation,
\begin{align}
&\n{(P_{\leq M}^{1'}g_{N,\hbar})(t,x;x)}_{L_{x}^{r}}\leq\n{(P_{\leq M}^{1'}g_{N,\hbar})(t,x;x)}_{L_{x}^{1}}^{\frac{2}{r}-1}\n{(P_{\leq M}^{1'}g_{N,\hbar})(t,x;x)}_{L_{x}^{2}}^{2-\frac{2}{r}}\leq I\cdot  II,\label{equ:convergence,uniform in h,bbgky-hartree,moment,low frequency}\\
&\n{(P_{>M}^{1'}g_{N,\hbar})(t,x;x)}_{L_{x}^{r}}\leq \n{(P_{>M}^{1'}g_{N,\hbar})(t,x;x)}_{L_{x}^{1}}^{\frac{3}{r}-2}
\n{(P_{>M}^{1'}g_{N,\hbar})(t,x;x)}_{L_{x}^{3/2}}^{3-\frac{3}{r}}\leq III\cdot IV.
\label{equ:convergence,uniform in h,bbgky-hartree,moment,high frequency}
\end{align}
Next, we separately estimate the above terms on the right hand side of $(\ref{equ:convergence,uniform in h,bbgky-hartree,moment,low frequency})$ and $(\ref{equ:convergence,uniform in h,bbgky-hartree,moment,high frequency})$.

For $I$, by triangle inequality we have
\begin{align*}
\n{(P_{\leq M}^{1'}g_{N,\hbar})(t,x;x)}_{L_{x}^{1}}\leq
\n{(P_{\leq M}^{1'}\hbar\nabla_{x_{1}}\ga_{N,\hbar}^{(1)})(t,x;x)}_{L_{x}^{1}}+
\n{(P_{\leq M}^{1'}\hbar\nabla_{x_{1}}\ga_{H,\hbar}^{(1)})(t,x;x)}_{L_{x}^{1}}.
\end{align*}
By Cauchy-Schwarz and the $N$-body energy bound $(\ref{equ:n-body energy bound})$, we have
\begin{align}
&\n{(P_{\leq M}^{1'}\hbar\nabla_{x_{1}}\ga_{N,\hbar}^{(1)})(t,x;x)}_{L_{x}^{1}}\leq  \n{\hbar\nabla \psi_{N,\hbar}}_{L^{2}}\n{P_{\leq M}^{1}\psi_{N,\hbar}}_{L^{2}}\leq E_{0,\hbar}^{1/2}.
\label{equ:convergence,uniform in h,bbgky-hartree,moment,low frequency,L1 norm}
\end{align}
Similarly, by Cauchy-Schwarz and the energy bound for $\phi_{N,\hbar}$, we have
\begin{align}\label{equ:convergence,uniform in h,bbgky-hartree,moment,low frequency,L1 norm,product}
&\n{(P_{\leq M}^{1'}\hbar\nabla_{x_{1}}\ga_{H,\hbar}^{(1)})(t,x;x)}_{L_{x}^{1}}\\
=&\n{(\hbar\nabla_{x_{1}}\phi_{N,\hbar})(t,x) \ol{P_{\leq M}\phi_{N,\hbar}}(t,x)}_{L_{x}^{1}}
\leq \n{\hbar\nabla_{x_{1}}\phi_{N,\hbar}}_{L_{x}^{2}}\n{P_{\leq M}\phi_{N,\hbar}}_{L^{2}}\leq E_{0}^{1/2}.\notag
\end{align}
With $E_{0}\leq E_{0,\hbar}$, we combine $(\ref{equ:convergence,uniform in h,bbgky-hartree,moment,low frequency,L1 norm})$ and $(\ref{equ:convergence,uniform in h,bbgky-hartree,moment,low frequency,L1 norm,product})$ to obtain
\begin{align}\label{equ:convergence,uniform in h,bbgky-hartree,moment,I}
I=\n{(P_{\leq M}^{1'}g_{N,\hbar})(t,x;x)}_{L_{x}^{1}}^{\frac{2}{r}-1}\lesssim E_{0,\hbar}^{\frac{1}{r}-\frac{1}{2}}.
\end{align}

For $II$, we use Bernstein inequality and estimate $(\ref{equ:decay estimate, uniform in h,bbgky-hartree})$ to get
\begin{align}\label{equ:convergence,uniform in h,bbgky-hartree,moment,low frequency,L2 norm}
\n{(P_{\leq M}^{1'}g_{N,\hbar})(t,x;x)}_{L_{x}^{2}}\leq &\n{(P_{\leq M}^{1'}g_{N,\hbar})(t,x;x')}_{L_{x}^{2}L_{x'}^{\wq}}
\lesssim M^{\frac{d}{2}}\n{g_{N,\hbar}(t,x;x')}_{L_{x}^{2}L_{x'}^{2}}
\lesssim \frac{M^{\frac{d}{2}}}{(\ln N)^{100}},
\end{align}
and hence
\begin{align}\label{equ:convergence,uniform in h,bbgky-hartree,moment,II}
II\lesssim \lrs{\frac{M^{\frac{d}{2}}}{(\ln N)^{100}}}^{2-\frac{2}{r}}.
\end{align}

For $III$, by triangle inequality we have
\begin{align}
\n{(P_{>M}^{1'}g_{N,\hbar})(t,x;x)}_{L_{x}^{1}}\leq \bbn{\lrs{P_{>M}^{1'}\hbar\nabla_{x_{1}}\ga_{N,\hbar}^{(1)}}(t,x;x)}_{L_{x}^{1}}+
\bbn{\lrs{P_{>M}^{1'}\hbar\nabla_{x_{1}}\ga_{H,\hbar}^{(1)}}(t,x;x)}_{L_{x}^{1}}.
\end{align}
We use Cauchy-Schwarz, Bernstein, and the $N$-body energy bound $(\ref{equ:n-body energy bound})$ to obtain
\begin{align}\label{equ:convergence,uniform in h,bbgky-hartree,moment,high frequency,L1 norm}
&\bbn{\lrs{P_{>M}^{1'}\hbar\nabla_{x_{1}}\ga_{N,\hbar}^{(1)}}(t,x;x)}_{L_{x}^{1}}\\
=&\int \bbabs{\int \hbar\nabla_{x_{1}} \psi_{N,\hbar}(t,x,x_{2,N})\ol{P_{>M}^{1'}\psi_{N,\hbar}}(t,x,x_{2,N})dx_{2,N}}dx\notag\\
\leq& \n{\hbar\nabla_{x_{1}}\psi_{N,\hbar}}_{L^{2}}\n{P_{>M}^{1}\psi_{N,\hbar}}_{L^{2}}\notag\\
\leq& M^{-1}\n{\hbar\nabla_{x_{1}}\psi_{N,\hbar}}_{L^{2}}\n{\lra{\nabla_{x_{1}}}P_{>M}^{1}\psi_{N,\hbar}}_{L^{2}}\notag\\
\lesssim &\frac{E_{0,\hbar}}{\hbar M}\notag
\end{align}
In the same method, we use the energy bound for $\phi_{N,\hbar}$ to get
\begin{align}\label{equ:convergence,uniform in h,bbgky-hartree,moment,high frequency,L1 norm,hartree}
\bbn{\lrs{P_{>M}^{1'}\hbar\nabla_{x_{1}}\ga_{H,\hbar}^{(1)}}(t,x;x)}_{L_{x}^{1}}\lesssim
\frac{E_{0}}{\hbar M}.
\end{align}
Combining $(\ref{equ:convergence,uniform in h,bbgky-hartree,moment,high frequency,L1 norm})$ with $(\ref{equ:convergence,uniform in h,bbgky-hartree,moment,high frequency,L1 norm,hartree})$, we have
\begin{align}\label{equ:convergence,uniform in h,bbgky-hartree,moment,III}
III\lesssim \lrs{\frac{E_{0,\hbar}}{\hbar M}}^{\frac{3}{r}-2}.
\end{align}

For $IV$, we use H\"{o}lder, Minkowski, Sobolev, and the $N$-body energy bound $(\ref{equ:n-body energy bound})$ to obtain
\begin{align}\label{equ:convergence,uniform in h,bbgky-hartree,moment,low frequency,L3/2 norm}
&\bbn{\lrs{P_{>M}^{1'}\hbar\nabla_{x_{1}}\ga_{N,\hbar}^{(1)}}(t,x;x)}_{L_{x}^{3/2}}\\
=&\lrc{\int \bbabs{\int \hbar\nabla_{x_{1}} \psi_{N,\hbar}(t,x,x_{2,N})\ol{P_{>M}^{1'}\psi_{N,\hbar}}(t,x,x_{2,N})dx_{2,N}}^{\frac{3}{2}}dx}^{\frac{2}{3}}\notag\\
\leq& \n{\hbar\nabla_{x_{1}}\psi_{N,\hbar}}_{L_{x_{1}}^{2}L^{2}_{\textbf{x}_{2,N}}}
\n{P_{>M}\psi_{N,\hbar}}_{L_{x_{1}}^{6}L^{2}_{\textbf{x}_{2,N}}}\notag\\
\leq& \n{\hbar\nabla_{x_{1}}\psi_{N,\hbar}}_{L_{x_{1}}^{2}L^{2}_{\textbf{x}_{2,N}}}
\n{P_{>M}\psi_{N,\hbar}}_{L^{2}_{\textbf{x}_{2,N}}L_{x_{1}}^{6}}\notag\\
\leq& \n{\hbar\nabla_{x_{1}}\psi_{N,\hbar}}_{L_{x_{1}}^{2}L^{2}_{\textbf{x}_{2,N}}}
\n{\lra{\nabla_{x_{1}}}\psi_{N,\hbar}}_{L^{2}_{\textbf{x}_{2,N}}L_{x_{1}}^{2}}\notag\\
\lesssim &\frac{E_{0,\hbar}}{\hbar}\notag
\end{align}
In the same method, we use the energy bound for $\phi_{N,\hbar}$ to get
\begin{align}\label{equ:convergence,uniform in h,bbgky-hartree,moment,low frequency,L3/2 norm,hartree}
\bbn{\lrs{P_{>M}^{1'}\hbar\nabla_{x_{1}}\ga_{H,\hbar}^{(1)}}(t,x;x)}_{L_{x}^{3/2}}\lesssim \frac{E_{0}}{\hbar}.
\end{align}
Combining $(\ref{equ:convergence,uniform in h,bbgky-hartree,moment,low frequency,L3/2 norm})$ with $(\ref{equ:convergence,uniform in h,bbgky-hartree,moment,low frequency,L3/2 norm,hartree})$, we have
\begin{align}\label{equ:convergence,uniform in h,bbgky-hartree,moment,IV}
IV\lesssim \lrs{\frac{E_{0,\hbar}}{\hbar}}^{3-\frac{3}{r}}.
\end{align}

Putting together with estimates $(\ref{equ:convergence,uniform in h,bbgky-hartree,moment,I})$, $(\ref{equ:convergence,uniform in h,bbgky-hartree,moment,II})$, $(\ref{equ:convergence,uniform in h,bbgky-hartree,moment,III})$ and $(\ref{equ:convergence,uniform in h,bbgky-hartree,moment,IV})$, we arrive at
\begin{align}
&\n{J_{N,\hbar}^{(1)}(t,x;x)-J_{N,\hbar}(t,x)}_{L_{x}^{r}}\\
\lesssim &E_{0,\hbar}^{\frac{1}{r}-\frac{1}{2}}\lrs{\frac{M^{\frac{d}{2}}}{(\ln N)^{100}}}^{2-\frac{2}{r}}+\lrs{\frac{E_{0,\hbar}}{\hbar M}}^{\frac{3}{r}-2}\lrs{\frac{E_{0,\hbar}}{\hbar}}^{3-\frac{3}{r}}\notag
\end{align}
Setting $M=(\ln N)^{20}$,
\begin{align*}
\leq& \frac{E_{0,\hbar}}{\hbar}\lrs{\frac{1}{(\ln N)^{10}}}^{\min\lr{1-\frac{1}{r},\frac{3}{r}-2}}.
\end{align*}
For fixed $r\in (1,3/2)$, we make use of the restriction $(\ref{equ:restriction,N,h,bbgky-hnls})$ to obtain
\begin{align}
\n{J_{N,\hbar}^{(1)}(t,x;x)-J_{N,\hbar}(t,x)}_{L_{x}^{r}}
\lesssim & \frac{1}{(\ln N)^{5\min\lr{1-\frac{1}{r},\frac{3}{r}-2}}},
\end{align}
which completes the proof of $(\ref{equ:convergence, uniform in h,bbgky-hartree,moment})$.

For the pressure estimate $(\ref{equ:convergence, uniform in h,bbgky-hartree,pressure})$, we set
\begin{align}
p^{\pm}_{N,\hbar}(t,x_{1},x_{1}')=\lrc{B_{N,1,2}^{\pm}\lrs{\ga_{N,\hbar}^{(2)}-\ga_{H,\hbar}^{(2)}}}(t,x_{1},x_{1}').
\end{align}
Again we split
\begin{align}
p^{\pm}_{N,\hbar}=\lrs{P_{\leq M}^{1'}+P_{>M}^{1'}}p^{\pm}_{N,\hbar}
\end{align}
with $M$ to be determined. We use H\"{o}lder and Bernstein inequalities to obtain
\begin{align}
&\bbn{\lrs{P_{\leq M}^{1'}p^{\pm}_{N,\hbar}}(t,x;x)}_{L^{1}([0,T_{0}];L_{x}^{1}(B_{R}))}\\
\leq&R^{\frac{d}{2}}\bbn{\lrs{P_{\leq M}^{1}p^{\pm}_{N,\hbar}}(t,x;x)}_{L^{1}([0,T_{0}];L_{x}^{2}(B_{R}))}\notag\\
\leq& R^{\frac{d}{2}}\n{\lrs{P_{\leq M}^{1'}p^{\pm}_{N,\hbar}}(t,x;x')}_{L^{1}([0,T_{0}];L_{x}^{2}L_{x'}^{\infty}(\R^{d}))}\notag\\
\leq&M^{\frac{d}{2}}R^{\frac{d}{2}}\n{\lrs{P_{\leq M}^{1'}p^{\pm}_{N,\hbar}}(t,x;x')}_{L^{1}([0,T_{0}];L_{x}^{2}L_{x'}^{2}(\R^{d}))}\notag
\end{align}
By estimate $(\ref{equ:decay estimate, uniform in h,bbgky-hartree,collapsing part})$, we arrive at
\begin{align}\label{equ:convergence,uniform in h,bbgky-hartree,pressure,low frequency}
\bbn{\lrs{P_{\leq M}^{1'}p^{\pm}_{N,\hbar}}(t,x;x)}_{L^{1}([0,T_{0}];L_{x}^{1}(B_{R}))}\leq \frac{\hbar M^{\frac{d}{2}}R^{\frac{d}{2}}}{(\ln N)^{100}}.
\end{align}

On the other hand, we note that
\begin{align*}
P_{>M}^{1'}B_{N,1,2}^{+}\ga_{N,\hbar}^{(2)}(t,x_{1};x_{1}')=&\int V_{N}(x_{1}-x_{2})\psi_{N,\hbar}(t,x_{1},\textbf{x}_{2,N})\ol{P_{>M}^{1}\psi_{N,\hbar}}(t,x_{1}',\textbf{x}_{2,N})d\textbf{x}_{2,N}.
\end{align*}
Hence, by Cauchy-Schwarz we have
\begin{align}\label{equ:convergence,uniform in h,bbgky-hartree,pressure,high frequency}
&\int |P_{>M}^{1'}B_{N,1,2}^{+}\ga_{N,\hbar}^{(2)}(t,x_{1};x_{1})| dx_{1}\\
\leq &\lra{\psi_{N,\hbar},V_{N}(x_{1}-x_{2})\psi_{N,\hbar}}^{1/2}
\lra{P_{>M}^{1}\psi_{N,\hbar},V_{N}(x_{1}-x_{2})P_{>M}^{1}\psi_{N,\hbar}}^{1/2}\notag
\end{align}
By estimate $(\ref{equ:Sobolev type estimate, 3d, L1})$ in
Lemma \ref{lemma:Sobolev type estimate, 3d}, Bernstein inequality, and the $N$-body energy bound $(\ref{equ:n-body energy bound})$,
\begin{align*}
&\lesssim \lra{\psi_{N,\hbar},(1-\Delta_{x_{1}})(1-\Delta_{x_{2}})\psi_{N,\hbar}}^{1/2}\lra{P_{>M}^{1}\psi_{N,\hbar},\lrc{(1-\Delta_{x_{1}})
(1-\Delta_{x_{2}})}^{\frac{d}{4}+}P_{>M}^{1}\psi_{N,\hbar}}^{1/2}\\
&\lesssim \frac{1}{M^{\lrs{1-\frac{d}{4}}-}} \n{\lra{\nabla_{x_{1}}}\lra{\nabla_{x_{2}}}\psi_{N,\hbar}}_{L^{2}}\n{\lra{\nabla_{x_{1}}}
\lra{\nabla_{x_{2}}}P_{>M}^{1}\psi_{N,\hbar}}_{L^{2}}\\
&\lesssim \frac{E_{0,\hbar}^{2}}{\hbar^{4}M^{\lrs{1-\frac{d}{4}}-}}.
\end{align*}
In the same method, we use the energy bound for $\phi_{N,\hbar}$ to get
\begin{align}\label{equ:convergence,uniform in h,bbgky-hartree,pressure,high frequency,hartree}
\n{P_{>M}^{1'}B_{N,1,2}^{\pm}\ga_{H,\hbar}^{(2)}(t,x;x)}_{L_{x}^{1}}
\lesssim \frac{E_{0}^{2}}{\hbar^{4}M^{\lrs{1-\frac{d}{4}}-}}.
\end{align}

Estimates $(\ref{equ:convergence,uniform in h,bbgky-hartree,pressure,low frequency})$, $(\ref{equ:convergence,uniform in h,bbgky-hartree,pressure,high frequency})$, and $(\ref{equ:convergence,uniform in h,bbgky-hartree,pressure,high frequency,hartree})$ together give
\begin{align}
&\bbn{\lrs{B_{N,1,2}^{\pm}\ga_{N,\hbar}^{(2)}}\lrs{t,x,x}
-\lrs{\rho_{N,\hbar}V_{N}*\rho_{N,\hbar}}(t,x)}_{L^{1}([0,T_{0}];L^{1}(B_{R}))}\\
=&\bbn{\lrc{B_{N,1,2}^{\pm}\lrs{\ga_{N,\hbar}^{(2)}-\ga_{H,\hbar}^{(2)}}}\lrs{t,x;x}
}_{L^{1}([0,T_{0}];L^{1}(B_{R}))}\notag\\
\lesssim &\frac{\hbar M^{\frac{d}{2}}R^{\frac{d}{2}}}{(\ln N)^{100}}+\frac{T_{0}E_{0,\hbar}^{2}}{\hbar^{4}M^{\lrs{1-\frac{d}{4}}-}}\notag
\end{align}
By taking $M=(\ln N)^{50}$,
\begin{align*}
\leq & \frac{\hbar R^{\frac{d}{2}}}{(\ln N)^{10}}+\frac{T_{0}E_{0,\hbar}^{4}}{\hbar^{4}(\ln N)^{10}}.
\end{align*}
For fixed $T_{0}$, we utilize the restriction $(\ref{equ:restriction,N,h,bbgky-hnls})$ to get
\begin{align*}
\bbn{\lrs{B_{N,1,2}^{\pm}\ga_{N,\hbar}^{(2)}}\lrs{t,x;x}
-\lrs{\rho_{N,\hbar}V_{N}*\rho_{N,\hbar}}(t,x)}_{L^{1}([0,T_{0}];L^{1}(B_{R}))}\lesssim \frac{R^{d/2}+T_{0}}{\ln N},
\end{align*}
which completes the proof of $(\ref{equ:convergence, uniform in h,bbgky-hartree,pressure})$.

\end{proof}
The proof of Theorem $\ref{theorem:bbgky-to-n-hartree}$ is hence concluded assuming
$(\ref{equ:decay estimate, uniform in h,bbgky-hartree})$ and $(\ref{equ:decay estimate, uniform in h,bbgky-hartree,collapsing part})$
included in Proposition
$\ref{proposition:decay estimate, uniform in h, bbgky-hartree}$. The rest of Section \ref{section:BBGKY hierarchy v.s. H-NLS: Long-time Uniform in h Estimates} is to prove Proposition $\ref{proposition:decay estimate, uniform in h, bbgky-hartree}$.

\subsection{A Tool Box of Space-time Estimates}\label{section:A Tool Box of Space-time Estimates}
We reproduce and rewrite \cite[Section 2]{CH21quantitative} with $\hbar$ for our purpose here and provide some preliminary estimates for $w_{N,\hbar}^{(k)}$.
We start by rewriting the $3D$ cubic BBGKY hierarchy
$(\ref{equ:bbgky hierarchy,differential form})$ in integral form
\begin{align}\label{equ:cubic BBGKY hierarchy, integral form}
\ga_{N,\hbar}^{(k)}=&U_{\hbar}^{(k)}\ga_{N,\hbar}^{(k)}(0)+\int_{0}^{t_{k}}
U_{\hbar}^{(k)}(t_{k}-t_{k+1})V_{N,\hbar}^{(k)}\ga_{N,\hbar}^{(k)}(t_{k+1})dt_{k+1}\\
&+\frac{N-k}{N}\int_{0}^{t_{k}}U_{\hbar}^{(k)}(t_{k}-t_{k+1})B_{N,\hbar}^{(k+1)}
\ga_{N,\hbar}^{(k+1)}(t_{k+1})dt_{k+1}\notag
\end{align}
where we have adopted the shorthands\footnote{Please notice that we have divided by $\hbar$ to
use $(\ref{equ:propogator with h})$.}
\begin{align}
&U_{\hbar}^{(k)}=\prod_{j=1}^{k}e^{it\hbar\Delta_{x_{j}}/2}e^{-it\hbar\Delta_{x_{j}'}/2},
\label{equ:propogator with h}\\
&V_{N,\hbar}^{(k)}\ga_{N,\hbar}^{(k)}=\frac{1}{N}\sum_{1\leq i<j\leq k}[V_{N,\hbar}(x_{i}-x_{j}),\ga_{N,\hbar}^{(k)}],\\
&V_{N,\hbar}(x)=\frac{1}{\hbar}N^{d\be}V(N^{\be}x),\\
&B_{N,\hbar}^{(k+1)}\ga_{N,\hbar}^{(k+1)}=
\sum_{j=1}^{k}B_{N,\hbar,j,k+1}\ga_{N,\hbar}^{(k+1)}=\sum_{j=1}^{k}\operatorname{Tr}_{k+1}
\lrc{V_{N,\hbar}(x_{j}-x_{k+1}),\ga_{N,\hbar}^{(k+1)}},
\end{align}
and we have omitted the $(-i)$ in front of the second and third terms in the right hand side of $(\ref{equ:cubic BBGKY hierarchy, integral form})$ as it serves as $1$ in our estimates. In addition to $(\ref{equ:cubic BBGKY hierarchy, integral form})$, we write $(\ref{equ:h-nls hierarchy,differential form})$ in integral form
\begin{align}\label{equ:H-NLS hierarchy}
\ga_{H,\hbar}^{(k)}(t_{k})=U_{\hbar}^{(k)}(t_{k})\ga_{H,\hbar}^{(k)}(0)+
\int_{0}^{t_{k}}U_{\hbar}^{(k)}(t_{k}-t_{k+1})B_{N,\hbar}^{(k+1)}\ga_{H,\hbar}^{(k+1)}(t_{k+1})dt_{k+1}.
\end{align}

The difference $w_{N,\hbar}^{(k)}=\ga_{N,\hbar}^{(k)}-\ga_{H,\hbar}^{(k)}$ solves the hierarchy
\begin{align}\label{equ:difference hierarchy}
w_{N,\hbar}^{(k)}(t_{k})=&U_{\hbar}^{(k)}(t_{k})w_{N,\hbar}^{(k)}(0)+\int_{0}^{t_{k}}
U_{\hbar}^{(k)}(t_{k}-t_{k+1})V_{N,\hbar}^{(k)}\ga_{N,\hbar}^{(k)}(t_{k+1})dt_{k+1}\\
&-\frac{k}{N}\int_{0}^{t_{k}}U_{\hbar}^{(k)}(t_{k}-t_{k+1})B_{N,\hbar}^{(k+1)}\ga_{N,\hbar}^{(k+1)}(t_{k+1})dt_{k+1}\notag\\
&+\int_{0}^{t_{k}}U_{\hbar}^{(k)}(t_{k}-t_{k+1})B_{N,\hbar}^{(k+1)}w_{N,\hbar}^{(k+1)}(t_{k+1})dt_{k+1}.\notag
\end{align}

Iterating hierarchy $(\ref{equ:difference hierarchy})$ $l_{c}$ times\footnote{$l_{c}$ means ``coupling level".} at the last term of $(\ref{equ:difference hierarchy})$, we have
\begin{align}\label{equ:the difference hierarchy, decomposition}
w_{N,\hbar}^{(k)}(t_{k})=&FP^{(k,l_{c})}(t_{k})+DP^{(k,l_{c})}(t_{k})+EP^{(k,l_{c})}(t_{k})+IP^{(k,l_{c})}(t_{k})
\end{align}
where we have grouped the terms in $w_{N,\hbar}^{(k)}(t_{k})$ into four parts: the
 free/ driving/ error/ interaction parts. We remark that $(\ref{equ:the difference hierarchy, decomposition})$ holds for all $l_{c}\geq 1$ and we will select $l_{c}$
 depending on what aspect of $w_{N,\hbar}^{(k)}$ we need in Section
 $\ref{section:A Klainerman-Machedon Bound 1st}-\ref{section:Convergence Rate for Every Finite Time}$. To write out the four parts of $w_{N,\hbar}^{(k)}$, we define the notation that, for $j\geq 1$,
\begin{align}
&J_{N,\hbar}^{(k,j)}(t_{k},\underline{t}_{(k,j)})(f^{(k+j)}(t_{k+j}))\\
=&\lrs{U_{\hbar}^{(k)}(t_{k}-t_{k+1})B_{N,\hbar}^{(k+1)}}\ccc\lrs{U_{\hbar}^{(k+j-1)}(t_{k+j-1}-t_{k+j})B_{N,\hbar}^{(k+j)}}
f^{(k+j)}(t_{k+j}),\notag
\end{align}
and $J_{N,\hbar}^{(k,0)}(t_{k},\underline{t}_{(k,0)})=f^{(k)}(t_{k})$, where $\underline{t}_{(k,j)}=(t_{k+1},...,t_{k+j})$ for $j\geq 1$. In this notation, the free part of $w_{N,\hbar}^{(k)}$ at $l_{c}$ coupling level is
\begin{align}
FP^{(k,l_{c})}(t_{k})=&U_{\hbar}^{(k)}(t_{k})w_{N,\hbar}^{(k)}(0)+\\
&\sum_{j=1}^{l_{c}}\int_{0}^{t_{k}}\ccc\int_{0}^{t_{k+j-1}}U_{\hbar}^{(k)}(t_{k}-t_{k+1})
B_{N,\hbar}^{(k+1)}\ccc \notag\\
&\times U_{\hbar}^{(k+j-1)}(t_{k+j-1}-t_{k+j})B_{N,\hbar}^{(k+j)}\lrs{U_{\hbar}^{(k+j)}(t_{k+j})w_{N,\hbar}^{(k+j)}(0)}
d\underline{t}_{(k,j)}\notag\\
=&\sum_{j=0}^{l_{c}}\int_{0}^{t_{k}}\ccc \int_{0}^{t_{k+j-1}}
J_{N,\hbar}^{(k,j)}(t_{k},\underline{t}_{(k,j)})(f_{FP}^{(k,j)}(t_{k+j}))d\underline{t}_{(k,j)}\notag
\end{align}
where in the $j=0$ case, it is meant that there are no time integrals and $J_{N,\hbar}^{(k,0)}$ is the identity operator, and
\begin{align}\label{equ:fp for k,j}
f_{FP}^{(k,j)}(t_{k+j})=U_{\hbar}^{(k+j)}(t_{k+j})w_{N,\hbar}^{(k+j)}(0).
\end{align}
The driving part is given by
\begin{align}
DP^{(k,l_{c})}(t_{k})
=&\int_{0}^{t_{k}}U_{\hbar}^{(k)}(t_{k}-t_{k+1})V_{N,\hbar}^{(k)}\ga_{N,\hbar}^{(k)}(t_{k+1})dt_{k+1}+ \\
&\sum_{j=1}^{l_{c}}\int_{0}^{t_{k}}\ccc\int_{0}^{t_{k+j-1}}U_{\hbar}^{(k)}(t_{k}-t_{k+1})
B_{N,\hbar}^{(k+1)}\ccc U_{\hbar}^{(k+j-1)}(t_{k+j-1}-t_{k+j})B_{N,\hbar}^{(k+j)}\notag\\
&\quad \lrs{\int_{0}^{t_{k+j}}U_{\hbar}^{(k+j)}(t_{k+j}-t_{k+j+1})V_{N,\hbar}^{(k+j)}\ga_{N,\hbar}^{(k+j)}(t_{k+j+1})
dt_{k+j+1}}d\underline{t}_{(k,j)}\notag\\
=&\sum_{j=0}^{l_{c}}\int_{0}^{t_{k}}\ccc \int_{0}^{t_{k+j-1}}J_{N,\hbar}^{(k,j)}(\underline{t}_{(k,j)})
\lrs{f_{DP}^{(k,j)}(t_{k+j})}d\underline{t}_{(k,j)}\notag
\end{align}
where in the $j=0$ case, it is meant that there are no time integrals and $J_{N,\hbar}^{(k,0)}$ is the identity operator, and
\begin{align} \label{equ:dp for k,j}
f_{DP}^{(k,j)}(t_{k+j})=\int_{0}^{t_{k+j}}U_{\hbar}^{(k+j)}(t_{k+j}-t_{k+j+1})V_{N,\hbar}^{(k+j)}
\ga_{N,\hbar}^{(k+j)}(t_{k+j+1})dt_{k+j+1}.
\end{align}
The error part is given by
\begin{align}
&EP^{(k,l_{c})}(t_{k})\\
=&-\frac{k}{N}\int_{0}^{t_{k}}U_{\hbar}^{(k)}(t_{k}-t_{k+1})B_{N,\hbar}^{(k+1)}\ga_{N,\hbar}^{(k+1)}(t_{k+1})dt_{k+1} \notag \\
&-\sum_{j=1}^{l_{c}}\frac{k+j}{N}\int_{0}^{t_{k}}\ccc\int_{0}^{t_{k+j-1}}U_{\hbar}^{(k)}(t_{k}-t_{k+1})
B_{N,\hbar}^{(k+1)}\ccc U_{\hbar}^{(k+j-1)}(t_{k+j-1}-t_{k+j})B_{N,\hbar}^{(k+j)}\notag\\
&\times\lrs{\int_{0}^{t_{k+j}}U_{\hbar}^{(k+j)}(t_{k+j}-t_{k+j+1})B_{N,\hbar}^{(k+j+1)}\ga_{N,\hbar}^{(k+j+1)}(t_{k+j+1})
dt_{k+j+1}}d\underline{t}_{(k,j)}\notag\\
=&\sum_{j=1}^{l_{c}+1}\int_{0}^{t_{k}}\ccc \int_{0}^{t_{k+j-1}}J_{N,\hbar}^{(k,j)}(\underline{t}_{(k,j)})
\lrs{f_{EP}^{(k,j)}(t_{k+j})}d\underline{t}_{(k,j)}\notag
\end{align}
where in the $j=0$ case, it is meant that there are no time integrals and $J_{N,\hbar}^{(k,0)}$ is the identity operator, and
\begin{align}\label{equ:ep for k,j}
f_{EP}^{(k,j)}(t_{k+j})=-\frac{k+j-1}{N}
\ga_{N,\hbar}^{(k+j)}.
\end{align}
The interaction part is given by
\begin{align}
IP^{(k,l_{c})}(t_{k})=&\int_{0}^{t_{k}}\ccc \int_{0}^{t_{k+l_{c}}}
U_{\hbar}^{(k)}(t_{k}-t_{k+1})B_{N,\hbar}^{(k+1)}\ccc \\
&\ccc U^{(k+l_{c})}(t_{k+l_{c}}-t_{k+l_{c}+1})B_{N,\hbar}^{(k+l_{c}+1)}
\lrs{w_{N,\hbar}^{(k+l_{c}+1)}(t_{k+l_{c}+1})}dt_{k+1}\ccc dt_{k+l_{c}+1}\notag\\
=&\int_{0}^{t_{k}}\ccc \int_{0}^{t_{k+l_{c}}}J_{N,\hbar}^{(k,l_{c}+1)}(t_{k},\underline{t}_{(k,l_{c}+1)})
\lrs{w_{N,\hbar}^{(k+l_{c}+1)}(t_{k+l_{c}+1})}d\underline{t}_{(k,l_{c}+1)}\notag
\end{align}
where
\begin{align}\label{equ:ip for k,j}
f_{IP}^{(k,l_{c}+1)}=w_{N,\hbar}^{(k+l_{c}+1)}(t_{k+l_{c}+1}).
\end{align}

There are around $\frac{(k+l_{c})!}{k!}$ many summands in each part. They can be grouped together by using the KM board game argument \cite{KM08}, which is below.
\begin{lemma}[\cite{KM08}, Lemma 2.1]\footnote{More advanced version of this combinatoric is now available, see \cite{CH20,CSZ21}.} \label{lemma:KM combinatorics bound}
For $j\geq 1$, one can express
\begin{align}
&\int_{0}^{t_{k}}\ccc\int_{0}^{t_{k+j-1}}J_{N,\hbar}^{(k,j)}(t_{k},\underline{t}_{(k,j)})(f^{(k+j)})
d\underline{t}_{(k,j)}\\
=&\sum_{m}\int_{D}J_{N,\hbar}^{(k,j)}(t_{k},\underline{t}_{(k,j)},\mu_{m})(f^{(k+j)})d\underline{t}_{(k,j)}.\notag
\end{align}
Here $D\subset [0,t_{k}]^{j}$, $\mu_{m}$ are a set of maps from $\lr{k+1,...,k+j}$ to
$\lr{1,...,k+j-1}$ and $\mu_{m}(l)<l$ for all $l$, and
\begin{align}
&J_{N,\hbar}^{(k,j)}(t_{k},\underline{t}_{(k,j)},\mu_{m})(f^{(k+j)})\\
=&\lrs{U_{\hbar}^{(k)}(t_{k}-t_{k+1})B_{N,\hbar,\mu_{m}(k+1),k+1}}\ccc
\lrs{U_{\hbar}^{(k+j-1)}(t_{k+j-1}-t_{k+j})B_{N,\hbar,\mu_{m}(k+j),k+j}}
f^{(k+j)}(t_{k+j})\notag
\end{align}
The summing number can be controlled by $2^{k+2j-2}$.
\end{lemma}

Then we are able to estimate $J_{N,\hbar}^{(k,j)}(t_{k},\underline{t}_{(k,j)})
(f^{(k+j)})$ via collapsing estimates in Lemma $\ref{lemma:collapsing estimate d=3, propogator version}$.
\begin{lemma}\label{lemma:iteration estimate}
Let $d=3$ and $\al=d+1/2$.
For $j\geq 1$,
\begin{align}\label{equ:iteration estimate, time L infty}
&\bbn{\int_{0}^{t_{k}}\ccc \int_{0}^{t_{k+j-1}}S_{\hbar}^{(1,k)}J_{N,\hbar}^{(k,j)}(t_{k},\underline{t}_{(k,j)})
(f^{(k+j)})d\underline{t}_{(k,j)}}_{L_{t_{k}}^{\wq}[0,T]L_{x,x'}^{2}}\\
\leq& 2^{k}4^{j}\lrs{C_{V}\hbar^{-\al}T^{1/2}}^{j-1}\int_{[0,T]}
\bbn{
S_{\hbar}^{(1,k+j-1)}B_{N,\hbar,1,k+j}(f^{(k+j)}(t_{k+j}))
}_{L_{x,x'}^{2}}dt_{k+j}\notag
\end{align}

\begin{align}\label{equ:iteration estimate, time L1}
&\bbn{S_{\hbar}^{(1,k)}B_{N,\hbar,1,k+1}\int_{0}^{t_{k+1}}\ccc \int_{0}^{t_{k+j}}J_{N,\hbar}^{(k+1,j)}(t_{k+1},\underline{t}_{(k+1,j)})
(f^{(k+j+1)})d\underline{t}_{(k+1,j)}}_{L_{t_{k+1}}^{1}[0,T]L_{x,x'}^{2}}\\
\leq&2^{k+1}4^{j} (C_{V}\hbar^{-\al}T^{1/2})^{j}
\int_{[0,T]}
\bbn{
S_{\hbar}^{(1,k+j)}B_{N,\hbar,1,k+j+1}(f^{(k+j+1)}(t_{k+j+1}))
}_{L_{x,x'}^{2}}dt_{k+j+1}.\notag
\end{align}

\end{lemma}
\begin{proof}
This is well-known for $\hbar=1$. We include a proof for completeness.
For $(\ref{equ:iteration estimate, time L infty})$, we start by using Lemma \ref{lemma:KM combinatorics bound},
\begin{align}
&\bbn{\int_{0}^{t_{k}}\ccc \int_{0}^{t_{k+j-1}}S_{\hbar}^{(1,k)}J_{N,\hbar}^{(k,j)}(t_{k},\underline{t}_{(k,j)})
(f^{(k+j)})d\underline{t}_{(k,j)}}_{L_{t_{k}}^{\wq}[0,T]L_{x,x'}^{2}}\\
\leq &2^{k}4^{j}\bbn{\int_{D}S_{\hbar}^{(1,k)}J_{N,\hbar}^{(k,j)}(t_{k},\underline{t}_{(k,j)},\mu_{m})
(f^{(k+j)})d\underline{t}_{(k,j)}}_{L_{t_{k}}^{\wq}[0,T]L_{x,x'}^{2}}\notag\\
\leq &2^{k}4^{j}\int_{[0,T]^{j}}\bbn{S_{\hbar}^{(1,k)}J_{N,\hbar}^{(k,j)}
(t_{k},\underline{t}_{(k,j)},\mu_{m})
(f^{(k+j)})}_{L_{x,x'}^{2}}d\underline{t}_{(k,j)}\notag
\end{align}
Cauchy-Schwarz at $dt_{k+1}$,
\begin{align}\label{equ:interation estimate, starting point}
\leq 2^{k}4^{j}T^{1/2}\int_{[0,T]^{j-1}}\bbn{
S_{\hbar}^{(1,k)}B_{N,\hbar,\mu_{m}(k+1),k+1}U_{\hbar}^{(k+1)}(t_{k+1}-t_{k+2})...
}_{L_{t_{k+1}}^{2}[0,T]L_{x,x'}^{2}}d\underline{t}_{(k+1,j-1)}
\end{align}
By Lemma \ref{lemma:collapsing estimate d=3, propogator version},
\begin{align}
\leq 2^{k}4^{j}C_{V}\hbar^{-\al}T^{1/2}\int_{[0,T]^{j-1}}
\bbn{
S_{\hbar}^{(1,k+1)}B_{N,\hbar,\mu_{m}(k+2),k+2}U_{\hbar}^{(k+1)}(t_{k+2}-t_{k+3})...
}_{L_{x,x'}^{2}}d\underline{t}_{(k+1,j-1)}
\end{align}
Repeating such a process gives
\begin{align}
\leq 2^{k}4^{j}(C_{V}\hbar^{-\al}T^{1/2})^{j-1}\int_{[0,T]}
\bbn{
S_{\hbar}^{(1,k+j-1)}B_{N,\hbar,\mu_{m}(k+j),k+j}(f^{(k+j)}(t_{k+j}))
}_{L_{x,x'}^{2}}dt_{k+j}
\end{align}
By symmetry,
\begin{align}
=2^{k}4^{j}(C_{V}\hbar^{-\al}T^{1/2})^{j-1}\int_{[0,T]}
\bbn{
S_{\hbar}^{(1,k+j-1)}B_{N,\hbar,1,k+j}(f^{(k+j)}(t_{k+j}))
}_{L_{x,x'}^{2}}dt_{k+j}
\end{align}

For $(\ref{equ:iteration estimate, time L1})$, we apply Lemma \ref{lemma:KM combinatorics bound} again to obtain
\begin{align}
&\bbn{\int_{0}^{t_{k+1}}\ccc \int_{0}^{t_{k+j}}S_{\hbar}^{(1,k)}B_{N,\hbar,1,k+1}J_{N,\hbar}^{(k+1,j)}(t_{k+1},\underline{t}_{(k+1,j)})
(f^{(k+j+1)})d\underline{t}_{(k+1,j)}}_{L_{t_{k+1}}^{1}[0,T]L_{x,x'}^{2}}\\
\leq&2^{k+1}4^{j}\bbn{\int_{D}S_{\hbar}^{(1,k)}B_{N,\hbar,1,k+1}J_{N,\hbar}^{(k+1,j)}
(t_{k+1},\underline{t}_{(k+1,j)},\mu_{m})
(f^{(k+j+1)})d\underline{t}_{(k+1,j)}}_{L_{t_{k+1}}^{1}[0,T]L_{x,x'}^{2}}\notag\\
\leq&2^{k+1}4^{j}\int_{[0,T]^{j}}\bbn{S_{\hbar}^{(1,k)}B_{N,\hbar,1,k+1}J_{N,\hbar}^{(k+1,j)}
(t_{k+1},\underline{t}_{(k+1,j)},\mu_{m})
(f^{(k+j+1)})}_{L_{t_{k+1}}^{1}[0,T]L_{x,x'}^{2}}d\underline{t}_{(k+1,j)}\notag
\end{align}
Cauchy-Schwarz at $dt_{k+1}$,
\begin{align}
\leq &2^{k+1}4^{j}T^{1/2}\int_{[0,T]^{j}}\bbn{S_{\hbar}^{(1,k)}B_{N,\hbar,1,k+1}
U_{\hbar}^{(k+1)}(t_{k+1}-t_{k+2})
\ccc}_{L_{t_{k+1}}^{2}[0,T]L_{x,x'}^{2}}d\underline{t}_{(k+1,j)}\notag
\end{align}
Iterating the same process as $(\ref{equ:interation estimate, starting point})$, we obtain
\begin{align}
\leq&2^{k+1}4^{j} (C_{V}\hbar^{-\al}T^{1/2})^{j}
\int_{[0,T]}
\bbn{
S_{\hbar}^{(1,k+j)}B_{N,\hbar,1,k+j+1}(f^{(k+j+1)}(t_{k+j+1}))
}_{L_{x,x'}^{2}}dt_{k+j+1}.
\end{align}
\end{proof}

Away from Lemma $\ref{lemma:iteration estimate}$, we obtain below crude estimates of the driving part, error part and the interaction part.
\begin{lemma}\label{lemma:collapsing estimate for dp,ep,ip,last step}
Let $k\leq (\ln N)^{10}$ and $j\leq (\ln N)^{10}$. For the driving part, we have
\begin{align}\label{equ:collapsing estimate for dp,last step,j=0}
&\n{S_{\hbar}^{(1,k)}f_{DP}^{(k,0)}
(t_{k})}_{L_{t_{k}}^{\wq}[0,T]L_{x,x'}^{2}}
\leq N^{\frac{5}{2}\be-1}(C_{V}\hbar^{-\al}T^{1/2})k^{2}(2E_{0,\hbar})^{k}
\end{align}
and
\begin{align}\label{equ:collapsing estimate for dp,last step,j>0}
&\int_{[0,T]}
\bbn{
S_{\hbar}^{(1,k+j-1)}B_{N,\hbar,1,k+j}(f_{DP}^{(k,j)}(t_{k+j}))
}_{L_{x,x'}^{2}}dt_{k+j}\\
\leq&N^{\frac{5}{2}\be-1}\lrs{C_{V}\hbar^{-\al}T^{1/2}}^{2}(k+j)^{2}\lrs{2E_{0,\hbar}}^{k+j}.\notag
\end{align}

For the error part, we have
\begin{align}\label{equ:collapsing estimate for ep,last step,j>0}
&\int_{[0,T]}
\bbn{
S_{\hbar}^{(1,k+j-1)}B_{N,\hbar,1,k+j}(f_{EP}^{(k,j)}(t_{k+j}))
}_{L_{x,x'}^{2}}dt_{k+j}\\
\leq &N^{\frac{5}{2}\be-1}\lrs{C_{V}\hbar^{-\al}T^{1/2}}(k+j)\lrs{2E_{0,\hbar}}^{k+j}.\notag
\end{align}

For the interaction part, we have
\begin{align}\label{equ:collapsing estimate for ip,last step,j>0}
&\int_{[0,T]}
\bbn{
S_{\hbar}^{(1,k+j-1)}B_{N,\hbar,1,k+j}(f_{IP}^{(k,j)}(t_{k+j}))
}_{L_{x,x'}^{2}}dt_{k+j}\\
\leq &N^{\frac{5}{2}\be}\lrs{C_{V}\hbar^{-\al}T^{1/2}}\lrs{4E_{0,\hbar}}^{k+j}.\notag
\end{align}
\end{lemma}
\begin{proof}
For $(\ref{equ:collapsing estimate for dp,last step,j=0})$, plugging in $f_{DP}^{(k,0)}$, we need to estimate
\begin{align}
\bbn{S_{\hbar}^{(1,k)}\int_{0}^{t_{k}}U_{\hbar}^{(k)}(t_{k}-t_{k+1})V_{N,\hbar}^{(k)}
\ga_{N,\hbar}^{(k)}(t_{k+1})dt_{k+1}}_{L_{t_{k}}^{\wq}[0,T]L_{x,x'}^{2}}
\end{align}
By $(\ref{equ:collapsing estimate for dp,last step,scaling,j=0})$ in Lemma \ref{lemma:collapsing estimate for dp,last step,scaling},
\begin{align*}
\leq N^{\frac{5}{2}\be-1}\hbar\lrs{C_{V}\hbar^{-\al}T^{1/2}}k^{2}\n{S_{\hbar}^{(1,k)}
\ga_{N,\hbar}^{(k)}(t_{k+1})}_{L_{t_{k+1}}^{\infty}L_{x,x'}^{2}}.
\end{align*}
Using the $N$-body energy bound $(\ref{equ:n-body energy bound})$ and discarding the unimportant factor\footnote{Keeping this $\hbar$
does not give much better estimate in the end. In fact, as we
will see, $\hbar^{-\al}$ accumulates but this $\hbar$ stays as just one
factor.} $\hbar$, we arrive at
\begin{align}
&\bbn{S_{\hbar}^{(1,k)}\int_{0}^{t_{k}}U_{\hbar}^{(k)}(t_{k}-t_{k+1})V_{N,\hbar}^{(k)}
\ga_{N,\hbar}^{(k)}(t_{k+1})dt_{k+1}}_{L_{t_{k}}^{\wq}[0,T]L_{x,x'}^{2}}\\
\leq& N^{\frac{5}{2}\be-1}(C_{V}\hbar^{-\al}T^{1/2})k^{2}(2E_{0,\hbar})^{k},\notag
\end{align}
which completes the proof of $(\ref{equ:collapsing estimate for dp,last step,j=0})$.

For $(\ref{equ:collapsing estimate for dp,last step,j>0})$, we insert $f_{DP}^{(k,j)}$ defined in $(\ref{equ:dp for k,j})$ to otain
\begin{align}
&\int_{[0,T]}
\bbn{
S_{\hbar}^{(1,k+j-1)}B_{N,\hbar,1,k+j}(f_{DP}^{(k,j)}(t_{k+j}))
}_{L_{x,x'}^{2}}dt_{k+j}\\
=&\bbn{S_{\hbar}^{(1,k+j-1)}B_{N,\hbar,1,k+j}\int_{0}^{t_{k+j}}
U_{\hbar}^{(k+j)}(t_{k+j}-t_{k+j+1})V_{N,\hbar}^{(k+j)}
\ga_{N,\hbar}^{(k+j)}(t_{k+j+1})dt_{k+j+1}}_{L_{t_{k+j}}^{1}[0,T]L_{x,x'}^{2}}\notag
\end{align}
Utilizing $(\ref{equ:collapsing estimate for dp,last step,scaling,j>=0})$ in Lemma \ref{lemma:collapsing estimate for dp,last step,scaling},
\begin{align*}
\leq N^{\frac{5}{2}\be-1}\hbar(C_{V}\hbar^{-\al}T^{1/2})^{2}(k+j)^{2}\bbn{S_{\hbar}^{(1,k+j)}
\ga_{N,\hbar}^{(k+j)}(t_{k+j+1})}_{L_{t_{k+j+1}}^{\infty}L_{x,x'}^{2}}
\end{align*}
Making use of the $N$-body energy bound $(\ref{equ:n-body energy bound})$ and discarding the unimportant small factor $\hbar$, $(\ref{equ:collapsing estimate for dp,last step,j>0})$ is then proved.

For the error part $(\ref{equ:collapsing estimate for ep,last step,j>0})$, inserting $f_{EP}^{(k,j)}$ we have
\begin{align}
&\int_{[0,T]}
\bbn{
S_{\hbar}^{(1,k+j-1)}B_{N,\hbar,1,k+j}(f_{EP}^{(k,j)}(t_{k+j}))
}_{L_{x,x'}^{2}}dt_{k+j}\\
=&\frac{k+j-1}{N}
\int_{[0,T]}
\bbn{
S_{\hbar}^{(1,k+j-1)}B_{N,\hbar,1,k+j}(\ga_{N,\hbar}^{(k+j)}(t_{k+j}))
}_{L_{x,x'}^{2}}dt_{k+j}\notag
\end{align}
By $(\ref{equ:collapsing estimate for ep,ip,last step})$ in Lemma \ref{lemma:collapsing estimate for ep,ip,last step} and the $N$-body energy bound $(\ref{equ:n-body energy bound})$,
\begin{align*}
\leq &N^{\frac{5}{2}\be-1}\hbar^{2}T^{1/2}\lrs{C_{V}\hbar^{-\al}T^{1/2}}(k+j)\lrs{2E_{0,\hbar}}^{k+j}.\notag
\end{align*}
Discarding the unimportant small factor $\hbar^{2}T^{1/2}$, we complete the proof of $(\ref{equ:collapsing estimate for ep,last step,j>0})$.

For the interaction part $(\ref{equ:collapsing estimate for ip,last step,j>0})$, inserting $f_{IP}^{(k,j)}$ we have
\begin{align}
&\int_{[0,T]}
\bbn{
S_{\hbar}^{(1,k+j-1)}B_{N,\hbar,1,k+j}(f_{IP}^{(k,j)}(t_{k+j}))
}_{L_{x,x'}^{2}}dt_{k+j}\\
\leq &
\int_{[0,T]}
\bbn{
S_{\hbar}^{(1,k+j-1)}B_{N,\hbar,1,k+j}(\ga_{N,\hbar}^{(k+j)}(t_{k+j}))
}_{L_{x,x'}^{2}}dt_{k+j}\notag\\
&+\int_{[0,T]}
\bbn{
S_{\hbar}^{(1,k+j-1)}B_{N,\hbar,1,k+j}(\ga_{H,\hbar}^{(k+j)}(t_{k+j}))
}_{L_{x,x'}^{2}}dt_{k+j}.\notag
\end{align}
By $(\ref{equ:collapsing estimate for ep,ip,last step})$ in Lemma \ref{lemma:collapsing estimate for ep,ip,last step} and the $N$-body energy bound $(\ref{equ:n-body energy bound})$,
\begin{align*}
\leq &N^{\frac{5}{2}\be}\hbar^{2}T^{1/2}\lrs{C_{V}\hbar^{-\al}T^{1/2}}\lrs{4E_{0,\hbar}}^{k+j}.\notag
\end{align*}
By discarding the unimportant small factor $\hbar^{2}T^{1/2}$, we complete the proof of $(\ref{equ:collapsing estimate for ip,last step,j>0})$.
\end{proof}

\subsection{A Klainerman-Machedon Bound 1st}\label{section:A Klainerman-Machedon Bound 1st}
Via the preliminary estimates in Section \ref{section:A Tool Box of Space-time Estimates}, we are able to provide a ``preliminary" Klainerman-Machedon bound for $w_{N,\hbar}^{(k)}$. Here,
``preliminary" certainly means, ``not final" as we will improve it once we have used it
to prove $(\ref{equ:decay estimate, uniform in h,bbgky-hartree})$.
\begin{lemma}\label{lemma:bound estimate for ip}
Let $t_{0}\in[0,\infty)$, $T\leq \frac{\hbar^{2\al}}{\lrs{64E_{0,\hbar}C_{V}e}^{2}}$, and
$\al=d+\frac{1}{2}$. For $k\leq (\ln N)^{10}$, we have
\begin{align}\label{equ:KM bound estimate for ip}
\int_{[t_{0},t_{0}+T]}
\bbn{
S_{\hbar}^{(1,k)}B_{N,\hbar,1,k+1}w_{N,\hbar}^{(k+1)}(t_{k+1})
}_{L_{x,x'}^{2}}dt_{k+1}\leq (16E_{0,\hbar})^{k}.
\end{align}
It holds for sufficiently small $T$ but independent of the initial time.
\end{lemma}

\begin{proof}
We give a proof following the method in \cite{Che13,CH16correlation,CH16on} which was inspired by \cite{CP14}. We might as well take $t_{0}=0$ for convenience, as the general case also holds from time translation.
 Decompose $w_{N,\hbar}^{(k)}$ as in $(\ref{equ:the difference hierarchy, decomposition})$, it suffices to prove that
\begin{align}
&\int_{[0,T]}
\bbn{
S_{\hbar}^{(1,k)}B_{N,\hbar,1,k+1}FP^{(k+1,l_{c})}(t_{k+1})
}_{L_{x,x'}^{2}}dt_{k+1}\leq (8E_{0,\hbar})^{k},\label{equ:KM bound estimate,fp}\\
&\int_{[0,T]}
\bbn{
S_{\hbar}^{(1,k)}B_{N,\hbar,1,k+1}DP^{(k+1,l_{c})}(t_{k+1})
}_{L_{x,x'}^{2}}dt_{k+1}\leq (8E_{0,\hbar})^{k},\label{equ:KM bound estimate,dp}\\
&\int_{[0,T]}
\bbn{
S_{\hbar}^{(1,k)}B_{N,\hbar,1,k+1}EP^{(k+1,l_{c})}(t_{k+1})
}_{L_{x,x'}^{2}}dt_{k+1}\leq (8E_{0,\hbar})^{k},\label{equ:KM bound estimate,ep}\\
&\int_{[0,T]}
\bbn{
S_{\hbar}^{(1,k)}B_{N,\hbar,1,k+1}IP^{(k+1,l_{c})}(t_{k+1})
}_{L_{x,x'}^{2}}dt_{k+1}\leq (8E_{0,\hbar})^{k}.\label{equ:KM bound estimate,ip}
\end{align}

For the FP part $(\ref{equ:KM bound estimate,fp})$, we start by using estimate $(\ref{equ:iteration estimate, time L1})$ in Lemma \ref{lemma:iteration estimate} to obtain
\begin{align*}
&\int_{[0,T]}
\bbn{
S_{\hbar}^{(1,k)}B_{N,\hbar,1,k+1}FP^{(k+1,l_{c})}(t_{k+1})
}_{L_{x,x'}^{2}}dt_{k+1}\\
\leq & \int_{[0,T]}\n{S_{\hbar}^{(1,k)}B_{N,\hbar,1,k+1}f_{FP}^{(k+1,0)}(t_{k+1})}_{L_{x,x'}^{2}}dt_{k+1}\\
&+\sum_{j=1}^{l_{c}}2^{k+1}4^{j}\lrs{C_{V}\hbar^{-\al}T^{1/2}}^{j}\int_{[0,T]}
\bbn{
S_{\hbar}^{(1,k+j)}B_{N,\hbar,1,k+j+1}(f_{FP}^{(k+1,j)}(t_{k+j+1}))
}_{L_{x,x'}^{2}}dt_{k+j+1}
\end{align*}
Plugging in $f_{FP}^{(k+j)}$, applying Cauchy-Schwarz at $dt_{k+j+1}$ and then Lemma \ref{lemma:collapsing estimate d=3, propogator version},
\begin{align*}
\leq 2^{k+1}\sum_{j=0}^{l_{c}}\lrs{4C_{V}\hbar^{-\al}T^{1/2}}^{j+1}
\n{S_{\hbar}^{(1,k+j+1)}w_{N,\hbar}^{(k+j+1)}(0)}_{L_{x,x'}^{2}}
\end{align*}
We have required that $l_{c}\leq \ln N$ thus we can use the $N$-body energy bound $(\ref{equ:n-body energy bound})$ to obtain
\begin{align*}
\leq (8E_{0,\hbar})^{k}\sum_{j=0}^{l_{c}}\lrs{16E_{0,\hbar}C_{V}\hbar^{-\al}T^{1/2}}^{j+1}\leq (8E_{0,\hbar})^{k}
\end{align*}
if we plug in $T\leq \frac{\hbar^{2\al}}{\lrs{64E_{0,\hbar}C_{V}e}^{2}}$.

For the DP Part $(\ref{equ:KM bound estimate,dp})$, the above process gives
\begin{align*}
&\int_{[0,T]}\bbn{
S_{\hbar}^{(1,k)}B_{N,\hbar,1,k+1}DP^{(k+1,l_{c})}(t_{k+1})
}_{L_{x,x'}^{2}}dt_{k+1}\\
\leq& \int_{[0,T]}\bbn{
S_{\hbar}^{(1,k)}B_{N,\hbar,1,k+1}f_{DP}^{(k+1,0)}(t_{k+1})
}_{L_{x,x'}^{2}}dt_{k+1}\\
&+2^{k+1}\sum_{j=1}^{l_{c}}4^{j} (C_{V}\hbar^{-\al}T^{1/2})^{j}
\int_{[0,T]}
\bbn{
S_{\hbar}^{(1,k+j)}B_{N,\hbar,1,k+j+1}(f_{DP}^{(k+j+1)}(t_{k+j+1}))
}_{L_{x,x'}^{2}}dt_{k+j+1}
\end{align*}
As $k\leq (\ln N)^{10}$ and $j\leq l_{c}\leq \ln N$, we can use estimate $(\ref{equ:collapsing estimate for dp,last step,j>0})$ in Lemma \ref{lemma:collapsing estimate for dp,ep,ip,last step} to get
\begin{align*}
\leq &N^{\frac{5}{2}\be-1}2^{k+1}\sum_{j=0}^{l_{c}}\lrs{4C_{V}\hbar^{-\al}T^{1/2}}^{j+2}
(k+j+1)^{2}\lrs{2E_{0,\hbar}}^{k+j+1}\\
 \leq& N^{\frac{5}{2}\be-1}(8E_{0,\hbar})^{k}
 \sum_{j=0}^{l_{c}}(16E_{0,\hbar}C_{V}\hbar^{-\al}T^{1/2})^{j+2}\\
 \leq& (8E_{0,\hbar})^{k}
\end{align*}
if we plug in $T\leq \frac{\hbar^{2\al}}{\lrs{64E_{0,\hbar}C_{V}e}^{2}}$.

Similarly, for the error part $(\ref{equ:KM bound estimate,ep})$, we have
\begin{align*}
&\int_{[0,T]}\bbn{
S_{\hbar}^{(1,k)}B_{N,\hbar,1,k+1}EP^{(k+1,l_{c})}(t_{k+1})
}_{L_{x,x'}^{2}}dt_{k+1}\\
\leq&2^{k+1}\sum_{j=1}^{l_{c}+1}4^{j}(C_{V}\hbar^{-\al}T^{1/2})^{j}
\int_{[0,T]}
\bbn{
S_{\hbar}^{(1,k+j)}B_{N,\hbar,1,k+j+1}(f_{EP}^{(k+j+1)}(t_{k+j+1}))
}_{L_{x,x'}^{2}}dt_{k+j+1}
\end{align*}
Plugging in $f_{EP}^{(k,j)}$ and using estimate $(\ref{equ:collapsing estimate for ep,last step,j>0})$ in Lemma \ref{lemma:collapsing estimate for dp,ep,ip,last step},
\begin{align*}
\leq& N^{\frac{5}{2}\be-1}2^{k+1}\sum_{j=1}^{l_{c}+1}\lrs{4C_{V}\hbar^{-\al}T^{1/2}}^{j+1}(k+j+1)\lrs{2E_{0,\hbar}}^{k+j+1}\\
\leq& N^{\frac{5}{2}\be-1}(8E_{0,\hbar})^{k}\sum_{j=1}^{l_{c}+1}
\lrs{16E_{0,\hbar}C_{V}\hbar^{-\al}T^{1/2}}^{j+1}\\
\leq& (8E_{0,\hbar})^{k}
\end{align*}
if we plug in $T\leq \frac{\hbar^{2\al}}{\lrs{64E_{0,\hbar}C_{V}e}^{2}}$.

Finally, for the interaction part $(\ref{equ:KM bound estimate,ip})$, we have
\begin{align*}
&\int_{[0,T]}\bbn{
S_{\hbar}^{(1,k)}B_{N,\hbar,1,k+1}IP^{(k+1,l_{c})}(t_{k+1})
}_{L_{x,x'}^{2}}dt_{k+1}\\
\leq&2^{k+1}4^{l_{c}+1} (C_{V}\hbar^{-\al}T^{1/2})^{l_{c}+1}
\int_{[0,T]}
\bbn{
S_{\hbar}^{(1,k+l_{c}+1)}B_{N,\hbar,1,k+l_{c}+2}w_{N,\hbar}^{(k+l_{c}+2)}(t_{k+l_{c}+2})
}_{L_{x,x'}^{2}}dt_{k+l_{c}+2}
\end{align*}
By estimate $(\ref{equ:collapsing estimate for ip,last step,j>0})$ in Lemma \ref{lemma:collapsing estimate for dp,ep,ip,last step},
\begin{align*}
\leq&N^{\frac{5}{2}\be}2^{k+1} (4C_{V}\hbar^{-\al}T^{1/2})^{l_{c}+2}(4E_{0,\hbar})^{k+l_{c}+2}
\end{align*}
Plugging in $T\leq \frac{\hbar^{2\al}}{\lrs{64E_{0,\hbar}C_{V}e}^{2}}$ and taking $l_{c}+1=\ln N$,
\begin{align*}
\leq& 2N^{\frac{5}{2}\be-1}(8E_{0,\hbar})^{k}
\end{align*}
and we have completed the proof of Lemma \ref{lemma:bound estimate for ip}.

\end{proof}

\subsection{Feeding the Strichartz Bound into the $H^1$ Estimate}\label{section:Feeding the Strichartz Bound into the $H^1$ Estimate}
In the section, we first provide estimates for the four parts in the expansion of
$w_{N,\hbar}^{(k)}$ via the preliminary crude estimates established in Section \ref{section:A Tool Box of Space-time Estimates}. Then with the help of the KM bound we prove in
Section \ref{section:A Klainerman-Machedon Bound 1st}, we can establish a strong stepping estimate for $w_{N,\hbar}^{(k)}$ which is Proposition $\ref{lemma:decay estimate for fp,dp,ip}$.

\begin{lemma}\label{lemma:estimate for fp,dp,ip}
Let $\al=d+1/2$.
 For $k\leq (\ln N)^{2}$ and $l_{c}\leq \ln N$, we have the following estimates.

For the free part,
\begin{align} \label{equ:estimate for fp}
\sup_{t_{k}\in[t_{0},t_{0}+T]}\n{S_{\hbar}^{(1,k)}FP^{(k,l_{c})}(t_{k})}_{L_{x,x'}^{2}}
\leq 2^{k}\sum_{j=0}^{l_{c}}\lrs{4C_{V}\hbar^{-\al}T^{1/2}}^{j}
\n{S_{\hbar}^{(1,k+j)}w_{N,\hbar}^{(k+j)}(t_{0})}_{L_{x,x'}^{2}}
\end{align}

For the driving part,
\begin{align}\label{equ:estimate for dp}
\sup_{t_{k}\in[t_{0},t_{0}+T]}\n{S_{\hbar}^{(1,k)}DP^{(k,l_{c})}(t_{k})}_{L_{x,x'}^{2}}
\leq& \lrs{8E_{0,\hbar}}^{k}N^{\frac{5}{2}\be-1}\sum_{j=0}^{l_{c}}\lrs{16E_{0,\hbar}C_{V}\hbar^{-\al}T^{1/2}}^{j+1}
\end{align}

For the error part,
\begin{align}\label{equ:estimate for ep}
\sup_{t_{k}\in[t_{0},t_{0}+T]}\n{S_{\hbar}^{(1,k)}EP^{(k,l_{c})}(t_{k})}_{L_{x,x'}^{2}}
\leq& \lrs{8E_{0,\hbar}}^{k}N^{\frac{5}{2}\be-1}\sum_{j=0}^{l_{c}}\lrs{16E_{0,\hbar}C_{V}\hbar^{-\al}T^{1/2}}^{j+1}
\end{align}

For the interaction part,
\begin{align}\label{equ:estimate for ip}
&\sup_{t_{k}\in[t_{0},t_{0}+T]}\n{S_{\hbar}^{(1,k)}IP^{(k,l_{c})}(t_{k})}_{L_{x,x'}^{2}}\\
\leq& 2^{k}4^{l_{c}+1}\lrs{C_{V}\hbar^{-\al}T^{1/2}}^{l_{c}}\int_{[t_{0},t_{0}+T]}
\bbn{
S_{\hbar}^{(1,k+l_{c})}B_{N,\hbar,1,k+l_{c}+1}w_{N,\hbar}^{(k+l_{c}+1)}(t_{k+l_{c}+1})
}_{L_{x,x'}^{2}}dt_{k+l_{c}+1}.\notag
\end{align}
\end{lemma}
\begin{proof}
For convenience, we might as well take $t_{0}=0$ as the proof works the same for general case by time translation.

For the free part, applying estimate $(\ref{equ:iteration estimate, time L infty})$ in Lemma \ref{lemma:iteration estimate}, we arrive at
\begin{align*}
&\n{S_{\hbar}^{(1,k)}FP^{(k,l_{c})}}_{L_{t_{k}}^{\wq}[0,T]L_{x,x'}^{2}}\\
\leq & \n{S_{\hbar}^{(1,k)}f_{FP}^{(k,0)}(t_{k})}_{L_{t_{k}}^{\wq}[0,T]L_{x,x'}^{2}}\\
&+\sum_{j=1}^{l_{c}}2^{k}4^{j}\lrs{C_{V}\hbar^{-\al}T^{1/2}}^{j-1}\int_{[0,T]}
\bbn{
S_{\hbar}^{(1,k+j-1)}B_{N,\hbar,1,k+j}(f_{FP}^{(k+j)}(t_{k+j}))
}_{L_{x,x'}^{2}}dt_{k+j}
\end{align*}
Plugging in $f_{FP}^{(k+j)}$ and applying Cauchy-Schwarz at $dt_{k+j}$,
\begin{align*}
\leq & \n{S_{\hbar}^{(1,k)}U_{\hbar}^{(k)}(t_{k})w_{N,\hbar}^{(k)}(0)}_{L_{t_{k}}^{\wq}[0,T]L_{x,x'}^{2}}\\
&+\sum_{j=1}^{l_{c}}2^{k}4^{j}\lrs{C_{V}\hbar^{-\al}T^{1/2}}^{j-1}T^{1/2}\bbn{
S_{\hbar}^{(1,k+j-1)}B_{N,\hbar,1,k+j}U_{\hbar}^{(k+j)}(t_{k+j})w_{N,\hbar}^{(k+j)}(0)
}_{L_{t_{k+j}}^{2}[0,T]L_{x,x'}^{2}}
\end{align*}
Applying the KM collapsing estimate (Lemma \ref{lemma:collapsing estimate d=3, propogator version}) for $j\geq 1$,
\begin{align}
\leq& \sum_{j=0}^{l_{c}}2^{k}\lrs{4C_{V}\hbar^{-\al}T^{1/2}}^{j}\bbn{
S_{\hbar}^{(1,k+j)}w_{N,\hbar}^{(k+j)}(0)
}_{L_{x,x'}^{2}}.
\end{align}
We have $(\ref{equ:estimate for fp})$ as claimed.

For the driving part, the same process yields
\begin{align*}
&\n{S_{\hbar}^{(1,k)}DP^{(k,l_{c})}}_{L_{t_{k}}^{\wq}[0,T]L_{x,x'}^{2}}\\
\leq& \n{S_{\hbar}^{(1,k)}f_{DP}^{(k,0)}
(t_{k})}_{L_{t_{k}}^{\wq}[0,T]L_{x,x'}^{2}}\\
&+ 2^{k}\sum_{j=1}^{l_{c}}
4^{j}\lrs{C_{V}\hbar^{-\al}T^{1/2}}^{j-1}
\int_{[0,T]}
\bbn{
S_{\hbar}^{(1,k+j-1)}B_{N,\hbar,1,k+j}(f_{DP}^{(k,j)}(t_{k+j}))
}_{L_{x,x'}^{2}}dt_{k+j}
\end{align*}
Plugging in $f_{DP}^{(k,j)}$ and using estimates $(\ref{equ:collapsing estimate for dp,last step,j=0})$ and
$(\ref{equ:collapsing estimate for dp,last step,j>0})$ gives
\begin{align*}
\leq &N^{\frac{5}{2}\be-1}2^{k}\sum_{j=0}^{l_{c}}(k+j)^{2}\lrs{4C_{V}\hbar^{-\al}T^{1/2}}^{j+1}(2E_{0,\hbar})^{k+j}\\
\leq &N^{\frac{5}{2}\be-1}(8E_{0,\hbar})^{k}\sum_{j=0}^{l_{c}}\lrs{16E_{0,\hbar}C_{V}\hbar^{-\al}T^{1/2}}^{j+1}
\end{align*}
which completes the proof for the driving part.

For the error part, it reads
\begin{align*}
&\n{S_{\hbar}^{(1,k)}EP^{(k,l_{c})}}_{L_{t_{k}}^{\wq}[0,T]L_{x,x'}^{2}}\\
\leq&  2^{k}\sum_{j=1}^{l_{c}+1}
4^{j}\lrs{C_{V}\hbar^{-\al}T^{1/2}}^{j-1}
\int_{[0,T]}
\bbn{
S_{\hbar}^{(1,k+j-1)}B_{N,\hbar,1,k+j}(f_{EP}^{(k,j)}(t_{k+j}))
}_{L_{x,x'}^{2}}dt_{k+j}
\end{align*}
Plugging in $f_{EP}^{(k,j)}$ and using estimate $(\ref{equ:collapsing estimate for ep,last step,j>0})$ provides
\begin{align*}
\leq& N^{\frac{5}{2}\be-1}2^{k}\sum_{j=1}^{l_{c}+1}
(k+j)\lrs{4C_{V}\hbar^{-\al}T^{1/2}}^{j}(2E_{0,\hbar})^{k+j}\\
\leq& N^{\frac{5}{2}\be-1}(8E_{0,\hbar})^{k}\sum_{j=1}^{l_{c}+1}
\lrs{16E_{0,\hbar}C_{V}\hbar^{-\al}T^{1/2}}^{j}
\end{align*}
which completes the proof for the error part.

For the interaction part, we have similarly
\begin{align*}
&\n{S_{\hbar}^{(1,k)}IP^{(k,l_{c})}}_{L_{t_{k}}^{\wq}[0,T]L_{x,x'}^{2}}\\
=&
\bbn{\int_{0}^{t_{k}}\ccc \int_{0}^{t_{k+l_{c}}}S_{\hbar}^{(1,k)}J_{N,\hbar}^{(k,l_{c}+1)}(t_{k},\underline{t}_{(k,l_{c}+1)})
\lrs{w_{N,\hbar}^{(k+l_{c}+1)}(t_{k+l_{c}+1})}d\underline{t}_{(k,l_{c}+1)}}_{L_{t_{k}}^{\wq}[0,T]L_{x,x'}^{2}}\\
\leq& 2^{k}4^{l_{c}+1}\lrs{C_{V}\hbar^{-\al}T^{1/2}}^{l_{c}}\int_{[0,T]}
\bbn{
S_{\hbar}^{(1,k+l_{c})}B_{N,\hbar,1,k+l_{c}+1}w_{N,\hbar}^{(k+l_{c}+1)}(t_{k+l_{c}+1})
}_{L_{x,x'}^{2}}dt_{k+l_{c}+1}
\end{align*}
which is $(\ref{equ:estimate for ip})$.
\end{proof}
Notice that, we are not using the crude estimates in Lemma $\ref{lemma:collapsing estimate for dp,ep,ip,last step}$ for $(\ref{equ:estimate for ip})$. We will use the KM bound we refined in Lemma $\ref{lemma:bound estimate for ip}$ to strengthen our estimate in Proposition $\ref{lemma:decay estimate for fp,dp,ip}$. Before we start, we recall that $(\ref{equ:the difference hierarchy, decomposition})$ is true for
all $l_{c}\geq 1$, hence properties regarding $w_{N,\hbar}^{(k)}$ using $l_{c}$ equal to some number $A$ can be fed into the proof of another property of $w_{N,\hbar}^{(k)}$ using $l_{c}$ equal to some number $B$.

\begin{proposition}\label{lemma:decay estimate for fp,dp,ip}
Let $T \leq \frac{\hbar^{2\al}}{\lrs{64E_{0,\hbar}C_{V}e}^{2}}$ and $\al=d+1/2$. For $k\leq (\ln N)^{2}$, $l_{c}\leq \ln N$, we have
\begin{align}\label{equ:decay estimate for fp,dp,ip}
&\sup_{t\in[t_{0},t_{0}+T]}\n{S_{\hbar}^{(1,k)}w_{N,\hbar}^{(k)}(t)}_{L_{x,x'}^{2}}\\
\leq &
2^{k}\sum_{j=0}^{l_{c}}(4C_{V}\hbar^{-\al}T^{1/2})^{j}
\n{S_{\hbar}^{(1,k+j)}w_{N,\hbar}^{(k+j)}(t_{0})}_{L_{x,x'}^{2}}+(C_{0,\hbar})^{k}N^{\frac{5}{2}\be-1}+(C_{0,\hbar})^{k}\lrs{\frac{1}{e}}^{l_{c}+1},\notag
\end{align}
and
\begin{align}\label{equ:decay estimate for collapsing part}
&\int_{[t_{0},t_{0}+T]}\n{S_{\hbar}^{(1,1)}B_{N,\hbar,1,2}^{\pm}w_{N,\hbar}^{(2)}(t)}_{L_{x,x'}^{2}}dt\\
\leq &4\sum_{j=0}^{l_{c}}\lrs{4C_{V}\hbar^{-\al}T^{1/2}}^{j+1}
\n{S_{\hbar}^{(1,2+j)}w_{N,\hbar}^{(2+j)}(t_{0})}_{L_{x,x'}^{2}}
+
(C_{0,\hbar})^{2}N^{\frac{5}{2}\be-1}+(C_{0,\hbar})^{2}\lrs{\frac{1}{e}}^{l_{c}+1},\notag
\end{align}
where $C_{0,\hbar}=64E_{0,\hbar}$. Notice that $(\ref{equ:decay estimate for collapsing part})$ is stronger than $(\ref{equ:KM bound estimate for ip})$.
\end{proposition}
\begin{proof}
The conclusion of Lemma $\ref{lemma:estimate for fp,dp,ip}$ reads
\begin{align}
&\sup_{t\in[t_{0},t_{0}+T]}\n{S_{\hbar}^{(1,k)}w_{N,\hbar}^{(k)}(t)}_{L_{x,x'}^{2}}\\
\leq &
2^{k}\sum_{j=0}^{l_{c}}(4C_{V}\hbar^{-\al}T^{1/2})^{j}
\n{S_{\hbar}^{(1,k+j)}w_{N,\hbar}^{(k+j)}(t_{0})}_{L_{x,x'}^{2}}+2N^{\frac{5}{2}\be-1}(8E_{0,\hbar})^{k}\sum_{j=0}^{l_{c}}
\lrs{C_{V}\hbar^{-\al}T^{1/2}}^{j}\notag\\
&+ 2^{k}4^{l_{c}+1}\lrs{C_{V}\hbar^{-\al}T^{1/2}}^{l_{c}}\int_{[t_{0},t_{0}+T]}
\bbn{
S_{\hbar}^{(1,k+l_{c})}B_{N,\hbar,1,k+l_{c}+1}w_{N,\hbar}^{(k+l_{c}+1)}(t_{k+l_{c}+1})
}_{L_{x,x'}^{2}}dt_{k+l_{c}+1}.\notag
\end{align}
Since $k+l_{c}\leq (\ln N)^{10}$ and $T\leq \frac{\hbar^{2\al}}{\lrs{64E_{0,\hbar}C_{V}e}^{2}}$, we can employ KM bound in Lemma \ref{lemma:bound estimate for ip} to get
\begin{align*}
\leq &
2^{k}\sum_{j=0}^{l_{c}}(4C_{V}\hbar^{-\al}T^{1/2})^{j}
\n{S_{\hbar}^{(1,k+j)}w_{N,\hbar}^{(k+j)}(t_{0})}_{L_{x,x'}^{2}}+2N^{\frac{5}{2}\be-1}(8E_{0,\hbar})^{k}\sum_{j=0}^{l_{c}}
2^{k}4^{l_{c}+1}\lrs{C_{V}\hbar^{-\al}T^{1/2}}^{l_{c}}\notag\\
&+ 2^{k}4^{l_{c}+1}\lrs{C_{V}\hbar^{-\al}T^{1/2}}^{l_{c}}
(16E_{0,\hbar})^{k+l_{c}}
\end{align*}
Plugging in $T\leq \frac{\hbar^{2\al}}{\lrs{64E_{0,\hbar}C_{V}e}^{2}}$ and $C_{0,\hbar}=64E_{0,\hbar}$, we obtain $(\ref{equ:decay estimate for fp,dp,ip})$.

For $(\ref{equ:decay estimate for collapsing part})$, repeating the proof of KM bound in Lemma \ref{lemma:bound estimate for ip}, we have
\begin{align}
&\int_{[t_{0},t_{0}+T]}\n{S_{\hbar}^{(1,1)}B_{N,\hbar,1,2}^{\pm}w_{N,\hbar}^{(2)}(t)}_{L_{x,x'}^{2}}dt\\
\leq&4\sum_{j=0}^{l_{c}}\lrs{4C_{V}\hbar^{-\al}T^{1/2}}^{j+1}
\n{S_{\hbar}^{(1,2+j)}w_{N,\hbar}^{(2+j)}(t_{0})}_{L_{x,x'}^{2}}+2N^{\frac{5}{2}\be-1}(8E_{0,\hbar})^{2}\notag\\
&+4^{l_{c}+2} (C_{V}\hbar^{-\al}T^{1/2})^{l_{c}+1}
\int_{[t_{0},t_{0}+T]}
\bbn{
S_{\hbar}^{(1,2+l_{c})}B_{N,\hbar,1,3+l_{c}}w_{N,\hbar}^{(3+l_{c})}(t_{3+l_{c}})
}_{L_{x,x'}^{2}}dt_{3+l_{c}}\notag
\end{align}
Since $2+l_{c}\leq (\ln N)^{10}$ and $T\leq \frac{\hbar^{2\al}}{\lrs{64E_{0,\hbar}C_{V}e}^{2}}$, we can employ KM bound in Lemma \ref{lemma:bound estimate for ip} to get
\begin{align*}
\leq &4\sum_{j=0}^{l_{c}}\lrs{4C_{V}\hbar^{-\al}T^{1/2}}^{j+1}
\n{S_{\hbar}^{(1,2+j)}w_{N,\hbar}^{(2+j)}(t_{0})}_{L_{x,x'}^{2}}+2N^{\frac{5}{2}\be-1}(8E_{0,\hbar})^{2}\notag\\
&+4^{l_{c}+2} (C_{V}\hbar^{-\al}T^{1/2})^{l_{c}+1}\lrs{16E_{0,\hbar}}^{3+l_{c}}
\end{align*}
Plugging in $T\leq \frac{\hbar^{2\al}}{\lrs{64E_{0,\hbar}C_{V}e}^{2}}$ and $C_{0,\hbar}=64E_{0,\hbar}$, we obtain $(\ref{equ:decay estimate for collapsing part})$.

\end{proof}
\subsection{Convergence Rate for Every Finite Time}\label{section:Convergence Rate for Every Finite Time}
In the section, we will iteratively use Proposition \ref{lemma:decay estimate for fp,dp,ip} to obtain the convergence rate for every finite time at the price of weakening the convergence rate.
\begin{proposition}\label{proposition:decay estimate, uniform in h, bbgky-hartree}
Let $T_{0}<+\infty$ and  $\al=d+1/2$. For $k\leq (\ln N)^{2}-(1-\frac{5}{2}\be)\sum_{j=0}^{n(T_{0},\hbar)}\frac{\ln N}{2^{j}j!}$, we have
\begin{align}\label{equ:decay estimate}
\sup_{t\in [0,T_{0}]}\n{S_{\hbar}^{(1,k)}w_{N,\hbar}^{(k)}(t)}_{L_{x,x'}^{2}}\leq (e^{n(T_{0},\hbar)}C_{0,\hbar})^{k}N^{\frac{\frac{5}{2}\be-1}{2^{n(T_{0},\hbar)}n(T_{0},\hbar)!}},
\end{align}
and
\begin{align}\label{equ:decay estimate,collapsing part}
\int_{[0,T_{0}]}\n{S_{\hbar}^{(1,1)}B_{N,\hbar,1,2}^{\pm}w_{N,\hbar}^{(2)}(t)}_{L_{x,x'}^{2}}dt\leq 8n(T_{0},\hbar)C_{0,\hbar}^{2}N^{\frac{\frac{5}{2}\be-1}{2^{n(T_{0},\hbar)}n(T_{0},\hbar)!}}
\end{align}
where $n(T_{0},\hbar)=(8eC_{V}C_{0,h})^{2}T_{0}/\hbar^{2\al}$ and $C_{0,\hbar}=64E_{0,\hbar}$ as defined in Proposition $\ref{lemma:decay estimate for fp,dp,ip}$. Moreover, under the restriction $(\ref{equ:restriction,N,h,bbgky-hnls})$ that
\begin{align}
N\geq e^{(2)}\lrs{\lrc{C_{V}^{2}E_{0,\hbar}^{2}T_{0}/\hbar^{2\al}}^{2}},
\end{align}
for $N\geq N_{0}(\be)$ we have $(\ref{equ:decay estimate, uniform in h,bbgky-hartree})$ and $(\ref{equ:decay estimate, uniform in h,bbgky-hartree,collapsing part})$ which we restate here
\begin{align*}
&\sup_{t\in [0,T_{0}]}\n{S_{\hbar}^{(1,1)}w_{N,\hbar}^{(1)}(t)}_{L_{x,x'}^{2}}\leq \lrs{\frac{1}{\ln N}}^{100},\\
&\int_{[0,T_{0}]}\n{S_{\hbar}^{(1,1)}B_{N,\hbar,1,2}^{\pm}w_{N,\hbar}^{(2)}(t)}_{L_{x,x'}^{2}}dt\leq \lrs{\frac{1}{\ln N}}^{100}.
\end{align*}

\end{proposition}
\begin{proof}
Step 0. Set $\la= \frac{1}{8eC_{V}C_{0,\hbar}}$. Then for
\begin{align*}
k\leq (\ln N)^{2}-(1-\frac{5}{2}\be)\ln N,\quad l_{c}\leq (1-\frac{5}{2}\be)\ln N,
\end{align*}
by estimate $(\ref{equ:decay estimate for fp,dp,ip})$ in Proposition \ref{lemma:decay estimate for fp,dp,ip}, we have
\begin{align}
&\sup_{t\in[0,\la^{2} \hbar^{2\al}]}\n{S_{\hbar}^{(1,k)}w_{N,\hbar}^{(k)}(t)}_{L_{x,x'}^{2}}\label{equ:interation estimate, decay, step 1,first}\\
\leq &
2^{k}\sum_{j=0}^{l_{c}}(4C_{V}\la)^{j}
\n{S_{\hbar}^{(1,k+j)}w_{N,\hbar}^{(k+j)}(0)}_{L_{x,x'}^{2}}
+(C_{0,\hbar})^{k}N^{\frac{5}{2}\be-1}+(C_{0,\hbar})^{k}\lrs{\frac{1}{e}}^{l_{c}+1}.\notag
\end{align}
By initial condition $(\ref{equ:initial condition, factorized})$ in condition $(c)$, we plug in $\la=\frac{1}{8eC_{V}C_{0,\hbar}}$ and take $l_{c}=(1-\frac{5}{2}\be)\ln N$ to get
\begin{align}\label{equ:interation estimate, decay, step 1}
\sup_{t\in[0,\la^{2} \hbar^{2\al}]}\n{S_{\hbar}^{(1,k)}w_{N,\hbar}^{(k)}(t)}_{L_{x,x'}^{2}} \leq 4(C_{0,\hbar})^{k}N^{\frac{5}{2}\be-1}
\end{align}
for every $k\leq (\ln N)^{2}-(1-\frac{5}{2}\be)\ln N$.

Step 1. Let $t_{1}=\la^{2}\hbar^{2\al}$. For
\begin{align*}
k\leq (\ln N)^{2}-(1-\frac{5}{2}\be)\lrs{\ln N+\frac{\ln N}{2}},\quad l_{c}\leq (1-\frac{5}{2}\be)\ln N,
\end{align*}
we make use of estimate $(\ref{equ:decay estimate for fp,dp,ip})$ in Proposition \ref{lemma:decay estimate for fp,dp,ip} again to obtain
\begin{align*}
&\sup_{t\in[t_{1},t_{1}+\la^{2} \hbar^{2\al}]}\n{S_{\hbar}^{(1,k)}w_{N,\hbar}^{(k)}(t)}_{L_{x,x'}^{2}}\\
\leq &
2^{k}\sum_{j=0}^{l_{c}}(4C_{V}\la)^{j}
\n{S_{\hbar}^{(1,k+j)}w_{N,\hbar}^{(k+j)}(t_{1})}_{L_{x,x'}^{2}}+(C_{0,\hbar})^{k}N^{\frac{5}{2}\be-1}+(C_{0,\hbar})^{k}\lrs{\frac{1}{e}}^{l_{c}+1}
\end{align*}
Since $k+l_{c}\leq (\ln N)^{2}-(1-\frac{5}{2}\be)\ln N$, one can adopt estimate $(\ref{equ:interation estimate, decay, step 1})$ in Step 0 to reach
\begin{align*}
\leq &
N^{\frac{5}{2}\be-1}4(C_{0,\hbar})^{k}\sum_{j=0}^{l_{c}}(4C_{V}C_{0,\hbar}\la )^{j}
+(C_{0,\hbar})^{k}N^{\frac{5}{2}\be-1}+(C_{0,\hbar})^{k}\lrs{\frac{1}{e}}^{l_{c}+1}
\end{align*}
Recalling $\la=\frac{1}{8eC_{V}C_{0,\hbar}}$,
\begin{align*}
\leq &
N^{\frac{5}{2}\be-1}8(C_{0,\hbar})^{k}
+(C_{0,\hbar})^{k}N^{\frac{5}{2}\be-1}+(C_{0,\hbar})^{k}\lrs{\frac{1}{e}}^{l_{c}+1}
\end{align*}
By taking $l_{c}=(1-\frac{5}{2}\be)\ln N/2$, we arrive at
\begin{align}\label{equ:interation estimate, decay, step 1,second}
\sup_{t\in[t_{1},t_{1}+\la^{2} \hbar^{2\al}]}\n{S_{\hbar}^{(1,k)}w_{N,\hbar}^{(k)}(t)}_{L_{x,x'}^{2}}\leq& (eC_{0,\hbar})^{k}N^{\frac{\frac{5}{2}\be-1}{2}}
\end{align}
for every $k\leq (\ln N)^{2}-(1-\frac{5}{2}\be)\lrs{\ln N+\frac{\ln N}{2}}$.

Step $m$.
Let $t_{m}=m\la^{2}\hbar^{2\al}$. Now we assume $(\ref{equ:interation estimate, decay, step 1,second})$ is true for the case $n=m$, that is,
\begin{align}\label{equ:interation estimate, decay, step m}
&\sup_{t\in[t_{m},t_{m}+\la^{2} \hbar^{2\al}]}\n{S_{\hbar}^{(1,k)}w_{N,\hbar}^{(k)}(t)}_{L_{x,x'}^{2}}
\leq  (e^{m}C_{0,\hbar})^{k}N^{\frac{\frac{5}{2}\be-1}{2^{m}m!}}
\end{align}
for every $k\leq (\ln N)^{2}-(1-\frac{5}{2}\be)\sum_{j=0}^{m}\frac{\ln N}{2^{j}j!}$. Then we will prove it for $n=m+1$.

For
\begin{align*}
k\leq (\ln N)^{2}-(1-\frac{5}{2}\be)\sum_{j=0}^{m+1}\frac{\ln N}{2^{j}j!},\quad l_{c}\leq \frac{(1-\frac{5}{2}\be)\ln N}{2^{m+1}(m+1)!},
\end{align*}
one can employ estimate $(\ref{equ:decay estimate for fp,dp,ip})$ in Proposition \ref{lemma:decay estimate for fp,dp,ip} to reach
\begin{align*}
&\sup_{t\in[t_{m+1},t_{m+1}+\la^{2} \hbar^{2\al}]}\n{S_{\hbar}^{(1,k)}w_{N,\hbar}^{(k)}(t)}_{L_{x,x'}^{2}}\\
\leq &
2^{k}\sum_{j=0}^{l_{c}}(4C_{V}\la)^{j}
\n{S_{\hbar}^{(1,k+j)}w_{N,\hbar}^{(k+j)}(t_{m+1})}_{L_{x,x'}^{2}}+(C_{0,\hbar})^{k}N^{\frac{5}{2}\be-1}+(C_{0,\hbar})^{k}
\lrs{\frac{1}{e}}^{l_{c}+1}
\end{align*}
Since
$k+l_{c}\leq (\ln N)^{2}-(1-\frac{5}{2}\be)\sum_{j=0}^{m}\frac{\ln N}{2^{j}j!}$,
one can use estimate $(\ref{equ:interation estimate, decay, step m})$ in the case $n=m$ to get
\begin{align*}
\leq &N^{\frac{\frac{5}{2}\be-1}{2^{m}m!}}(2e^{m}C_{0,\hbar})^{k}\sum_{j=0}^{l_{c}}(4C_{V}\la)^{j}(e^{m}C_{0,\hbar})^{j}+
(C_{0,\hbar})^{k}N^{\frac{5}{2}\be-1}+(C_{0,\hbar})^{k}\lrs{\frac{1}{e}}^{l_{c}+1}
\end{align*}
Recalling $\la=\frac{1}{8eC_{V}C_{0,\hbar}}$,
\begin{align*}
\leq& (2e^{m}C_{0,\hbar})^{k} N^{\frac{\frac{5}{2}\be-1}{2^{m}m!}}\lrs{e^{m}}^{l_{c}+1}+
(C_{0,\hbar})^{k}N^{\frac{5}{2}\be-1}+(C_{0,\hbar})^{k}\lrs{\frac{1}{e}}^{l_{c}+1}
\end{align*}
Taking $l_{c}+1=\frac{(1-\frac{5}{2}\be)\ln N}{2^{m+1}(m+1)!}$, we arrive at
\begin{align*}
&\sup_{t\in[t_{m+1},t_{m+1}+\la^{2} \hbar^{2\al}]}\n{S_{\hbar}^{(1,k)}w_{N,\hbar}^{(k)}(t)}_{L_{x,x'}^{2}}\\
\leq& (2e^{m}C_{0,\hbar})^{k}N^{\frac{\frac{5}{2}\be-1}{2^{m+1}m!}}+(C_{0,\hbar})^{k}N^{\frac{5}{2}\be-1}+(C_{0,\hbar})^{k}N^{\frac{\frac{5}{2}\be-1}{2^{m+1}(m+1)!}}\\
\leq& (e^{m+1}C_{0,\hbar})^{k}N^{\frac{\frac{5}{2}\be-1}{2^{m+1}(m+1)!}}.
\end{align*}
This proves $(\ref{equ:interation estimate, decay, step m})$ and completes the proof of $(\ref{equ:decay estimate})$ as we can take $m=n(T_{0},\hbar)=(8eC_{V}C_{0,\hbar})^{2}T_{0}/\hbar^{2\al}$.

For $(\ref{equ:decay estimate,collapsing part})$, we can use estimate $(\ref{equ:decay estimate for collapsing part})$ in Proposition \ref{lemma:decay estimate for fp,dp,ip} to get to
\begin{align}
&\int_{[t_{m},t_{m}+\la^{2}\hbar^{2\al}]}\n{S_{\hbar}^{(1,1)}B_{N,\hbar,1,2}^{\pm}w_{N,\hbar}^{(2)}(t)}_{L_{x,x'}^{2}}dt\\
\leq &4\sum_{j=0}^{l_{c}}\lrs{4C_{V}\la }^{j+1}
\n{S_{\hbar}^{(1,2+j)}w_{N,\hbar}^{(2+j)}(t_{m})}_{L_{x,x'}^{2}}
+
C_{0,\hbar}^{2}N^{\frac{5}{2}\be-1}+C_{0,\hbar}^{2}\lrs{\frac{1}{e}}^{l_{c}+1}\notag
\end{align}
Plugging in estimate $(\ref{equ:interation estimate, decay, step m})$,
\begin{align*}
\leq 4\sum_{j=0}^{l_{c}}\lrs{4C_{V}\la }^{j+1}
(e^{m}C_{0,\hbar})^{j+2}N^{\frac{\frac{5}{2}\be-1}{2^{m}m!}}
+
C_{0,\hbar}^{2}N^{\frac{5}{2}\be-1}+C_{0,\hbar}^{2}\lrs{\frac{1}{e}}^{l_{c}+1}
\end{align*}
Recalling $\la=\frac{1}{8eC_{V}C_{0,\hbar}}$,
\begin{align*}
\leq 4C_{0,\hbar} (e^{m})^{l_{c}+2}N^{\frac{\frac{5}{2}\be-1}{2^{m}m!}}+C_{0,\hbar}^{2}N^{\frac{5}{2}\be-1}+C_{0,\hbar}^{2}\lrs{\frac{1}{e}}^{l_{c}+1}
\end{align*}
Setting $l_{c}+2=\frac{(1-\frac{5}{2}\be)\ln N}{2^{m+1}(m+1)!}$, we arrive at
\begin{align*}
\leq& 4C_{0,\hbar}N^{\frac{\frac{5}{2}\be-1}{2^{m+1}m!}}+C_{0,\hbar}^{2}N^{\frac{5}{2}\be-1}+eC_{0,\hbar}^{2}
N^{\frac{\frac{5}{2}\be-1}{2^{m+1}(m+1)!}}\\
\leq& 8C_{0,\hbar}^{2}N^{\frac{\frac{5}{2}\be-1}{2^{m+1}(m+1)!}}.
\end{align*}
Then by summing the integration time domain, we obtain
\begin{align}
&\int_{[0,T_{0}]}\n{S_{\hbar}^{(1,1)}B_{N,\hbar,1,2}^{\pm}w_{N,\hbar}^{(2)}(t)}_{L_{x,x'}^{2}}dt\\
\leq &\sum_{m=0}^{n(T_{0},\hbar)}\int_{[t_{m},t_{m+1}]}\n{S_{\hbar}^{(1,1)}B_{N,\hbar,1,2}^{\pm}w_{N,\hbar}^{(2)}(t)}_{L_{x,x'}^{2}}dt\notag\\
\leq& \sum_{m=0}^{n(T_{0},\hbar)}8C_{0,\hbar}^{2}N^{\frac{\frac{5}{2}\be-1}{2^{m+1}(m+1)!}}\notag\\
\leq& 8n(T_{0},\hbar)C_{0,\hbar}^{2}N^{\frac{\frac{5}{2}\be-1}{2^{n(T_{0},\hbar)}n(T_{0},\hbar)!}}.\notag
\end{align}
This completes the proof of $(\ref{equ:decay estimate,collapsing part})$.

For estimates $(\ref{equ:decay estimate, uniform in h,bbgky-hartree})$ and $(\ref{equ:decay estimate, uniform in h,bbgky-hartree,collapsing part})$, under the restriction $(\ref{equ:restriction,N,h,bbgky-hnls})$ that
\begin{align}
N\geq e^{(2)}\lrs{\lrc{C_{V}^{2}E_{0,\hbar}^{2}T_{0}/\hbar^{7}}^{2}}
\end{align}
which implies that $n(T_{0},\hbar)\leq \sqrt{C\ln\ln N}$ with an absolute constant $C$,
we have
\begin{align*}
2^{n(T_{0},\hbar)}n(T_{0},\hbar)!\leq n(T_{0},\hbar)^{n(T_{0},\hbar)}\leq (\sqrt{C\ln\ln N})^{\sqrt{C\ln\ln N}}\leq \sqrt{\ln N}.
\end{align*}
Also, we have
\begin{align*}
8n(T_{0},\hbar)C_{0,\hbar}^{2}\leq e^{n(T_{0},\hbar)}C_{0,\hbar}\leq n(T_{0},\hbar)^{n(T_{0},\hbar)}\leq \sqrt{\ln N}.
\end{align*}
Hence, we obtain
\begin{align*}
&\sup_{t\in [0,T_{0}]}\n{S_{\hbar}^{(1,1)}w_{N,\hbar}^{(1)}(t)}_{L_{x,x'}^{2}}\leq e^{n(T_{0},\hbar)}C_{0,\hbar}N^{\frac{\frac{5}{2}\be-1}{2^{n(T_{0},\hbar)}n(T_{0},\hbar)!}}\leq \frac{\sqrt{\ln N}}{N^{\frac{1-\frac{5}{2}\be}{\sqrt{\ln N}}}}\leq \lrs{\frac{1}{\ln N}}^{100},\\
&\int_{[0,T_{0}]}\n{S_{\hbar}^{(1,1)}B_{N,\hbar,1,2}^{\pm}w_{N,\hbar}^{(2)}(t)}_{L_{x,x'}^{2}}dt\leq 8n(T_{0},\hbar)C_{0,\hbar}^{2}N^{\frac{\frac{5}{2}\be-1}{2^{n(T_{0},\hbar)}n(T_{0},\hbar)!}}\leq\lrs{\frac{1}{\ln N}}^{100},
\end{align*}
for $N\geq N_{0}(\be)$. This completes the proof of estimates $(\ref{equ:decay estimate, uniform in h,bbgky-hartree})$ and $(\ref{equ:decay estimate, uniform in h,bbgky-hartree,collapsing part})$.
\end{proof}

\section{H-NLS v.s. the Compressible Euler Equation: a Modulated Energy Approach}\label{section:H-NLS v.s. the Compressible Euler Equation: a Modulated Energy Approach}
We will compare the H-NLS equation $(\ref{equ:N-hartree equation})$ and the compressible Euler equation $(\ref{equ:euler equation})$ before its blowup time by the method of modulated energy. Recall the H-NLS equation $(\ref{equ:N-hartree equation})$
\begin{align*}
\begin{cases}
&i\hbar\pa_{t}\phi_{N,\hbar}=-\frac{1}{2}\hbar^{2}\Delta\phi_{N,\hbar}+\lrs{V_{N}*|\phi_{N,\hbar}|^{2}}\phi_{N,\hbar},\\
&\phi_{N,\hbar}(0)=\phi_{N,\hbar}^{in},
\end{cases}
\end{align*}
with the mass density and momentum density defined by $(\ref{equ:quantum mass density,momentum density,one-body})$
\begin{align*}
&\rho_{N,\hbar}(t,x)=|\phi_{N,\hbar}(t,x)|^{2},\quad
J_{N,\hbar}(t,x)=\hbar\operatorname{Im}\lrs{\ol{\phi_{N,\hbar}}(t,x)\nabla \phi_{N,\hbar}(t,x)},
\check{}\end{align*}
and the compressible Euler equation $(\ref{equ:euler equation})$
\begin{equation*}
\begin{cases}
& \partial _{t}\rho +\nabla \cdot \left( \rho u\right) =0, \\
& \partial _{t}u+(u\cdot \nabla )u+b_{0}\nabla \rho =0,\\
& (\rho,u)|_{t=0}=(\rho^{in},u^{in}).
\end{cases}
\end{equation*}

Here is the main theorem of the section.
\begin{theorem}\label{thm:semiclassical limit}
Let $\phi_{N,\hbar}(t)$ be the solution to H-NLS equation with the initial data $\phi_{N,h}^{in}$.
Under the same conditions of Theorem $\ref{thm:main theorem,bbgky-to-euler}$.
Then we have\footnote{Under the restriction $(\ref{equ:restriction,N,h})$, the smallness factor $\frac{1}{\hbar^{4}N^{\be}}$ can be absorbed into $\hbar^{2}$.}
\begin{align}
&\n{\rho_{N,\hbar}-\rho}_{L^{\infty}([0,T_{0}];L^{2}(\R^{d}))}\leq C(T_{0})\lrs{\frac{1}{\hbar^{4}N^{\be}}+\hbar^{2}}^{\frac{1}{2}},\label{equ:convergence,uniform in h,hartree-euler,density}\\
&\n{J_{N,\hbar}-\rho u}_{L^{\infty}([0,T_{0}];L^{r}(\R^{d}))}\leq C(T_{0})\lrs{\frac{1}{\hbar^{4}N^{\be}}+\hbar^{2}}^{\frac{1}{2}\lrs{\frac{4}{r}-3}}\label{equ:convergence,uniform in h,hartree-euler,moment},
\end{align}
where $r\in [1,4/3)$,
\begin{align}
&\n{\rho_{N,\hbar}V_{N}*\rho_{N,\hbar}(t,x)-b_{0}\rho(t,x)^{2}}_{L^{1}([0,T_{0}];L^{1}(\R^{d}))}\leq
C(T_{0})\lrs{\frac{1}{\hbar^{4}N^{\be}}+\hbar^{2}}^{\frac{1}{2}}.\label{equ:convergence,uniform in h,hartree-euler,pressure}
\end{align}

\end{theorem}
\begin{proof}[\textbf{Proof of Theorem $\ref{thm:semiclassical limit}$}]
By  $(\ref{equ:quantitative estimate, density})$ and $(\ref{equ:quantitative estimate, moment})$ in Proposition \ref{proposition:upper bound for modulated energy}, we have
\begin{align*}
&\n{\rho_{N,\hbar}-\rho}_{L^{\infty}([0,T_{0}];L^{2}(\R^{d}))}\leq C(T_{0}) \lrs{\frac{1}{\hbar^{4}N^{\be}}+\hbar^{2}}^{\frac{1}{2}},\\
&\n{(i\hbar\nabla-u)\phi_{N,\hbar}}_{L^{\infty}([0,T_{0}];L^{2}(\R^{d}))}\leq C(T_{0})\lrs{\frac{1}{\hbar^{4}N^{\be}}+\hbar^{2}}^{\frac{1}{2}},
\end{align*}
which directly completes the proof of $(\ref{equ:convergence,uniform in h,hartree-euler,density})$.

For $(\ref{equ:convergence,uniform in h,hartree-euler,moment})$, by the triangle and H\"{o}lder's inequalities as well as estimates $(\ref{equ:quantitative estimate, density})$ and $(\ref{equ:quantitative estimate, moment})$ we have
\begin{align*}
\n{J_{N,\hbar}-\rho u}_{L^{1}(\R^{d})}\leq &\n{J_{N,\hbar}-\rho_{N,\hbar}u}_{L^{1}(\R^{d})}+\n{\rho_{N,\hbar}u-\rho u}_{L^{1}(\R^{d})}\\
=&\n{\operatorname{Im}\lrs{\ol{\phi_{N,\hbar}}\lrs{\hbar\nabla -iu}\phi_{N,\hbar}}}_{L^{1}(\R^{d})}+\n{\rho_{N,\hbar}u-\rho u}_{L^{1}(\R^{d})}\\
\leq& \n{\phi_{N,\hbar}}_{L^{2}(\R^{d})}\n{(i\hbar\nabla-u)\phi_{N,\hbar}}_{L^{2}(\R^{d})}+\n{u}_{L^{2}(\R^{d})}
\n{\rho_{N,\hbar}-\rho}_{L^{2}(\R^{d})}\\
\leq &C(T_{0}) \lrs{\frac{1}{\hbar^{4}N^{\be}}+\hbar^{2}}^{\frac{1}{2}}
\end{align*}
On the other hand, by the energy bound for $\phi_{N,\hbar}$ and the uniform bound for $\n{\rho_{N,\hbar}}_{L^{2}}$ we have
\begin{align}
\n{J_{N,\hbar}}_{L^{4/3}}\leq &\n{\hbar\nabla \phi_{N,\hbar}}_{L^{2}}\n{\phi_{N,\hbar}}_{L^{4}}\lesssim E_{0},
\end{align}
where we used energy bound and uniform bound for $\n{\rho_{N,\hbar}}_{L^{2}}$ in the last inequality. Hence, by interpolation inequality we obtain
\begin{align}
\n{J_{N,\hbar}-\rho u}_{L^{\infty}([0,T_{0}];L^{r}(\R^{d}))}\leq& \n{J_{N,\hbar}-\rho u}_{L^{\infty}([0,T_{0}];L^{1}(\R^{d}))}^{1-\al}\n{J_{N,\hbar}-\rho u}_{L^{\infty}([0,T_{0}];L^{4/3}(\R^{d}))}^{\al}\\
\leq &C\lrs{\frac{1}{\hbar^{4}N^{\be}}+\hbar^{2}}^{\frac{1-\al}{2}}E_{0}^{\al},\notag
\end{align}
where $\al=4-4/r$.
This completes the proof of $(\ref{equ:convergence,uniform in h,hartree-euler,moment})$.

For $(\ref{equ:convergence,uniform in h,hartree-euler,pressure})$, by triangle inequality we have
\begin{align}
\n{\rho_{N,\hbar}V_{N}*\rho_{N,\hbar}-b_{0}\rho^{2}}_{L^{1}(\R^{d})}\leq \n{\rho_{N,\hbar}V_{N}*\rho_{N,\hbar}-
b_{0}(\rho_{N,\hbar})^{2}}_{L^{1}(\R^{d})}+
b_{0}\n{(\rho_{N,\hbar})^{2}-\rho^{2}}_{L^{1}(\R^{d})}
\end{align}
By the approximation of identity estimate $(\ref{equ:estimate for error term, modulated energy})$ which reads
\begin{align}
\n{\rho_{N,\hbar}V_{N}*\rho_{N,\hbar}-b_{0}(\rho_{N,\hbar})^{2}}_{L^{1}(\R^{d})}\lesssim \frac{1}{\hbar^{4}N^{\be}}
\end{align}
and estimate $(\ref{equ:convergence,uniform in h,hartree-euler,density})$, we have
\begin{align*}
\n{\rho_{N,\hbar}V_{N}*\rho_{N,\hbar}-b_{0}\rho^{2}}_{L^{1}(\R^{d})}\lesssim & \frac{1}{\hbar^{4}N^{\be}}+\n{\rho_{N,\hbar}-\rho}_{L^{2}(\R^{d})}\lrs{\n{\rho_{N,\hbar}}_{L^{2}(\R^{d})}+\n{\rho}_{L^{2}(\R^{d})}}\\
\leq &C(T_{0})\lrs{\frac{1}{\hbar^{4}N^{\be}}+\hbar^{2}}^{\frac{1}{2}}.
\end{align*}
By taking $L^{\infty}$ norm at $dt$, we complete the proof of $(\ref{equ:convergence,uniform in h,hartree-euler,pressure})$. Thus we have proved Theorem $\ref{thm:semiclassical limit}$ assuming
Proposition $\ref{proposition:upper bound for modulated energy}$ and $(\ref{equ:estimate for error term, modulated energy})$. The rest of this section is to prove them.
\end{proof}

\subsection{The Evolution of the Modulated Energy}\label{section:The Evolution of the Modulated Energy}
We consider the following modulated energy
\begin{align}\label{equ:modulated energy}
\mathcal{M}\lrc{\phi_{N,\hbar},\rho,u}(t)=&\frac{1}{2}\int_{\R^{d}}|(i\hbar\nabla-u)\phi_{N,\hbar}(t)|^{2}dx\\
&+\frac{1}{2}\lra{V_{N}*\rho_{N,\hbar},
\rho_{N,\hbar}}+\frac{b_{0}}{2}\int_{\R^{d}}\rho^{2}dx-b_{0}\int_{\R^{d}}\rho\rho_{N,\hbar}dx.\notag
\end{align}
We need to derive a time evolution equation for $\mathcal{M}\lrc{\phi_{N,\hbar},\rho,u}(t)$.
The related quantities for $\phi_{N,\hbar}$ are given as the following.
\begin{lemma}\label{lemma:moment equation}
We have the following estimates regarding $\phi_{N,\hbar}$:
\begin{align}
&\pa_{t}\rho_{N,\hbar}+\operatorname{div}J_{N,\hbar}=0,\label{equ:mass conservation}\\
&\pa_{t}J_{N,\hbar}^{j}+\sum_{j,k}\pa_{k}\lrc{\hbar^{2}\operatorname{Re}\lrs{\pa_{j}\ol{\phi_{N,\hbar}}
\pa_{k}\phi_{N,\hbar}}-\frac{\hbar^{2}}{4}\pa_{jk}\rho_{N,\hbar}}+
\lrs{\pa_{j}\lrs{V_{N}*\rho_{N,\hbar}}}\rho_{N,\hbar}=0,
\label{equ:momontum conservation}\\
&E_{N,\hbar}(t)\equiv E_{N,\hbar}(0),\label{equ:energy conservation law}
\end{align}
where the energy $E_{N,\hbar}(t)$ is defined by
\begin{align}
E_{N,\hbar}(t)=\frac{1}{2}\n{\hbar\nabla \phi_{N,\hbar}(t)}_{L^{2}}^{2}+\frac{1}{2}\lra{V_{N}*\rho_{N,\hbar},\rho_{N,\hbar}}(t).
\end{align}
We also have the approximation of identity estimate:
\begin{align}
&\n{\rho_{N,\hbar}V_{N}*\rho_{N,\hbar}-b_{0}(\rho_{N,\hbar})^{2}}_{L^{1}(\R^{d})}\lesssim \frac{1}{\hbar^{4}N^{\be}}.\label{equ:estimate for error term, modulated energy}
\end{align}
\end{lemma}
\begin{proof}
We omit the proof of $(\ref{equ:mass conservation})-(\ref{equ:energy conservation law})$ as this is a direct computation and is well-known in $H^{1}$ wellposedness
theory. For $(\ref{equ:estimate for error term, modulated energy})$, we set $W_{N}=V_{N}-b_{0}\delta$ and rewrite
\begin{align*}
\n{\rho_{N,\hbar}V_{N}*\rho_{N,\hbar}-b_{0}(\rho_{N,\hbar})^{2}}_{L^{1}(\R^{d})}=\n{\rho_{N,\hbar}W_{N}*\rho_{N,\hbar}}_{L^{1}(\R^{d})}
\end{align*}
By H\"{o}lder,
\begin{align*}
\leq&\n{W_{N}*\rho_{N,\hbar}}_{L^{3/2}}\n{\rho_{N,\hbar}}_{L^{3}}
\end{align*}
By Lemma \ref{lemma:quantative estimate for identity approximation},
\begin{align*}
\lesssim& N^{-\be}\n{\lra{\nabla}\rho_{N,\hbar}}_{L^{3/2}}\n{\rho_{N,\hbar}}_{L^{3}}
\end{align*}
By fractional Leibniz rule in Lemma \ref{lemma:leibniz rule} and Sobolev inequality,
\begin{align*}
\lesssim&N^{-\be}\n{\phi_{N,\hbar}}_{H^{1}}^{4}
\end{align*}
By the energy bound for $\phi_{N,\hbar}$,
\begin{align*}
\lesssim& \frac{1}{\hbar^{4}N^{\be}},
\end{align*}
which completes the proof of $(\ref{equ:estimate for error term, modulated energy})$.
\end{proof}

Next let us derive the time derivative of $\mathcal{M}\lrc{\phi_{N,\hbar},\rho,u}(t)$.
\begin{proposition}\label{proposition:the evolution of modulated energy}
There holds
\begin{align}\label{equ:the evolution of modulated energy}
&\frac{d}{dt}\mathcal{M}\lrc{\phi_{N,\hbar},\rho,u}(t)\\
=&-\int_{\R^{d}}\pa_{k}u^{j}\operatorname{Re}
\lrs{(\hbar\pa_{k}-iu^{k})\phi_{N,\hbar}\ol{\lrs{\hbar\pa_{j}-iu^{j}}\phi_{N,\hbar}}}\notag\\
&-\frac{b_{0}}{2}\int_{\R^{d}}\operatorname{div}u(\rho_{N,\hbar}-\rho)^{2}dx-\frac{\hbar^{2}}{4}\int_{\R^{d}}
 \rho_{N,\hbar}\lrs{\Delta \operatorname{div} u}dx+Er\notag
\end{align}
where the summation convention for repeated indices is used and the error term is given by
\begin{align}\label{equ:error term, modulated energy}
Er=\int_{\R^{d}}u^{j}(\pa_{j}(V_{N}*\rho_{N,\hbar}))\rho_{N,\hbar}dx+\frac{b_{0}}{2}\int_{\R^{d}} \operatorname{div}u (\rho_{N,\hbar})^{2}dx.
\end{align}
\end{proposition}
\begin{proof}
By energy conservation law $(\ref{equ:energy conservation law})$ in Lemma $(\ref{lemma:moment equation})$, we obtain
\begin{align*}
\frac{d}{dt}\mathcal{M}\lrc{\phi_{N,\hbar},\rho,u}(t)=&\frac{1}{2}\frac{d}{dt}\n{\hbar\nabla \phi_{N,\hbar}(t)}_{L^{2}}^{2}+\frac{1}{2}\frac{d}{dt}\int_{\R^{d}}|u|^{2}\rho_{N,\hbar}dx
-\frac{d}{dt}\int_{\R^{d}}J_{N,\hbar}udx\\
&+\frac{1}{2}\frac{d}{dt}\lra{V_{N}*\rho_{N,\hbar},\rho_{N,\hbar}}+\frac{b_{0}}{2}\frac{d}{dt}\int_{\R^{d}} \rho^{2}dx-b_{0}\frac{d}{dt}\int_{\R^{d}}\rho_{N,\hbar}\rho dx\\
=&\frac{1}{2}\frac{d}{dt}\int_{\R^{d}}|u|^{2}\rho_{N,\hbar}dx
-\frac{d}{dt}\int_{\R^{d}}J_{N,\hbar}udx\\
&+\frac{b_{0}}{2}\frac{d}{dt}\int_{\R^{d}} \rho^{2}dx-b_{0}\frac{d}{dt}\int_{\R^{d}}\rho_{N,\hbar}\rho dx.
\end{align*}

Next, we calculate the above four terms separately. For the first term, by $(\ref{equ:euler equation})$ and $(\ref{equ:mass conservation})$ we find
\begin{align}\label{equ:modulated energy, the first term}
\frac{1}{2}\frac{d}{dt}\int_{\R^{d}}|u|^{2}\rho_{N,\hbar}dx
=&\int_{\R^{d}}\lrc{u\pa_{t}u\rho_{N,\hbar}+\frac{1}{2}|u|^{2}\pa_{t}\rho_{N,\hbar}}dx\\
=&\int_{\R^{d}}\lrc{\pa_{t}u^{j} \rho_{N,\hbar} u^{j}-\frac{1}{2}|u|^{2}\operatorname{div}J_{N,\hbar}}dx\notag\\
=&\int_{\R^{d}}\lrc{-\rho_{N,\hbar} u^{j}u^{k}\pa_{k}u_{j}-b_{0}\rho_{N,\hbar} u^{j}\pa_{j}\rho +J_{N,\hbar}^{j}u^{k}\pa_{j}u^{k}}dx\notag
\end{align}
where we have used integration by parts in the last equality.

For the second term, via $(\ref{equ:momontum conservation})$ and $(\ref{equ:euler equation})$ we have
\begin{align}\label{equ:modulated energy, the second term}
&-\frac{d}{dt}\int_{\R^{d}}J_{N,\hbar}udx\\
=& \int_{\R^{d}}\lrs{-\pa_{t}J_{N,\hbar}u-J_{N,\hbar}\pa_{t}u}dx\notag\\
=& \int_{\R^{d}} \lrs{
\partial _{k}\left( \hbar^{2}\operatorname{Re}(\partial _{j}\overline{\phi
_{N,h }}\partial _{k}\phi_{N,h })-\frac{\hbar^{2}}{4}\partial _{jk}^{2}\rho
_{N,\hbar}\right) +(\pa_{j}(V_{N}*\rho_{N,\hbar}))\rho_{N,\hbar}}u^{j}dx\notag\\
&+\int_{\R^{d}}J_{N,\hbar}^{j}u^{k}\pa_{k}u^{j}dx
+b_{0}\int_{\R^{d}}J_{N,\hbar}^{j}\pa_{j}\rho dx\notag
\end{align}
Integrating by parts and using $(\ref{equ:error term, modulated energy})$,
\begin{align*}
=&\int_{\R^{d}} -\hbar^{2}
\partial_{k}u^{j}\lrc{ \operatorname{Re}(\partial _{j}\overline{\phi
_{N,h }}\partial _{k}\phi _{N,h })} dx-\int_{\R^{d}}
\frac{\hbar^{2}}{4}\rho
_{N,h }\partial _{jk}^{2}\pa_{k}u^{j}dx\\
&-\frac{b_{0}}{2}\int \operatorname{div}u (\rho_{N,\hbar})^{2}dx +Er\notag
+\int_{\R^{d}}J_{N,\hbar}^{j}u^{k}\pa_{k}u^{j}dx
+b_{0}\int_{\R^{d}}J_{N,\hbar}^{j}\pa_{j}\rho dx.\notag
\end{align*}

For the third term, using $(\ref{equ:euler equation})$ and integration by parts, we obtain
\begin{align}\label{equ:modulated energy, the third term}
\frac{b_{0}}{2}\frac{d}{dt}\int_{\R^{d}} \rho^{2}dx
=&b_{0}\int _{\R^{d}}\rho \pa_{t}\rho dx
=-b_{0}\int _{\R^{d}}\rho \operatorname{div}(\rho u)dx\\
=&b_{0}\int_{\R^{d}}\lrs{\pa_{j}\rho}\rho u^{j}dx
=-\frac{b_{0}}{2}\int _{\R^{d}}\rho^{2} \operatorname{div}u dx\notag.
\end{align}

For the forth term, plugging in $(\ref{equ:euler equation})$ and $(\ref{equ:mass conservation})$, we integrate by parts to get
\begin{align}\label{equ:modulated energy, the forth term}
-b_{0}\frac{d}{dt}\int_{\R^{d}}\rho_{N,\hbar}\rho dx
=&b_{0}\int_{\R^{d}}-\rho\pa_{t}\rho_{N,\hbar}-\rho_{N,\hbar}\pa_{t}\rho dx \\
=&b_{0}\int_{\R^{d}} \rho \operatorname{div}J_{N,\hbar}+\rho_{N,\hbar}\operatorname{div}(\rho u) dx\notag\\
=&b_{0}\int_{\R^{d}} -\pa_{j}\rho J_{N,\hbar}^{j}+\rho_{N,\hbar}\rho \operatorname{div}u +\rho_{N,\hbar}u^{j}\pa_{j}\rho dx.\notag
\end{align}

Summing up $(\ref{equ:modulated energy, the first term})-(\ref{equ:modulated energy, the forth term})$, we conclude
\begin{align*}
&\frac{d}{dt}\mathcal{M}\lrc{\phi_{N,\hbar},\rho,u}(t)\\
=&\int_{\R^{d}}\lrc{-\rho_{N,\hbar} u^{j}u^{k}\pa_{k}u^{j}-b_{0}\rho_{N,\hbar} u^{j}\pa_{j}\rho +J_{N,\hbar}^{j}u^{k}\pa_{j}u^{k}}dx\\
&+\int_{\R^{d}} -\hbar ^{2}
\partial_{k}u^{j}\lrc{ \operatorname{Re}(\partial _{j}\overline{\phi
_{N,h }}\partial _{k}\phi _{N,h })} dx-\int_{\R^{d}}
\frac{\hbar^{2}}{4}\rho
_{N,h }\partial _{jk}^{2}\pa_{k}u^{j}dx\\
&-\frac{b_{0}}{2}\int \operatorname{div}u (\rho_{N,\hbar})^{2}dx +Er\notag
+\int_{\R^{d}}J_{N,\hbar}^{j}u^{k}\pa_{k}u^{j}dx
+\int_{\R^{d}}b_{0}J_{N,\hbar}^{j}\pa_{j}\rho dx\\
&-\frac{b_{0}}{2}\int _{\R^{d}}\rho^{2} \operatorname{div}u dx+\int_{\R^{d}} b_{0}\rho_{N,\hbar}\rho \operatorname{div}u +b_{0}\rho_{N,\hbar}u^{j}\pa_{j}\rho -b_{0}\pa_{j}\rho J_{N,\hbar}^{j} dx\\
=&-\int_{\R^{d}}\pa_{k}u^{j}\lr{\rho_{N,\hbar} u^{j}u^{k} +\hbar ^{2}
\lrc{ \operatorname{Re}(\partial _{j}\overline{\phi
_{N,h }}\partial _{k}\phi _{N,h })}- J_{N,\hbar}^{j}u^{k}-J_{N,\hbar}^{k}u^{j}}dx\\
&-\frac{b_{0}}{2}\int_{\R^{d}}\operatorname{div}u(\rho_{N,\hbar}-\rho)^{2}dx-\frac{\hbar^{2}}{4}\int_{\R^{d}}
\rho_{N,\hbar}\lrs{\Delta \operatorname{div} u}dx+Er
\end{align*}
which is equivalent to $(\ref{equ:the evolution of modulated energy})$. This completes the proof.

\end{proof}

\subsection{Modulated Energy Estimate}\label{section:Modulated Energy Estimate}
We first estimate for the error term $(\ref{equ:error term, modulated energy})$ and then establish Gronwall's inequality for the modulated energy $\mathcal{M}\lrc{\phi_{N,\hbar},\rho,u}(t)$.

\begin{lemma}\label{lemma:estimate for error term, modulated energy}
Let $Er$ be defined as in $(\ref{equ:error term, modulated energy})$. We have
\begin{align}
&|Er|\lesssim \frac{1}{\hbar^{4}N^{\be}}.\label{equ:estimate for error term, evolution of modulated energy}
\end{align}

\end{lemma}
\begin{proof}
For $(\ref{equ:estimate for error term, evolution of modulated energy})$, we decompose
\begin{align}
Er=&\sum_{j=1}^{3}\int_{\R^{d}}u^{j}(\pa_{j}(V_{N}*\rho_{N,\hbar}))\rho_{N,\hbar}dx+\frac{b_{0}}{2}\int_{\R^{d}} \operatorname{div}u (\rho_{N,\hbar})^{2}dx\\
=&I_{1}+I_{2}\notag
\end{align}
where
\begin{align}
I_{1}=&\sum_{j=1}^{3}\int_{\R^{d}}u^{j}(\pa_{j}(V_{N}*\rho_{N,\hbar}))\rho_{N,\hbar}dx\\
&-\sum_{j=1}^{3}\frac{1}{2}\int \pa_{j}u^{j}(y)\lrc{x^{j}-y^{j}}\pa_{j}\lrc{V_{N}(x-y)}\rho_{N,\hbar}(y)\rho_{N,\hbar}(x)dxdy\notag
\end{align}
and
\begin{align}
I_{2}=&\frac{b_{0}}{2}\int_{\R^{d}} \operatorname{div}u (\rho_{N,\hbar})^{2}dx\\
&+\sum_{j=1}^{3}\frac{1}{2}\int \pa_{j}u^{j}(y)\lrc{x^{j}-y^{j}}\pa_{j}\lrc{V_{N}(x-y)}\rho_{N,\hbar}(y)\rho_{N,\hbar}(x)dxdy\notag
\end{align}
with $x=(x^{1},x^{2},x^{3})$ and $y=(y^{1},y^{2},y^{3})$.

First, we deal with $I_{1}$. Note that
\begin{align}
&\int_{\R^{d}}u^{j}(\pa_{j}(V_{N}*\rho_{N,\hbar}))\rho_{N,\hbar}dx\\
=&\int u^{j}(x)\lrs{\pa_{j}V_{N}}(x-y)\rho_{N,\hbar}(x)\rho_{N,\hbar}(y)dxdy\notag\\
=&\int u^{j}(y)\lrs{\pa_{j}V_{N}}(y-x)\rho_{N,\hbar}(y)\rho_{N,\hbar}(x)dxdy.\notag
\end{align}
By the anti-symmetry of $\pa_{j}V_{N}$,
\begin{align*}
=&-\int u^{j}(y)\lrs{\pa_{j}V_{N}}(x-y)\rho_{N,\hbar}(x)\rho_{N,\hbar}(y)dxdy.
\end{align*}
Hence we obtain
\begin{align*}
I_{1}=&\frac{1}{2}\sum_{j=1}^{3}\int (u^{j}(x)-u^{j}(y))\pa_{j}V_{N}(x-y)\rho_{N,\hbar}(y)\rho_{N,\hbar}(x)dxdy\\
&-\frac{1}{2}\sum_{j=1}^{3}\int \pa_{j}u^{j}(y)\lrc{x^{j}-y^{j}}\pa_{j}V_{N}(x-y)\rho_{N,\hbar}(y)\rho_{N,\hbar}(x)dxdy
\end{align*}

It suffices to estimate the $j=1$ case. By Taylor's expansion, we get
\begin{align}
 (u^{1}(x)-u^{1}(y))=\sum_{i=1}^{3}\pa_{i}u^{1}(y)\lrc{x^{i}-y^{i}}+\frac{1}{2}\lrs{(x-y)\cdot \nabla}^{2}u^{1}(y+\theta(x-y)),
\end{align}
so we can rewrite
\begin{align*}
I_{1}=A_{1}+A_{2}+A_{3}
\end{align*}
where
\begin{align}
A_{1}=&\frac{1}{2}\int \frac{1}{2}\lrs{(x-y)\cdot \nabla}^{2}u^{1}(y+\theta(x-y))\pa_{1}V_{N}(x-y)\rho_{N,\hbar}(y)\rho_{N,\hbar}(x)dxdy,\\
A_{2}=&\frac{1}{2}\int \pa_{2}u^{1}(y)\lrc{x^{2}-y^{2}}\pa_{1}V_{N}(x-y)\rho_{N,\hbar}(y)\rho_{N,\hbar}(x)dxdy,\\
A_{3}=&\frac{1}{2}\int \pa_{3}u^{1}(y)\lrc{x^{3}-y^{3}}\pa_{1}V_{N}(x-y)\rho_{N,\hbar}(y)\rho_{N,\hbar}(x)dxdy.
\end{align}

For $A_{1}$,
\begin{align*}
|A_{1}|\lesssim&\frac{\n{D^{2}u}_{L^{\wq}}}{N^{2\be}} \int (N^{\be}|x-y|)^{2}|\pa_{1}V_{N}(x-y)|\rho_{N,\hbar}(x)\rho_{N,\hbar}(y)dxdy
\end{align*}
By H\"{o}lder,
\begin{align*}
\lesssim& \frac{\n{D^{2}u}_{L^{\wq}}}{N^{\be}}\n{|(|x|^{2}\pa_{1}V)_{N}|*\rho_{N,\hbar}}_{L^{2}}\n{\rho_{N,\hbar}}_{L^{2}}
\end{align*}
By Young's inequality, interpolation inequality, and the energy bound for $\phi_{N,\hbar}$,
\begin{align*}
\lesssim& \frac{\n{D^{2}u}_{L^{\wq}}\n{|x|^{2}\pa_{1}V}_{L^{1}}}{\hbar^{4}N^{\be}}.
\end{align*}

For $A_{2}$,
\begin{align}
A_{2}=&\frac{1}{2}\int \pa_{2}u^{1}(y)\lrc{x^{2}-y^{2}}\pa_{1}V_{N}(x-y)\rho_{N,\hbar}(y)\rho_{N,\hbar}(x)dxdy
\end{align}
By integration by parts,
\begin{align*}
=&-\frac{1}{2}\int \pa_{2}u^{1}(y)\lrc{x^{2}-y^{2}}V_{N}(x-y)\rho_{N,\hbar}(y)\pa_{1}\rho_{N,\hbar}(x)dxdy\\
=&-\frac{1}{2N^{\be}}\int \pa_{2}u^{1}(y)\lrc{N^{\be}(x^{2}-y^{2})}V_{N}(x-y)\rho_{N,\hbar}(y)\pa_{1}\rho_{N,\hbar}(x)dxdy.
\end{align*}
So we get
\begin{align}
|A_{2}|\lesssim  \frac{\n{Du}_{L^{\wq}}}{N^{\be}}\n{\wt{V}_{N}*\rho_{N,\hbar}}_{L^{3}}\n{\pa_{1}\rho_{N,\hbar}}_{L^{3/2}}
\end{align}
where we use the notation that $\wt{V}(x)=x^{2}V(x)$.
By Young's inequality and H\"{o}lder inequality,
\begin{align*}
\lesssim \frac{\n{Du}_{L^{\wq}}}{N^{\be}}\n{\wt{V}_{N}}_{L^{1}}\n{\rho_{N,\hbar}}_{L^{3}}\n{\phi_{N,\hbar}}_{L^{6}}
\n{\nabla \phi_{N,\hbar}}_{L^{2}}
\end{align*}
By Sobolev,
\begin{align*}
\lesssim \frac{\n{Du}_{L^{\wq}}}{N^{\be}}\n{\wt{V}}_{L^{1}}\n{\phi_{N,\hbar}}_{H^{1}}^{4}
\end{align*}
By the energy bound for $\phi_{N,\hbar}$,
\begin{align*}
\lesssim \frac{\n{Du}_{L^{\wq}}}{\hbar^{4}N^{\be}}\n{\wt{V}}_{L^{1}}.
\end{align*}

For $A_{3}$, we deal with it in the same way and obtain
\begin{align}
A_{3}\lesssim \frac{\n{Du}_{L^{\wq}}}{\hbar^{4}N^{\be}}\n{\wt{V}}_{L^{1}}
\end{align}

For $I_{2}$, it suffices to treat the case $j=1$. Let
\begin{align}
\widetilde{\widetilde{V}}(x)=-x^{1}\pa_{1}V(x),
\end{align}
then we have
\begin{align}
|I_{2}|=&\frac{1}{2}\bbabs{\lra{\pa_{1}u^{1}\wt{\wt{V}}_{N}*\rho_{N,\hbar},\rho_{N,\hbar}}-b_{0}\lra{\pa_{1}u^{1}\rho_{N,\hbar},\rho_{N,\hbar}}}\\
=&\frac{1}{2}\bbabs{\lra{\pa_{1}u^{1}(\wt{\wt{V}}_{N}-b_{0}\delta)*\rho_{N,\hbar},\rho_{N,\hbar}}}\notag
\end{align}
Since $\int \wt{\wt{V}}dx=\int V dx=b_{0}$, we can repeat the proof of the approximation of identity estimate $(\ref{equ:estimate for error term, modulated energy})$ to get to
\begin{align*}
\lesssim& \n{Du}_{L^{\wq}}\n{W_{N}*\rho_{N,\hbar}}_{L^{3/2}}\n{\rho_{N,\hbar}}_{L^{3}}\\
\lesssim& \frac{1}{\hbar^{4}N^{\be}}.
\end{align*}
Putting together the estimates of $I_{1}$ and $I_{2}$ completes the proof.
\end{proof}

We can now provide a closed estimate for the modulated energy.
\begin{proposition}\label{proposition:upper bound for modulated energy}
Let $\mathcal{M}\lrc{\phi_{N,\hbar},\rho,u}(t)$ be defined as in $(\ref{equ:modulated energy})$, we have
the lower bound estimate
\begin{align}\label{equ:lower bound for modulated energy}
\mathcal{M}\lrc{\phi_{N,\hbar},\rho,u}(t)+\frac{C}{\hbar^{4}N^{\be}}\geq 0
\end{align}
and the following Gronwall's inequality
\begin{align}\label{equ:upper bound for the evolution of modulated energy}
&\frac{d}{dt}\mathcal{M}\lrc{\phi_{N,\hbar},\rho,u}(t)
\lesssim \mathcal{M}\lrc{\phi_{N,\hbar},\rho,u}(t)+\frac{1}{\hbar^{4}N^{\be}}+
\hbar^{2}.
\end{align}
Moreover, we have
\begin{align}\label{equ:upper bound for modulated energy}
\mathcal{M}\lrc{\phi_{N,\hbar},\rho,u}(t)+\frac{C}{\hbar^{4}N^{\be}}\leq \exp(CT_{0})
\lrs{\mathcal{M}\lrc{\phi_{N,\hbar},\rho,u}(0)+\frac{C}{\hbar^{4}N^{\be}}+C\hbar^{2}t}
\end{align}
and
\begin{align}
&\n{\rho_{N,\hbar}-\rho}_{L^{\infty}([0,T_{0}];L^{2}(\R^{d}))}\leq C(T_{0}) \lrs{\frac{1}{\hbar^{4}N^{\be}}+\hbar^{2}}^{1/2},\label{equ:quantitative estimate, density}\\
&\n{(i\hbar\nabla-u)\phi_{N,\hbar}}_{L^{\infty}([0,T_{0}];L^{2}(\R^{d}))}\leq C(T_{0}) \lrs{\frac{1}{\hbar^{4}N^{\be}}+\hbar^{2}}^{1/2}.\label{equ:quantitative estimate, moment}
\end{align}
\end{proposition}
\begin{proof}
For $(\ref{equ:lower bound for modulated energy})$, we rewrite
\begin{align}
&\mathcal{M}\lrc{\phi_{N,\hbar},\rho,u}(t)\\
=&\frac{1}{2}\int_{\R^{d}}|(i\hbar\nabla-u)\phi_{N,\hbar}(t)|^{2}dx+\frac{b_{0}}{2}\int\lrs{ \rho_{N,\hbar}-\rho}^{2}dx+\frac{1}{2}\lra{W_{N}*\rho_{N,\hbar},\rho_{N,\hbar}},\notag
\end{align}
where $W_{N}=V_{N}-b_{0}\delta$.
By estimate $(\ref{equ:estimate for error term, modulated energy})$, we arrive at
\begin{align}
\mathcal{M}\lrc{\phi_{N,\hbar},\rho,u}(t)\gtrsim -\frac{1}{\hbar^{4}N^{\be}},
\end{align}
which completes the proof of $(\ref{equ:lower bound for modulated energy})$.

For $(\ref{equ:upper bound for the evolution of modulated energy})$, we make use of Proposition \ref{proposition:the evolution of modulated energy} to obtain\footnote{The regularity requirement that $s>\frac{d}{2}+3$ comes from $\n{\Delta \operatorname{div} u}_{L^{\infty}}$, the second term on the right side of $(\ref{equ:upper bound,modulated energy,regularity})$. One can reduce one derivative in requirement $(d)$ of Theorem \ref{thm:main theorem,bbgky-to-euler} by integration by parts at the price of weakening the convergence rate.}
\begin{align}
&\frac{d}{dt}\mathcal{M}\lrc{\phi_{N,\hbar},\rho,u}(t)\notag\\
=&-\int_{\R^{d}}\pa_{k}u^{j}\operatorname{Re}
\lrs{(\hbar\pa_{k}-iu^{k})\phi_{N,\hbar}\ol{\lrs{\hbar\pa_{j}-iu^{j}}\phi_{N,\hbar}}}\notag\\
&-\frac{b_{0}}{2}\int_{\R^{d}}\operatorname{div}u(\rho_{N,\hbar}-\rho)^{2}dx-\frac{\hbar^{2}}{4}\int_{\R^{d}}
 \rho_{N,\hbar}\lrs{\Delta \operatorname{div} u}dx+Er\notag\\
\lesssim &\n{Du}_{L^{\wq}}
\lrs{\int_{\R^{d}}|(i\hbar\nabla-u)\phi_{N,\hbar}(t)|^{2}dx+b_{0}\int\lrs{ \rho_{N,\hbar}-\rho}^{2}dx}\label{equ:upper bound,modulated energy,regularity}\\
&+\hbar^{2}\n{\rho_{N,\hbar}}_{L^{1}}\n{\Delta \operatorname{div} u}_{L^{\wq}}
+|Er|\notag
\end{align}
By the error term estimate $(\ref{equ:estimate for error term, evolution of modulated energy})$, we reach
\begin{align}
\frac{d}{dt}\mathcal{M}\lrc{\phi_{N,\hbar},\rho,u}(t)\lesssim\mathcal{M}\lrc{\phi_{N,\hbar},\rho,u}(t)+ \hbar^{2}+\frac{1}{\hbar^{4}N^{\be}}.
\end{align}
which completes the proof of $(\ref{equ:upper bound for the evolution of modulated energy})$.

Combining $(\ref{equ:lower bound for modulated energy})$ and $(\ref{equ:upper bound for the evolution of modulated energy})$, we have
\begin{align}
&\mathcal{M}\lrc{\phi_{N,\hbar},\rho,u}(t)+\frac{C}{\hbar^{4}N^{\be}}\\
=&\mathcal{M}\lrc{\phi_{N,\hbar},\rho,u}(0)+\frac{C}{\hbar^{4}N^{\be}}+
\int_{0}^{t}\frac{d}{d\tau}\lrs{\mathcal{M}\lrc{\phi_{N,\hbar},\rho,u}(\tau)+\frac{C}{\hbar^{4}N^{\be}}}d\tau\notag\\
\leq& \mathcal{M}\lrc{\phi_{N,\hbar},\rho,u}(0)+\frac{C}{\hbar^{4}N^{\be}}+
C\int_{0}^{t}\mathcal{M}\lrc{\phi_{N,\hbar},\rho,u}(\tau)+\frac{C}{\hbar^{4}N^{\be}}+
\hbar^{2}d\tau\notag\\
=&\lrs{\mathcal{M}\lrc{\phi_{N,\hbar},\rho,u}(0)+\frac{C}{\hbar^{4}N^{\be}}+C\hbar^{2}t}+
C\int_{0}^{t}\mathcal{M}\lrc{\phi_{N,\hbar},\rho,u}(\tau)+\frac{C}{\hbar^{4}N^{\be}}d\tau.\notag
\end{align}
Then by Gronwall's inequality, we obtain estimate $(\ref{equ:upper bound for modulated energy})$.

Finally, we deal with $(\ref{equ:quantitative estimate, density})$ and $(\ref{equ:quantitative estimate, moment})$. By error estimate $(\ref{equ:estimate for error term, modulated energy})$, we note that
\begin{align*}
&\int_{\R^{d}}|(i\hbar\nabla-u)\phi_{N,\hbar}(t)|^{2}dx+b_{0}\int\lrs{ \rho_{N,\hbar}-\rho}^{2}dx
\lesssim \mathcal{M}\lrc{\phi_{N,\hbar},\rho,u}(t)+\frac{1}{\hbar^{4}N^{\be}},\\
&\mathcal{M}\lrc{\phi_{N,\hbar},\rho,u}(0)\lesssim
\int_{\R^{d}}|(i\hbar\nabla-u^{in})\phi_{N,\hbar}^{in}|^{2}dx+b_{0}\int_{\R^{d}}\lrs{\rho_{N,\hbar}^{in}-\rho^{in}}^{2}dx+
\frac{1}{\hbar^{4}N^{\be}}.
\end{align*}
Hence, we can appeal to estimate $(\ref{equ:upper bound for modulated energy})$ and the initial condition $(\ref{equ:modulated energy, initial convergence rate})$ to get
\begin{align}
&\int_{\R^{d}}|(i\hbar\nabla-u)\phi_{N,\hbar}(t)|^{2}dx+b_{0}\int\lrs{ \rho_{N,\hbar}-\rho}^{2}dx\\
\leq &C\lrs{\mathcal{M}\lrc{\phi_{N,\hbar},\rho,u}(t)+\frac{1}{\hbar^{4}N^{\be}}}\notag \\
\leq &C(T_{0})\lrs{\hbar^{2}+\frac{1}{\hbar^{4}N^{\be}}}.\notag
\end{align}
This completes the proof of estimates $(\ref{equ:quantitative estimate, density})$ and $(\ref{equ:quantitative estimate, moment})$.
\end{proof}

\appendix

\section{Miscellaneous Lemmas}
\subsection{Collapsing Estimate and Strichartz Estimates}
\begin{lemma}[\cite{CP14,Che13,KM08}, KM Collapsing Estimate]\footnote{See also \cite{CP11,Che12,GSS14,
HS16,KSS11,Xie15} for many different versions of estimates of this type.}
There is a $C$ independent of $V$, $j$, $k$ and $N$ such that,
\begin{align}\label{equ:collapsing estimate d=3, propogator version,h=1}
\n{S^{(1,k)}B_{N,j,k+1}^{\pm}U^{(k+1)}f^{(k+1)}}_{L_{t}^{2}L_{x,x'}^{2}}\leq
C\n{V}_{L^{1}}\n{S^{(1,k+1)}f^{(k+1)}}_{L_{x,x'}^{2}}.
\end{align}
where $f^{(k+1)}(\textbf{x}_{k+1};\textbf{x}_{k+1}')$ is independent of $t$.
\end{lemma}
\begin{lemma} \label{lemma:collapsing estimate d=3, propogator version}
Let $d\leq 3$ and $\al=d+1/2$. Then we have
\begin{align}
\n{S_{\hbar}^{(1,k)}B_{N,\hbar,j,k+1}^{\pm}U_{\hbar}^{(k+1)}f^{(k+1)}}_{L_{t}^{2}L_{x,x'}^{2}}\leq
\frac{C\n{V}_{L^{1}}}{h^{\al}}\n{S_{\hbar}^{(1,k+1)}f^{(k+1)}}_{L_{x,x'}^{2}}
\end{align}
\end{lemma}
\begin{proof}
Let us define
\begin{align}
\lrs{\delta_{x}^{a}f}(x)=f(ax),\quad \lrs{\delta_{t}^{a}f}(t)=f(at).
\end{align}
By scaling,
\begin{align}
&\bbn{S_{\hbar}^{(1,k)}\operatorname{Tr}_{k+1}\lrs{
V_{N,\hbar}(x_{j}-x_{k+1})U_{\hbar}^{(k+1)}(t)f^{(k+1)}}}_{L_{t}^{2}L_{x,x'}^{2}}\\
=&\hbar^{kd+\frac{1}{2}}\bbn{\delta_{t}^{\hbar}\delta_{x}^{\hbar}\lrc{S_{\hbar}^{(1,k)}\operatorname{Tr}_{k+1}\lrs{
V_{N,\hbar}(x_{j}-x_{k+1})U_{\hbar}^{(k+1)}(t)f^{(k+1)}}}}_{L_{t}^{2}L_{x,x'}^{2}}\notag
\end{align}
Noting that $V_{N,\hbar}$ carries $\hbar^{-1}$,
\begin{align*}
=&\hbar^{kd-\frac{1}{2}}\bbn{S^{(1,k)}
\operatorname{Tr}_{k+1}\lrs{\hbar^{d}V_{N}(\hbar(x_{j}-x_{k+1}))U^{(k+1)}(t)
(\delta_{x}^{\hbar}\lrc{f^{(k+1)}}}}_{L_{t}^{2}L_{x,x'}^{2}}
\end{align*}
By estimate $(\ref{equ:collapsing estimate d=3, propogator version,h=1})$,
\begin{align*}
\leq&\hbar^{kd-\frac{1}{2}}C\n{\hbar^{d}V_{N}(\hbar x)}_{L^{1}}\n{S^{(1,k+1)}\delta_{x}^{\hbar}\lrc{f^{(k+1)}}}_{L_{x,x'}^{2}}\\
=&\frac{C\n{V}_{L^{1}}}{\hbar^{d+\frac{1}{2}}}
\bbn{S_{\hbar}^{(1,k+1)}f^{(k+1)}}_{L_{x,x'}^{2}}
\end{align*}
which completes the proof.
\end{proof}

\begin{lemma}[\cite{CH16on}, Lemmas 4.1, 4.3, and 4.6]\footnote{These are $X_{s,b}$ estimates in disguise. As we are not using the $X_{s,b}$ spaces directly in this paper, we will not go into the details.}
Let $\theta$ and $\wt{\theta}$ are cutoff functions supported $[-1,1]$ and $\wt{\theta}_{T}(t)=\wt{\theta}(t/T)$.
For the case $\hbar=1$,
we have
\begin{align}\label{equ:collapsing estimate for dp,last step,scaling,j=0,h=1}
&\bbn{S^{(1,k)}\theta(t_{k})\int_{0}^{t_{k}}U^{(k)}(t_{k}-t_{k+1})V_{N}(x_{1}-x_{2})\wt{\theta}_{T}(t_{k+1})\ga_{N}^{(k)}(t_{k+1})
dt_{k+1}}_{L_{t_{k}}^{\infty}L_{x,x'}^{2}}\\
\leq&N^{\frac{5}{2}\be}C_{V}C_{\theta}\bbn{S^{(1,k)}\wt{\theta}_{T}(t_{k+1})\ga_{N}^{(k)}
(t_{k+1})}_{L_{t_{k+1}}^{2}L_{x,x'}^{2}}\notag
\end{align}
and
\begin{align}\label{equ:collapsing estimate for dp,last step,scaling,j>=1,h=1}
&\bbn{S^{(1,k+j-1)}B_{N,1,k+j}\theta(t_{k+j})\int_{0}^{t_{k+j}}U^{(k+j)}(t_{k+j}-t_{k+j+1})V_{N,12}
\wt{\theta}_{T}\ga_{N}^{(k+j)}(t_{k+j+1})dt_{k+j+1}}_{L_{t_{k+j}}^{2}L_{x,x'}^{2}}\\
\leq& N^{\frac{5}{2}\be}C_{V}C_{\theta}\bbn{S^{(1,k+j)}
\wt{\theta}_{T}(t_{k+j+1})\ga_{N}^{(k+j)}(t_{k+j+1})}_{L_{t_{k+j+1}}^{2}L_{x,x'}^{2}}\notag
\end{align}
where
$V_{N,12}=N^{d\be}V(N^{\be}(x_{1}-x_{2}))$ and
\begin{align*}
C_{\theta}=|Supp(\theta)|\lrs{\n{\theta}_{L_{t}^{2}}+
\n{\theta'}_{L_{t}^{\frac{4}{3}}}}+\n{\theta}_{L_{t}^{\frac{4}{3}}}
+\n{\lra{\nabla_{t}}^{\frac{3}{4}}\theta}_{L_{t}^{2}}+\n{\theta}_{L_{t}^{\infty}}
\end{align*}
with $|Supp(\theta)|$ denoting the Lebesgue measure of the support of $\theta$.

\end{lemma}

\begin{lemma}\label{lemma:collapsing estimate for dp,last step,scaling}
For $j\geq 0$ and $k\geq 1$, we have the following estimates
\begin{align} \label{equ:collapsing estimate for dp,last step,scaling,j=0}
&\bbn{S_{\hbar}^{(1,k)}\int_{0}^{t_{k}}U_{\hbar}^{(k)}(t_{k}-t_{k+1})V_{N,\hbar}^{(k)}\ga_{N,\hbar}^{(k)}(t_{k+1})
dt_{k+1}}_{L_{t_{k}}^{\infty}[0,T]L_{x,x'}^{2}}\\
\leq& N^{\frac{5}{2}\be-1}\hbar(C_{V}\hbar^{-\al}T^{1/2})k^{2}\n{S_{\hbar}^{(1,k)}
\ga_{N,\hbar}^{(k)}(t_{k+1})}_{L_{t_{k+1}}^{\wq}L_{x,x'}^{2}}\notag
\end{align}
and
\begin{align}\label{equ:collapsing estimate for dp,last step,scaling,j>=0}
&\bbn{S_{\hbar}^{(1,k+j-1)}B_{N,\hbar,1,k+j}\int_{0}^{t_{k+j}}
U_{\hbar}^{(k+j)}(t_{k+j}-t_{k+j+1})V_{N,\hbar}^{(k+j)}
\ga_{N,\hbar}^{(k+j)}(t_{k+j+1})dt_{k+j+1}}_{L_{t_{k+j}}^{1}[0,T]L_{x,x'}^{2}}\\
\leq&N^{\frac{5}{2}\be-1}\hbar(C_{V}\hbar^{-\al}T^{1/2})^{2}(k+j)^{2}\bbn{S_{\hbar}^{(1,k+j)}
\ga_{N,\hbar}^{(k+j)}(t_{k+j+1})}_{L_{t_{k+j+1}}^{\infty}L_{x,x'}^{2}}.\notag
\end{align}
\end{lemma}
\begin{proof}
For $(\ref{equ:collapsing estimate for dp,last step,scaling,j=0})$, we have
\begin{align}
&\bbn{S_{\hbar}^{(1,k)}\int_{0}^{t_{k}}U_{\hbar}^{(k)}(t_{k}-t_{k+1})V_{N,\hbar}^{(k)}\ga_{N,\hbar}^{(k)}(t_{k+1})
dt_{k+1}}_{L_{t_{k}}^{\infty}[0,T]L_{x,x'}^{2}}\\
\leq&\bbn{S_{\hbar}^{(1,k)}\theta(t_{k})\int_{0}^{t_{k}}U_{\hbar}^{(k)}(t_{k}-t_{k+1})V_{N,12}\wt{\theta}_{T}(t_{k+1})\ga_{N,\hbar}^{(k)}(t_{k+1})
dt_{k+1}}_{L_{t_{k}}^{\infty}L_{x,x'}^{2}},\notag
\end{align}
where $\theta$ and $\wt{\theta}$ are cutoff functions supported $[-1,1]$ and $\wt{\theta}_{T}(t)=\wt{\theta}(t/T)$. For simplicity, we set
$$V_{N\hbar,12}=(N^{\be}\hbar)^{d}V(N^{\be}\hbar(x_{1}-x_{2})).$$
Then by scaling argument, we arrive at
\begin{align}
&\bbn{S_{\hbar}^{(1,k)}\theta(t_{k})\int_{0}^{t_{k}}U_{\hbar}^{(k)}(t_{k}-t_{k+1})V_{N,12}\wt{\theta}_{T}(t_{k+1})\ga_{N,\hbar}^{(k)}(t_{k+1})
dt_{k+1}}_{L_{t_{k}}^{\infty}L_{x,x'}^{2}}\\
=&\hbar^{kd}\bbn{\delta_{x}^{\hbar}\delta_{t}^{\hbar}\lrc{S_{\hbar}^{(1,k)}\theta(t_{k})\int_{0}^{t_{k}}
U_{\hbar}^{(k)}(t_{k}-t_{k+1})V_{N,12}\wt{\theta}_{T}(t_{k+1})\ga_{N,\hbar}^{(k)}(t_{k+1})
dt_{k+1}}}_{L_{t_{k}}^{\infty}L_{x,x'}^{2}}\notag\\
=&\hbar\hbar^{kd}\bbn{S^{(1,k)}\lrs{\delta_{t}^{\hbar}\theta}(t_{k})\int_{0}^{t_{k}}U^{(k)}(t_{k}-t_{k+1})
\delta_{x}^{\hbar}\lrc{V_{N,12}
\delta_{t}^{\hbar}\lrc{\wt{\theta}_{T}(t_{k+1})\ga_{N,\hbar}^{(k)}(t_{k+1})}}
dt_{k+1}}_{L_{t_{k}}^{\infty}L_{x,x'}^{2}}\notag\\
=&\frac{\hbar\hbar^{kd}}{\hbar^{d}}\bbn{S^{(1,k)}\lrs{\delta_{t}^{\hbar}\theta}(t_{k})\int_{0}^{t_{k}}U^{(k)}(t_{k}-t_{k+1})
V_{N\hbar,12}
\delta_{x}^{\hbar}\delta_{t}^{\hbar}\lrc{\wt{\theta}_{T}(t_{k+1})\ga_{N,\hbar}^{(k)}(t_{k+1})}
dt_{k+1}}_{L_{t_{k}}^{\infty}L_{x,x'}^{2}}\notag
\end{align}
By using estimate $(\ref{equ:collapsing estimate for dp,last step,scaling,j=0,h=1})$,
\begin{align*}
\leq& \frac{\hbar\hbar^{kd}(N^{\be}\hbar)^{\frac{5}{2}}C_{V}C_{\delta_{t}^{\hbar}\theta}}{\hbar^{d}}\bbn{S^{(1,k)}
\delta_{x}^{\hbar}\delta_{t}^{\hbar}\lrc{\wt{\theta}_{T}(t_{k+1})\ga_{N,\hbar}^{(k)}
(t_{k+1})}}_{L_{t_{k+1}}^{2}L_{x,x'}^{2}}\\
=& \frac{\hbar(N^{\be}\hbar)^{\frac{5}{2}}C_{V}C_{\delta_{t}^{\hbar}\theta}}{\hbar^{d}\hbar^{1/2}}\bbn{S^{(1,k)}
\wt{\theta}_{T}(t_{k+1})\ga_{N,\hbar}^{(k)}
(t_{k+1})}_{L_{t_{k+1}}^{2}L_{x,x'}^{2}}
\end{align*}
By taking $L^{\infty}$ at $dt_{k+1}$ and using the estimate that $\hbar^{\frac{3}{2}} C_{\delta_{t}^{\hbar}\theta}\leq C$,
\begin{align*}
\leq & \frac{N^{\frac{5}{2}\be}\hbar^{2}C_{V}T^{1/2}}{\hbar^{d}\hbar^{1/2}}\bbn{S^{(1,k)}
\ga_{N,\hbar}^{(k)}
(t_{k+1})}_{L_{t_{k+1}}^{\infty}L_{x,x'}^{2}}.
\end{align*}
We note that the $N^{-1}$, $k^{2}$ and $\hbar^{-1}$ factors come from the expansion of $V_{N,\hbar}^{(k)}$ and then arrive at $(\ref{equ:collapsing estimate for dp,last step,scaling,j=0})$.

Next, we deal with $(\ref{equ:collapsing estimate for dp,last step,scaling,j>=0})$. With the help of estimate $(\ref{equ:collapsing estimate for dp,last step,scaling,j>=1,h=1})$, we can use scaling argument in the same way as above to arrive at $(\ref{equ:collapsing estimate for dp,last step,scaling,j>=0})$, where the $N^{-1}$,  $(k+j)^{2}$, and $\hbar^{-1}$ factors come from the expansion of $V_{N,\hbar}^{(k+j)}$ and another $\hbar^{-1}$ factor comes from $B_{N,\hbar,1,k+j}$.

\end{proof}

\begin{lemma}\label{lemma:collapsing estimate for ep,ip,last step}
For $j\geq 0$ and $k\geq 1$, we have
\begin{align}\label{equ:collapsing estimate for ep,ip,last step}
&\int_{[0,T]}\n{S_{\hbar}^{(1,k+j)}B_{N,\hbar,1,k+j+1}\ga_{N,\hbar}^{(k+j+1)}(t_{k+j+1})}_{L_{x,x'}^{2}}dt_{k+j+1}\\
 \leq &N^{\frac{5}{2}\be}\hbar^{2}T^{1/2}(C_{V}\hbar^{-\al}T^{1/2})\bbn{S_{\hbar}^{(1,k+j+1)}
\ga_{N,\hbar}^{(k+j+1)}(t_{k+j+1})}_{L_{t_{k+j+1}}^{\infty}L_{x,x'}^{2}}.\notag
\end{align}
\end{lemma}
\begin{proof}
By taking $L^{\infty}$ at $dt_{k+j+1}$, it suffices to prove that
\begin{align}\label{equ:collapsing estimate for ep,ip,last step,L infty}
\n{S_{\hbar}^{(1,k+j)}B_{N,\hbar,1,k+j+1}\ga_{N,\hbar}^{(k+j+1)}(t_{k+j+1})}_{L_{x,x'}^{2}}
 \leq &N^{\frac{5}{2}\be}\hbar^{2}C_{V}\hbar^{-\al}\bbn{S_{\hbar}^{(1,k+j+1)}
\ga_{N,\hbar}^{(k+j+1)}}_{L_{x,x'}^{2}}.
\end{align}
For $\hbar =1$, we have
\begin{align}
\n{S^{(1,k+j)}B_{N,1,k+j+1}\ga_{N}^{(k+j+1)}(t_{k+j+1})}_{L_{x,x'}^{2}}
 \leq &N^{\frac{5}{2}\be}C_{V}\bbn{S^{(1,k+j+1)}
\ga_{N}^{(k+j+1)}}_{L_{x,x'}^{2}}.
\end{align}
By scaling, we arrive at $(\ref{equ:collapsing estimate for ep,ip,last step,L infty})$.
\end{proof}
\subsection{Convolution and Commutator Estimates}
\begin{lemma}[\cite{CH21quantitative}, Lemma A.5]\label{lemma:quantative estimate for identity approximation}
Let $W_{N}(x)=N^{d\be}V(N^{\be}x)-b_{0}\delta$, where $b_{0}=\int V(x)dx$. For any $0\leq s\leq 1$,
\begin{align}
\n{W_{N}*f}_{L^{p}}\lesssim N^{-\be s}\n{\lra{\nabla}^{s}f}_{L^{p}}
\end{align}
for any $1<p<\wq$. The implicit constant depends only on $\n{\lra{x}V(x)}_{L^{1}}$.
\end{lemma}

\begin{lemma}[Fractional Leibniz Rule]\label{lemma:leibniz rule}
\begin{align}
\n{\lra{\nabla}^{s}(fg)}_{L^{r}}\lesssim \n{\lra{\nabla}^{s}f}_{L^{p_{1}}}\n{g}_{L^{p_{2}}}+
\n{f}_{L^{q_{1}}}\n{\lra{\nabla}^{s}g}_{L^{q_{2}}}
\end{align}
where
\begin{align}
\frac{1}{r}=\frac{1}{p_{1}}+\frac{1}{p_{2}}=\frac{1}{q_{1}}+\frac{1}{q_{2}}
\end{align}
$r\in [1,\wq)$, $p_{i}$, $q_{i}\in (1,\infty]$, $s>0$.
\end{lemma}

\begin{lemma}[\cite{CH19,ESY07}]\label{lemma:Sobolev type estimate, 3d}
Let $d=3$, $\eta>d/4$ and $V_{N}(x)=N^{3\be}V(N^{\be}x)$. Then
\begin{align}
&V_{N}(x_{1}-x_{2})\leq C(\eta)\n{V}_{L^{1}}(1-\Delta_{x_{1}})^{\eta}
(1-\Delta_{x_{2}})^{\eta},\label{equ:Sobolev type estimate, 3d, L1}\\
&V_{N}(x_{1}-x_{2})\leq C N^{\be}\n{V}_{L^{3/2}}(1-\Delta_{x_{1}}),\label{equ:Sobolev type estimate, 3d, L2}\\
&V_{N}(x_{1}-x_{2})\leq C N^{3\be}\n{V}_{L^{\wq}}.\label{equ:Sobolev type estimate, 3d, Linffty}
\end{align}
\end{lemma}
\begin{proof}
For $(\ref{equ:Sobolev type estimate, 3d, L1})$ with $\eta=1$, $(\ref{equ:Sobolev type estimate, 3d, L2})$ and $(\ref{equ:Sobolev type estimate, 3d, Linffty})$, see \cite[Lemma A.3]{ESY07}. For $(\ref{equ:Sobolev type estimate, 3d, L1})$ with $\eta>3/4$, see \cite{CH19}. For completeness, we show a proof of $(\ref{equ:Sobolev type estimate, 3d, L1})$.
By direct calculation,
\begin{align*}
&\lra{\psi,V(x_{1}-x_{2})\psi}\\
=&\int dpd\xi_{1}d\xi_{2} \widehat{\psi}(\xi_{1},\xi_{2})\widehat{V}(p)\ol{\widehat{\psi}}(\xi_{1}+p,\xi_{2}-p)\\
=&\int dpd\xi_{1}d\xi_{2} \frac{\lra{\xi_{1}}^{\eta}\lra{\xi_{2}}^{\eta}\lra{\xi_{1}+p}^{\eta}\lra{\xi_{2}-p}^{\eta}}
{\lra{\xi_{1}}^{\eta}\lra{\xi_{2}}^{\eta}\lra{\xi_{1}+p}^{\eta}\lra{\xi_{2}-p}^{\eta}}
\widehat{\psi}(\xi_{1},\xi_{2})\widehat{V}(p)\ol{\widehat{\psi}}(\xi_{1}+p,\xi_{2}-p)
\end{align*}
By H\"{o}lder,
\begin{align*}
\leq& \n{\widehat{V}}_{L^{\wq}}\n{\lra{\xi_{1}}^{\eta}\lra{\xi_{2}}^{\eta}\widehat{\psi}(\xi_{1},\xi_{2})}_{L_{p}^{\infty}
L_{\xi_{1},\xi_{2}}^{2}}
\n{\lra{\xi_{1}+p}^{\eta}\lra{\xi_{2}-p}^{\eta}\ol{\widehat{\psi}}(\xi_{1}+p,\xi_{2}-p)}_{
L_{p}^{\infty}L_{\xi_{1},\xi_{2}}^{2}}\\
&\bbn{\frac{1}
{\lra{\xi_{1}}^{\eta}\lra{\xi_{2}}^{\eta}\lra{\xi_{1}+p}^{\eta}\lra{\xi_{2}-p}^{\eta}}}_{L_{p}^{1}L_{\xi_{1}}^{\infty}
L_{\xi_{2}}^{\infty}}\\
\leq &C(\eta)\n{V}_{L^{1}}\n{\lra{\nabla_{x_{1}}}^{\eta}\lra{\nabla_{x_{2}}}^{\eta}\psi}_{L^{2}}^{2}
\end{align*}
where in the last inequality we used
\begin{align*}
\bbn{\frac{1}
{\lra{\xi_{1}}^{\eta}\lra{\xi_{2}}^{\eta}\lra{\xi_{1}+p}^{\eta}\lra{\xi_{2}-p}^{\eta}}}_{L_{p}^{1}L_{\xi_{1}}^{\infty}
L_{\xi_{2}}^{\infty}}
\lesssim &\int \frac{1}{\lra{p}^{4\eta}} dp\leq C(\eta).
\end{align*}
\end{proof}

\section{Energy Estimate}
Recall the Hamiltonian $(\ref{equ:hamiltonian})$
\begin{align*}
H_{N,\hbar}=\sum_{j=1}^{N}-\frac{1}{2}\hbar^{2}\Delta_{x_{j}}+\frac{1}{N}\sum_{1\leq j<k\leq N}V_{N}(x_{j}-x_{k})
\end{align*}
and the derivative involving $\hbar$ in $(\ref{equ:kinetic operator,h})$
\begin{align*}
S_{\hbar,j}^{2}=1-\frac{\hbar^{2}}{2}\Delta_{x_{j}}.
\end{align*}

\begin{proposition}\label{lemma:energy estimate}
Let $\be<\frac{3}{5}$, $k\leq (\ln N)^{100}$ and $\hbar^{-1}\leq \ln N$\footnote{The restriction that $\hbar^{-1}\leq \ln N$ is not necessary and it can be removed at the price of reducing down the parameter $\be$.}.
There exists $N_{0}(\be)$ independent of $k$ and $\hbar$, such that
\begin{align}
\lra{\psi,(H_{N,\hbar}+N)^{k}\psi}\geq \frac{N^{k}}{2^{k}}\lra{\psi,S_{\hbar,1}^{2}S_{\hbar,2}^{2}\ccc S_{\hbar,k}^{2}\psi}.
\end{align}
for every $N\geq N_{0}(\be)$.
\end{proposition}
\begin{proof}
This proof has been done by many authors in many work. We include one here solely for completeness purposes.
For $k=0$ and $k=1$, the claim is trivial because of the positivity of the potential.
Now we assume the proposition is true for all $k\leq n$, and we prove it for $k=n+2$.
\begin{align}\label{equ:energy estimate, induction n}
\lra{\psi,(H_{N,\hbar}+N)^{n+2}\psi}
=&\lra{\lrs{H_{N,\hbar}+N}\psi,\lrs{H_{N,\hbar}+N}^{n}\lrs{H_{N,\hbar}+N}\psi}\\
\geq&\frac{N^{n}}{2^{n}}\lra{\psi,
(H_{N,\hbar}+N)S_{\hbar,1}^{2}\ccc S_{\hbar,n}^{2}(H_{N,\hbar}+N)\psi}.\notag
\end{align}
We set
\begin{align*}
H_{N,\hbar}^{(n)}=\sum_{j=1}^{n}S_{\hbar,j}^{2}+\frac{1}{n}\sum_{j<m}^{n}V_{jm}
\end{align*}
with $V_{jm}=N^{3\be}V(N^{\be}(x_{j}-x_{k}))$. Then we have
\begin{align*}
&\lra{\psi,
(H_{N,\hbar}+N)S_{\hbar,1}^{2}\ccc S_{\hbar,n}^{2}(H_{N,\hbar}+N)\psi}\\
=&\sum_{j_{1},j_{2}\geq n+1}\lra{\psi,S_{\hbar,j_{1}}^{2}S_{\hbar,1}^{2}\ccc
S_{\hbar,n}^{2}S_{\hbar,j_{2}}^{2}\psi}\\
&+\sum_{j\geq n+1}\lrs{\lra{\psi,S_{\hbar,j}^{2}S_{\hbar,1}^{2}\ccc S_{\hbar,n}^{2}
H_{N,\hbar}^{(n)}\psi}+c.c.}+\lra{\psi,H_{N,\hbar}^{(n)}S_{\hbar,1}^{2}\ccc S_{\hbar,n}^{2}H_{N,\hbar}^{(n)}\psi}.
\end{align*}
where $c.c.$ denotes the complex conjugate.
Since $H_{N,\hbar}^{(n)}S_{\hbar,1}^{2}\ccc S_{\hbar,n}^{2}H_{N,\hbar}^{(n)}\geq 0$, we have, using the symmetry
with respect to permutations,
\begin{align}\label{equ:expansion, induction n+2}
&\lra{\psi,
(H_{N,\hbar}+N)S_{\hbar,1}^{2}\ccc S_{\hbar,n}^{2}(H_{N,\hbar}+N)\psi}\\
\geq& (N-n)(N-n-1)\lra{\psi,S_{\hbar,1}^{2}S_{\hbar,2}^{2}\ccc S_{\hbar,n+2}^{2}\psi}\notag\\
&+(2n+1)(N-n)\lra{\psi,S_{\hbar,1}^{4}S_{\hbar,2}^{2}\ccc S_{\hbar,n+1}^{2}\psi}\notag\\
&+\frac{n(n+1)(N-n)}{2N}\lrs{
\lra{\psi,V_{12}S_{\hbar,1}^{2}S_{\hbar,2}^{2}\ccc S_{\hbar,n+1}^{2}\psi}+c.c.}\notag\\
&+\frac{(n+1)(N-n)(N-n-1)}{N}
\lrs{\lra{\psi,V_{1,n+2}S_{\hbar,1}^{2}S_{\hbar,2}^{2}\ccc S_{\hbar,n+1}^{2}\psi}+c.c.}\notag
\end{align}
Here we also used the fact that
\begin{align*}
\lra{\psi,V_{jm}S_{\hbar,1}^{2},\ccc S_{\hbar,n+1}^{2}\psi}
\geq 0
\end{align*}
if $j$, $m>n+1$, because of the positivity of the potential. Next, we will bound the last two terms on the r.h.s of $(\ref{equ:expansion, induction n+2})$ from below, so we might as well set $S_{\hbar,j}^{2}=1-\hbar^{2}\Delta_{x_{j}}$ for simplicity. Then we have
\begin{align*}
&\lra{\psi,V_{12}S_{\hbar,1}^{2}S_{\hbar,2}^{2}\ccc S_{\hbar,n+1}^{2}\psi}+c.c.\\
=&\lra{\psi,V_{12}(1-\hbar^{2}\Delta_{x_{1}})(1-\hbar^{2}\Delta_{x_{2}})S_{\hbar,3}^{2}\ccc S_{\hbar,n+1}^{2}\psi}+c.c.\\
\geq &\lra{\psi,\hbar\nabla V_{12}\hbar\nabla_{x_{2}}S_{\hbar,3}^{2}\ccc S_{\hbar,n+1}^{2}\psi}+c.c.\\
&+\lra{\hbar\nabla_{x_{2}}\psi,\hbar\nabla V_{12}\hbar\nabla_{x_{1}}\hbar\nabla_{x_{2}}S_{\hbar,3}^{2}\ccc S_{\hbar,n+1}^{2}\psi}+c.c.\\
&+\lra{\psi,\hbar\nabla V_{12}\hbar^{2}\Delta_{x_{1}}\hbar\nabla_{x_{2}}S_{\hbar,3}^{2}\ccc S_{\hbar,n+1}^{2}\psi}+c.c.\\
:=&I+II+III
\end{align*}
where $\nabla V_{12}=N^{4\be}\lrs{\nabla V}(N^{\be}(x_{1}-x_{2}))$.
Applying Cauchy-Schwarz, we get
\begin{align*}
I\geq& -2\left\{\al_{1}\lra{\psi,|\hbar\nabla V_{12}|S_{\hbar,3}^{2}\ccc S_{\hbar,n+1}^{2}\psi}\right.\\
&\left.+\al_{1}^{-1}
\lra{|\hbar\nabla_{x_{2}}|\psi,|\hbar\nabla V_{12}|S_{\hbar,3}^{2}\ccc S_{\hbar,n+1}^{2}|\hbar\nabla_{x_{2}}|\psi}\right\},
\end{align*}
\begin{align*}
II\geq& -2\left\{\al_{2}\lra{|\hbar\nabla_{x_{2}}|\psi,|\hbar\nabla V_{12}|
S_{\hbar,3}^{2}\ccc S_{\hbar,n+1}^{2}|\hbar\nabla_{x_{2}}|\psi} \right.\\
 &\left. +\al_{2}^{-1}\lra{|\hbar\nabla_{x_{1}}||\hbar\nabla_{x_{2}}|\psi,|\hbar\nabla V_{12}|
S_{\hbar,3}^{2}\ccc S_{\hbar,n+1}^{2}|\hbar\nabla_{x_{1}}||\hbar\nabla_{x_{2}}|\psi}\right\},
\end{align*}
\begin{align*}
III\geq &-2\left\{\al_{3}\lra{\psi,|\hbar\nabla V_{12}|S_{\hbar,3}^{2}\ccc S_{\hbar,n+1}^{2}\psi}\right.\\
&\left.+\al_{3}^{-1}
\lra{|\hbar\nabla_{x_{1}}|^{2}|\hbar\nabla_{x_{2}}|\psi,|\hbar\nabla V_{12}|S_{\hbar,3}^{2}\ccc S_{\hbar,n+1}^{2}|\hbar\nabla_{x_{1}}|^{2}|\hbar\nabla_{x_{2}}|\psi}\right\}.
\end{align*}

By Lemma \ref{lemma:Sobolev type estimate, 3d},
\begin{align*}
&I\geq -C \lr{\al_{1}N^{\be}\hbar^{-3}\lra{\psi,S_{\hbar,1}^{2}\ccc S_{\hbar,n+1}^{2}\psi}
+\al_{1}^{-1}N^{4\be}\hbar\lra{\psi,S_{\hbar,2}^{2}\ccc S_{\hbar,n+1}^{2}\psi}},\\
&II\geq -C\lr{\al_{2}N^{2\be}\hbar^{-1}\lra{\psi,S_{\hbar,1}^{2}\ccc S_{\hbar,n+1}^{2}\psi}
+\al_{2}^{-1}N^{2\be}\hbar^{-1}\lra{\psi,S_{\hbar,1}^{4}S_{\hbar,2}^{2}\ccc S_{\hbar,n+1}^{2}\psi}},\\
&III\geq -C \lr{\al_{3}N^{\be}\hbar^{-3}\lra{\psi,S_{\hbar,1}^{2}\ccc S_{\hbar,n+1}^{2}\psi}
+\al_{3}^{-1}N^{4\be}\hbar\lra{\psi,S_{\hbar,1}^{4}S_{\hbar,2}^{2}\ccc S_{\hbar,n+1}^{2}\psi}}.
\end{align*}
Optimizing the choice of $\al_{1}$, $\al_{2}$ and $\al_{3}$, we find that
\begin{align*}
&\lra{\psi,V_{12}S_{\hbar,1}^{2}S_{\hbar,2}^{2}\ccc S_{\hbar,n+1}^{2}\psi}+c.c.\\
\geq &-CN^{-3/2}N^{\frac{5}{2}\be}\hbar^{-1}\lr{N^{2}\lra{\psi,S_{\hbar,1}^{2}
\ccc S_{\hbar,n+1}^{2}\psi}+N\lra{\psi,S_{\hbar,1}^{4}S_{\hbar,2}^{2}\ccc S_{\hbar,n+1}^{2}\psi}}.
\end{align*}

As for the last term on the r.h.s of $(\ref{equ:expansion, induction n+2})$, we have
\begin{align*}
&\lra{\psi,V_{1,n+2}S_{\hbar,1}^{2}S_{\hbar,2}^{2}\ccc S_{\hbar,n+1}^{2}\psi}+c.c.\\
\geq & \lra{\psi,V_{1,n+2}(-\hbar^{2}\Delta_{x_{1}})S_{\hbar,2}^{2}\ccc S_{\hbar,n+1}^{2}\psi}+c.c.\\
\geq &\lra{\psi,|\hbar\nabla V_{1,n+2}||\hbar\nabla_{x_{1}}|S_{\hbar,2}^{2}\ccc S_{\hbar,n+1}^{2}\psi}+c.c.\\
\geq &-\al \lra{\psi,|\hbar\nabla V_{1,n+1}|
S_{\hbar,2}^{2}\ccc S_{\hbar,n+1}^{2}\psi}
-\al^{-1}\lra{|\hbar\nabla_{x_{1}}|\psi,|\hbar\nabla V_{1,n+2}|
S_{\hbar,2}^{2}\ccc S_{\hbar,n+1}^{2}|\hbar\nabla_{x_{1}}|\psi}\\
\geq &-C\lrs{\al N^{\be}\hbar^{-3}+\al^{-1}N^{2\be}\hbar^{-1}}\lra{\psi,S_{\hbar,1}^{2}\ccc S_{\hbar,n+2}^{2}\psi}\\
\geq&-CN^{\frac{3}{2}\be}\hbar^{-2}
\end{align*}
where we optimized the choice of $\al$. Then we get
\begin{align*}
&\lra{\psi,
(H_{N,\hbar}+N)S_{\hbar,1}^{2}\ccc S_{\hbar,n}^{2}(H_{N,\hbar}+N)\psi}\\
\geq&(N-n)(N-n-1)\lrs{1-\frac{CN^{\frac{5}{2}\be}\hbar^{-1}n^{2}}{N^{1/2}(N-n)}-
\frac{CN^{\frac{3}{2}\be}\hbar^{-2}n}{N}}\lra{\psi,S_{\hbar,1}^{2}\ccc S_{\hbar,n+2}^{2}\psi}\\
&+(2n+1)(N-n)\lrs{1-\frac{CN^{\frac{5}{2}\be}\hbar^{-1}n}{N^{3/2}}}\lra{\psi,S_{\hbar,1}^{4}S_{\hbar,2}^{2}\ccc S_{\hbar,n+1}^{2}\psi}
\end{align*}
Since $\be<\frac{3}{5}$, $n\leq (\ln N)^{100}$ and $\hbar^{-1}\leq \ln N$, we can find $N_{0}(\be)$ which is independent of $n$ and $\hbar$, so that
\begin{align*}
&\lra{\psi,
(H_{N,\hbar}+N)S_{\hbar,1}^{2}\ccc S_{\hbar,n}^{2}(H_{N,\hbar}+N)\psi}
\geq\frac{N^{2}}{4}\lra{\psi,S_{\hbar,1}^{2}\ccc S_{\hbar,n+2}^{2}\psi}
\end{align*}
for every $N\geq N_{0}(\be)$. Together with $(\ref{equ:energy estimate, induction n})$, this completes the proof.
\end{proof}

\bibliographystyle{abbrv}
\bibliography{references}

\end{document}